\documentclass[11pt,reqno]{amsart}
\usepackage[T1]{fontenc}
\usepackage{lmodern}
\usepackage{amsmath,amssymb,amsthm,mathtools}
\usepackage{geometry}
\usepackage{booktabs,longtable,array}
\usepackage[expansion=false]{microtype}
\usepackage[hidelinks]{hyperref}
\usepackage{url}
\urldef{\certificaterepository}\url{https://github.com/MColbrook/Kenig_spectral_radius}
\usepackage{tikz}
\usepackage{placeins}
\usepackage{xcolor}

\numberwithin{equation}{section}
\newtheorem{theorem}{Theorem}[section]
\newtheorem{proposition}[theorem]{Proposition}
\newtheorem{lemma}[theorem]{Lemma}
\newtheorem{corollary}[theorem]{Corollary}
\theoremstyle{definition}

\theoremstyle{remark}
\newtheorem{remark}[theorem]{Remark}
\newcommand{\R}{\mathbb R}
\newcommand{\C}{\mathbb C}
\newcommand{\T}{\mathbb T}

\newcommand{\Z}{\mathbb Z}
\newcommand{\dd}{\,\mathrm d}
\newcommand{\ii}{\mathrm i}
\newcommand{\pv}{\operatorname{pv}}
\newcommand{\essspec}{\sigma_{\mathrm{ess}}}
\newcommand{\Lip}{\operatorname{Lip}}
\newcommand{\HS}{\mathrm{HS}}

\newcommand{\supp}{\operatorname{supp}}
\newcommand{\osc}{\operatorname{osc}}

\newcommand{\sgn}{\operatorname{sgn}}
\newcommand{\Herm}{\operatorname{Herm}}
\newcommand{\norm}[1]{\left\lVert#1\right\rVert}
\newcommand{\abs}[1]{\left\lvert#1\right\rvert}
\newcommand{\one}{\mathbf 1}
\newcommand{\cH}{\mathcal H}
\newcommand{\cE}{\mathcal E}
\newcommand{\cB}{\mathcal B}

\allowdisplaybreaks[1]
\title[Kenig's spectral-radius conjecture]{A counterexample to Kenig's conjecture for the Laplace double-layer operator}
\author{Matthew J. Colbrook}
\address{Department of Applied Mathematics and Theoretical Physics,
University of Cambridge, Wilberforce Road, Cambridge CB3 0WA,
United Kingdom}
\email{m.colbrook@damtp.cam.ac.uk}
\author{Siavash Sadeghi}
\address{Department of Mathematics and Statistics, Mathematics Building,
University of Reading, Whiteknights campus, Reading RG6 6AX,
United Kingdom}
\email{s.sadeghi@pgr.reading.ac.uk}
\date{\today}

\begin{document}
\begin{abstract}
We disprove Kenig's 1994 spectral-radius conjecture for the double-layer operator in every dimension $d\ge2$. For every $t>1/2$ sufficiently close to $1/2$, we construct a bounded simply connected Lipschitz domain with connected boundary such that the double-layer operator on $L^2$ with respect to the surface measure has $\pm \ii t$ in its Fredholm essential spectrum. In particular, the essential spectral radius is greater than $1/2$. In the plane, the boundary may be chosen differentiable everywhere and smooth away from a single point. Our planar construction develops a refinement mechanism that produces unbounded growth of adjoint resolvent solutions with fixed forcing while keeping the graph slopes uniformly bounded. A closed recursion for the density and its image under an auxiliary singular integral reduces the growth estimate to a certified inequality for Hermitian $2\times2$ matrices. The certificate uses continuous estimates, exact rational arithmetic and rigorous interval enclosures. Localisation and rescaled graph insertions then produce a weakly null sequence of normalised approximate eigenvectors for the adjoint double-layer operator on one bounded boundary. A compact intertwining relation on Cartesian products of the planar domain with cubes gives the higher-dimensional examples without changing the spectral parameter.
\end{abstract}

\maketitle
\vspace{-\baselineskip}
\vspace{-\baselineskip}
{\small\tableofcontents}

\section{Introduction}\label{sec:introduction}

Let \(\Omega\subset\R^d\) be a bounded Lipschitz domain and let
\(\Gamma=\partial\Omega\), with outward unit normal \(n\). The double-layer operator
\begin{equation}\label{eq:closed-double-layer}
 (D_\Gamma f)(X)=\frac{1}{|S^{d-1}|}\pv\int_\Gamma
       \frac{(X-Y)\cdot n(Y)}{|X-Y|^d}f(Y)\dd\sigma(Y)
\end{equation}
and its adjoint are central to classical
potential theory, reducing the Dirichlet and Neumann problems for
Laplace's equation to integral equations for densities on the boundary
\cite{Verchota}. Layer-potential formulations also underpin boundary element
methods; see \cite{Costabel1988} for the Sobolev-space theory
and \cite[Section~1]{CWS} for second-kind formulations. The spectral theory of $D_\Gamma$ is key to the analysis of transmission problems across Lipschitz interfaces, as studied, for example, by Escauriaza and Mitrea \cite{escauriaza2004transmission}. 

Throughout, spectra are taken on complex Hilbert spaces, and
\(\essspec(T)\) denotes the Fredholm essential spectrum:
\(\lambda\in\essspec(T)\) precisely when \(T-\lambda I\) is not
Fredholm. Equivalently, this is the spectrum of the image of \(T\)
in the Calkin algebra, the quotient of the bounded operators by the
compact operators.
For a bounded operator \(T\) on a Hilbert space \(H\), its spectral radius
and essential spectral radius are denoted by \(r(T;H)\) and
\(r_{\mathrm{ess}}(T;H)\), respectively. The essential spectral radius is independent of the various standard definitions of essential spectra \cite[Corollary 4.11]{edmunds1987spectral}. If \(H\) is a closed invariant
subspace, these expressions refer to the restriction of \(T\) to
\(H\). We omit the space when it is clear from the context.

The subspace
\[
 L^2_0(\Gamma)=\left\{f\in L^2(\Gamma,\dd\sigma):
                  \int_\Gamma f\dd\sigma=0\right\}
\]
is invariant under \(D_\Gamma^*\). For connected \(\Gamma\), Kenig conjectured in 1994
\cite{Kenig} that \(r(D_\Gamma^*;L^2_0(\Gamma))<1/2\).
Its essential-spectrum formulation
\cite{CHPV} asks whether
\begin{equation}\label{eq:conjecture}
 r_{\mathrm{ess}}(D_\Gamma;L^2(\Gamma,\dd\sigma))<\frac12
\end{equation}
for every bounded Lipschitz domain. The strict gap in
\eqref{eq:conjecture} is allowed to depend on the domain.
The two formulations are equivalent
by \cite[Lemma~2.1]{CHPV}.

The significance of the strict spectral bound is its relation to
Neumann-series inversion. For connected \(\Gamma\), let
\(B=D_\Gamma^*|_{L^2_0(\Gamma)}\). Verchota's \(L^2\) layer-potential
theory gives invertibility of \(B\pm\tfrac12I\)
\cite{Verchota}; see also \cite[Section~2]{CHPV}. The stronger condition
\(r(B)<1/2\) is equivalent to operator-norm convergence of the series
\[
 (B+\tfrac12I)^{-1}=2\sum_{n=0}^{\infty}(-2B)^n,
 \qquad
 (B-\tfrac12I)^{-1}=-2\sum_{n=0}^{\infty}(2B)^n.
\]
This connection motivates the conjecture; see
\cite[Section~1.1, footnote~4]{CHPV}.
The spectral theory for the double layer operator has also been studied in Sobolev energy spaces. The connection with the variational problem of Poincar\'e and symmetrization was studied by Khavinson, Putinar and Shapiro \cite{KhavinsonPutinarShapiro2007}, as well as by Perfekt and Putinar \cite{PerfektPutinar2014} in the case of Lipschitz domains. Contraction estimates for Neumann iteration in terms of the energy norm were proved by Steinbach and Wendland \cite{SteinbachWendland2001}. The setting is different from the \(L^2\) surface measure space in Kenig's conjecture.

We construct a counterexample to Kenig's conjecture in the
plane and lift it to every higher dimension. Moreover, we show that the bound \(r_{\mathrm{ess}}\le1/2\) also fails.

\begin{theorem}\label{thm:main}
There exists \(\varepsilon>0\) such that, for every
\(t\in[1/2,1/2+\varepsilon)\) and every integer \(d\ge2\), there is a
bounded simply connected Lipschitz domain \(\Omega_{t,d}\subset\R^d\)
whose connected boundary \(\Gamma_{t,d}=\partial\Omega_{t,d}\) satisfies
\begin{equation}\label{eq:main-essential-point}
 \pm\ii t\in\essspec(D_{\Gamma_{t,d}};L^2(\Gamma_{t,d},\dd\sigma)).
\end{equation}
In particular,
\[
 r_{\mathrm{ess}}(D_{\Gamma_{t,d}};L^2(\Gamma_{t,d},\dd\sigma))\ge t,
 \qquad
 r(D_{\Gamma_{t,d}}^*;L^2_0(\Gamma_{t,d}))\ge t.
\]
\end{theorem}

For each fixed \(t\), the planar domain
\(\Omega_t=\Omega_{t,2}\) in Theorem~\ref{thm:main} may be chosen in the form
\[
 \Omega_t=\{(x,y):-1<x<1,\ -H<y<f(x)\},
 \qquad H>1+\norm f_\infty,
\]
where \(f:[-1,1]\to\R\) is Lipschitz, with
\(\Lip(f)\le27/2\) and \(f(0)=f'(0)=0\).
The function \(f\) vanishes outside disjoint intervals accumulating
only at \(0\). Its graph has a horizontal tangent at the origin,
although oscillations occur arbitrarily close to it. The other three
sides are straight. Figure~\ref{fig:domain} illustrates this geometry.

A minor modification gives a planar example whose boundary
is differentiable everywhere and \(C^\infty\) except at a
single point; see Remark~\ref{rem:planar-regularity}.
Although the double-layer operator is compact for every
\(C^1\) boundary \cite{FJR1978}, its essential spectral
radius can therefore exceed \(1/2\) when the normal exists
everywhere and is discontinuous at only one point.

In every dimension \(d\ge3\), we take (see Section~\ref{sec:lifting})
\[
 \Omega_{t,d}=\Omega_t\times(-1,1)^{d-2}.
\]

\begin{figure}[t!]
\centering
\begin{tikzpicture}[x=1cm,y=1cm,font=\footnotesize,
  line cap=round,line join=round]
\pgfmathdeclarefunction{domaincurve}{2}{%
  \pgfmathparse{min(1,max(0,2-2*abs(#1)))*
    (-0.38*#1+0.085*sin((3+2*#2)*180*#1)
      +0.012*sin((10+6*#2)*180*#1))}}
\tikzset{retained/.style={black,line width=.8pt},
  tapered/.style={black!50,line width=.65pt},
  boundary/.style={black,line width=.65pt}}
\path[use as bounding box] (0,0) rectangle (14.8,3.8);
\node[font=\small] at (3.35,3.55) {(a) Bounded domain};
\node[font=\small] at (11.1,3.55) {(b) Enlargement near $0$};

% Overall domain; horizontal separations and heights are exaggerated.
\draw[boundary,fill=black!5] (.30,.25) rectangle (6.40,2.55);
\foreach \q in {6,5,4,3,2,1} {
  \pgfmathsetmacro{\cc}{3.35+1.90/pow(2,\q-1)}
  \pgfmathsetmacro{\rr}{(\cc-3.35)*.32/(\q+1)}
  \fill[white] ({\cc-\rr},1.95) rectangle ({\cc+\rr},3.20);
  \fill[black!5] ({\cc-\rr},1.95)
    -- plot[domain=-1:1,samples=141,variable=\t]
      ({\cc+\rr*\t},{2.55+6*\rr*domaincurve(\t,\q)})
    -- ({\cc+\rr},1.95) -- cycle;
  \draw[boundary] plot[domain=-1:1,samples=141,variable=\t]
      ({\cc+\rr*\t},{2.55+6*\rr*domaincurve(\t,\q)});
}
\node[font=\Large] at (3.35,1.3) {$\Omega$};
\node at (1.15,2.85) {$\Gamma$};
\draw[black!45,densely dashed,line width=.4pt]
  (3.18,2.05) rectangle (5.72,3.17);
\fill (3.35,2.55) circle (1.1pt);
\node at (3.00,3.03) {$0$};
\draw[black!55,line width=.35pt] (3.08,2.96) -- (3.32,2.61);
\draw[black!40,densely dashed,line width=.4pt,->]
  (5.72,3.17) -- (6.35,3.17) -- (7.35,2.85);

% Enlarged graph. The relative widths decrease towards the origin.
\draw[black!25,densely dotted,line width=.35pt] (7.40,2.05) -- (14.70,2.05);
\draw[boundary] (7.40,2.05) -- (7.70,2.05);
\fill (7.70,2.05) circle (1.1pt);
\node[below,inner sep=3pt] at (7.70,2.05) {$0$};
\node[inner sep=0pt] at (7.88,2.05) {$\cdots$};
\foreach \q in {5,4,3,2,1} {
  \pgfmathsetmacro{\cc}{7.70+5.6/pow(2,\q-1)}
  \pgfmathsetmacro{\rr}{(\cc-7.70)*.32/(\q+1)}
  \ifnum\q=5
    \draw[boundary] (8.00,2.05) -- ({\cc-\rr},2.05);
  \fi
  \draw[tapered] plot[domain=-1:-.5,samples=61,variable=\t]
    ({\cc+\rr*\t},{2.05+2.4*\rr*domaincurve(\t,\q)});
  \draw[retained] plot[domain=-.5:.5,samples=141,variable=\t]
    ({\cc+\rr*\t},{2.05+2.4*\rr*domaincurve(\t,\q)});
  \draw[tapered] plot[domain=.5:1,samples=61,variable=\t]
    ({\cc+\rr*\t},{2.05+2.4*\rr*domaincurve(\t,\q)});
  \ifnum\q>1
    \pgfmathsetmacro{\cn}{7.70+5.6/pow(2,\q-2)}
    \pgfmathsetmacro{\rn}{(\cn-7.70)*.32/\q}
    \draw[boundary] ({\cc+\rr},2.05) -- ({\cn-\rn},2.05);
  \else
    \draw[boundary] ({\cc+\rr},2.05) -- (14.70,2.05);
  \fi
}
\draw[retained] (7.85,.70) -- (8.40,.70);
\node[anchor=west,inner sep=3pt] at (8.40,.70) {retained segment};
\draw[tapered] (11.65,.70) -- (12.20,.70);
\node[anchor=west,inner sep=3pt] at (12.20,.70) {tapered portions};
\end{tikzpicture}
\caption{Schematic of the planar domain \(\Omega_t\) in
Theorem~\ref{thm:main} and an
enlargement of its upper boundary. The disjoint graph insertions
accumulate at $0$, where the boundary has a horizontal tangent.
Each central segment is retained; the outer portions taper to the
flat boundary. Only finitely many insertions are shown, with
exaggerated horizontal separations, widths and heights.
See Section~\ref{sec:gluing} for the construction.}
\label{fig:domain}
\end{figure}

\subsection{Known cases}
The conjecture holds for several broad classes of domains.
For \(C^1\) boundaries, compactness of the double-layer operator
gives \(r_{\mathrm{ess}}(D_\Gamma;L^2(\Gamma,\dd\sigma))=0\)
\cite{FJR1978}. Fabes, Sand and Seo \cite{FSS1992} established the conjecture
for convex domains. Chang and Lee \cite{ChangLee2008}
extended this result to a class of locally convex Lipschitz
domains. I.~Mitrea \cite{Mitrea1999} proved the conjecture
for domains with sufficiently small Lipschitz character.
Compactness also holds for Lipschitz boundaries whose unit normal
has vanishing mean oscillation \cite{HMT2010}.

Positive results also extend to boundaries with corners.
Shelepov's spectral calculation \cite{Shelepov1991} gives
\eqref{eq:conjecture} for planar domains whose boundaries have
bounded tangent variation and no cusps, including polygons.
Elschner \cite{Elschner1992} established the
conjecture for Lipschitz polyhedra in three dimensions.
Chandler-Wilde, Hagger, Perfekt and Virtanen
\cite[Corollary~5.7]{CHPV} extended this result to Lipschitz
curvilinear polyhedra. Their localisation and Floquet analysis also
treats locally dilation-invariant boundaries, and their numerical
examples give further evidence for the conjecture on oscillatory
boundaries. These computations are supported by error estimates
with explicit constants \cite[Section~6]{CHPV}.

In the energy-space setting, Perfekt and Putinar
\cite{PerfektPutinar2017} determined the essential spectrum
for planar curvilinear polygons explicitly in terms of their
corner angles. For rotationally symmetric conical points in three dimensions,
Helsing and Perfekt \cite{HelsingPerfekt2018} described the
essential spectrum in both surface-measure \(L^2\) and the
energy space.

Ruiz \cite{Ruiz2026} studies related spectral questions on
self-similar chains of smooth domains accumulating at a point,
whose union has non-Lipschitz boundary. His block-Toeplitz
representation yields an explicit spectral calculation for
concentric annuli. On arclength \(L^2\) densities with zero integral
over the boundary of each annulus, an eigenvalue of modulus \(1/2\)
has infinite multiplicity \cite[Section~8.3]{Ruiz2026}.

Chandler-Wilde and Spence \cite{CWS} construct Lipschitz domains
for which the double-layer operator has arbitrarily large essential
norm and essential numerical radius. Such lower bounds do not,
by themselves, determine the essential spectral radius.
Their construction also inserts rescaled graph segments near an
accumulation point \cite[Figure~1 and Definition~4.7]{CWS}.

For elliptic systems, Mitrea and Tucker \cite{MitreaTucker2003}
constructed counterexamples to the corresponding spectral-radius
conjecture on polygonal domains. Their proof for \(p=2\) uses
validated interval computations. 
Theorem~\ref{thm:main} concerns the scalar Laplace operator on
surface-measure \(L^2\) in every dimension. 

\subsection{The construction}

We now describe the construction. The main analytical
ingredient is a coupled refinement law for the density and its image
under an auxiliary singular integral. This law yields a sufficient
matrix condition for growth of adjoint resolvent solutions with fixed
forcing, while a stopping rule keeps the graph slopes uniformly
bounded. Localisation and gluing turn the resulting approximate
eigenvectors into a weakly null sequence on one bounded boundary,
giving the required essential spectrum.

Fix \(t\) in the interval of
Theorem~\ref{thm:main}, and let \(z=-\ii t\). We work first with
graphs of the form
\[
 \Gamma_F=\{(x,F(x)):x\in\R\},\qquad
 F(x)=mx+h(x),\qquad h(x+1)=h(x).
\]
Such a graph is invariant under translation by \((1,m)\). We call
the graph over one horizontal period a \emph{cell}, and \(m\) its \emph{tilt}.
Let \(P_F\) be the double-layer operator on period-one densities,
in horizontal coordinates with measure \(\dd x\). Its kernel is
obtained by symmetrically summing the contributions from all translated
periods; Section~\ref{subsec:periodization} gives the precise definition.
Our first objective is a sequence of smooth profiles \(F_q\), with
\(\norm{F_q'}_\infty\le9/2\), and mean-zero unit densities \(u_q\)
on one period such that
\[
 \norm{(P_{F_q}^*-z)u_q}_2\rightarrow0.
\]
The profiles change with \(q\), with more refinements required for
a smaller residual.

The geometry uses two fixed graphs with tilts $m=\pm 1/2$. In each graph there is a straight interval $J$ having a slope $-m$. On this straight interval, small replicas of a full period of the other graph are inserted. The graph tilt inherited from the affine component provides for the tilt of the replica graph, only the periodic correction is added. Further refining is limited only to the corresponding straight components of the replicas and does not take place elsewhere, thus preventing accumulation of slopes with the depth of refining. The smooth cut-offs and sufficiently different scales ensure a uniform bound, and hence give realising graphs with slopes bounded by $9/2$. However, the obtained operator uniform bounds do not restrict the growth of the resolvent solution with fixed forcing. The two basic profiles do not depend on $t$, while the scales of refining and the domain can depend on $t$.

To analyse an insertion, we let its period tend to zero while keeping
the preceding graph fixed. The limit has separate coordinates for
the position on the preceding graph and the position within one period
of the inserted graph. Iterating gives an operator \(D^{(j)}\) on a product of
\(j+1\) periods, one for each scale. We solve
\(((D^{(j)})^*-z)u_j=\eta\), where the forcing \(\eta\) is fixed and
extended constantly in each new coordinate. The central estimate is
that \(\norm{u_j}_2\) grows with \(j\). Indeed,
\[
 \frac{\norm{((D^{(j)})^*-z)u_j}_2}{\norm{u_j}_2}
       =\frac{\norm{\eta}_2}{\norm{u_j}_2},
\]
so growth gives a small residual after normalisation.

A refinement couples the density \(u\) to an auxiliary singular
integral \(Ru\). For a graph \(F\), the operator \(R_F\) is obtained
by symmetrically periodising the kernel
\[
 -\frac{x-y}{2\pi\bigl((x-y)^2+(F(x)-F(y))^2\bigr)}.
\]
On the parts where refinement continues, keeping both components
gives the exact update
\[
 \binom{u_{\mathrm{new}}(\boldsymbol{x},v)}
       {(R_{\mathrm{new}}u_{\mathrm{new}})(\boldsymbol{x},v)}
       =M_m(v)\binom{u_{\mathrm{old}}(\boldsymbol{x})}
       {(R_{\mathrm{old}}u_{\mathrm{old}})(\boldsymbol{x})},
\]
where $\boldsymbol{x}$ collects the preceding coordinates and $v$ is the new periodic
coordinate.
The old pair is extended constantly in $v$.
The entries of $M_m$ are determined by the two local resolvents. We see that a scalar local update is insufficient since the new density depends on both the old density value and its auxiliary image at the same coarse point. This image contains nonlocal information from the preceding scales. The pair give a closed recursion through successive refinements.

The matrix estimate measures growth by a positive quadratic form.
For a Hermitian \(2\times2\) matrix \(X\), define
\[
 K_m(X)=\int_J M_m(v)^*XM_m(v)\dd v,\qquad A=K_{1/2}K_{-1/2}.
\]
These maps preserve positive semidefinite matrices; \(X>Y\) means
that \(X-Y\) is positive definite. At \(z=-\ii/2\),
the certificate proves \(A(X)>(21/20)X\) for one explicit positive
definite \(X\). Thus the integral of the quadratic form in
\((u,Ru)\) increases by a factor greater than one after two
refinements. The uniform bound for \(R\) then forces the density norm
to grow. The strict margin persists at nearby \(z=-\ii t\).
Strong convergence of the operators and their adjoints at each fixed
depth transfers the normalised densities and their small residuals
from the product space to smooth periodic graphs.

The Floquet transform next passes from one period to the full graph.
It decomposes the full-line operator into operators on one period,
with phase \(\theta\) corresponding to
\(u(x+1)=e^{\ii\theta}u(x)\). On the mean-zero densities above,
the adjoint operators at small nonzero phases approach \(P_F^*\).
Combining a short interval of such phases and then truncating gives
compactly supported unit densities on the full graphs with residuals
tending to zero; see Lemma~\ref{lem:full-line-packets}.

Finally, take a sufficiently long segment from each successive graph,
containing the support of its density. The extra length controls the
operator's action far from that support. We rescale these segments
and taper their ends to form the upper boundary of the planar domain
\(\Omega_t\) described above. The scales are chosen so that the sides
and bottom contribute a residual tending to zero.
On the bounded boundary \(\Gamma\), the transferred densities satisfy
\begin{equation}\label{eq:desired-sequence}
 \norm{\psi_q}_2=1,\qquad \psi_q\rightharpoonup0,\qquad
 \norm{(D_\Gamma^*+\ii t)\psi_q}_2\rightarrow0.
\end{equation}
The boundary $\Gamma$ and the operator $D_\Gamma^*$ are now fixed:
all the densities lie in the same Hilbert space and have vanishing
residuals for the same operator.
Their disjoint supports imply $\psi_q\rightharpoonup0$, so every
compact operator maps them to a sequence converging to zero in norm.
If $D_\Gamma^*+it$ were Fredholm, a bounded left inverse modulo compact
operators would therefore force $\|\psi_q\|_2\to0$, contradicting
$\|\psi_q\|_2=1$.
Taking adjoints and using that the kernel is real gives both signs in
\eqref{eq:main-essential-point} in the plane.
Repeating the top function periodically also gives a single
Lipschitz graph whose double-layer operator has \(\pm\ii t\) in its
essential spectrum; see Corollary~\ref{cor:periodic-graph}.

For higher dimensions, a normalised smooth cutoff of zero integral
in the additional coordinates defines an isometric embedding into
the mean-zero space on the product boundary. The embedding intertwines
the adjoint operators modulo a Hilbert--Schmidt operator. Compactness
therefore transfers the weakly null approximate eigenvectors at the
same spectral parameter; see Section~\ref{sec:lifting}.

\subsection{Outline}
Sections \ref{sec:operators}--\ref{sec:finite-depth} establish the operator and approximation lemmas. Section~\ref{sec:criterion} establishes a sufficient matrix condition for growth under refinement and proves that this growth
yields approximate eigenvectors on one period. Section~\ref{sec:local-data} uses certified polynomial
data, smoothing and stability in the spectral parameter to construct two fixed smooth graphs
satisfying the hypotheses of this criterion. Section~\ref{sec:gluing} localises the one-period densities
on the full graphs and inserts rescaled copies into a single bounded boundary, completing the planar proof. The periodic-graph corollary follows from the same construction. Section~\ref{sec:lifting} lifts the bounded-domain result to every higher dimension. Appendix \ref{cert:appendix} gives the numerical certificate in detail, including the exact input polynomials, interval enclosures and residual estimates. Figure~\ref{fig:proof-flow} summarises the main steps of the proof and their dependencies.

\begin{figure}[htbp]
\centering
\begin{tikzpicture}[
  font=\footnotesize,
  stage/.style={draw=black!65, fill=black!3, rounded corners=2pt,
    align=center, text width=5.7cm, inner sep=6pt},
  main/.style={stage, text width=10.5cm},
  flow/.style={->, line width=.6pt, draw=black!70}
]
\node[stage] (operators) at (-3.3,0) {
  \textbf{Section~\ref{sec:operators}: Graph operators}\\
  Coordinates and periodisation\\
  Operator bounds and Floquet fibres
};
\node[stage] (limit) at (-3.3,-1.65) {
  \textbf{Section~\ref{sec:two-scale}: Two-scale limit}\\
  Identify the limit of one fine refinement
};
\node[stage] (refine) at (-3.3,-3.3) {
  \textbf{Section~\ref{sec:finite-depth}: Finite refinements}\\
  Stopping keeps slopes bounded; derive the resolvent recursion
};
\node[stage] (certificate) at (3.3,0) {
  \textbf{Appendix~\ref{cert:appendix}: Local certificate}\\
  Rigorous inverse and residual bounds;\\
  growth inequality $A(X)>(21/20)X$
};
\node[stage] (profiles) at (3.3,-3.3) {
  \textbf{Section~\ref{sec:local-data}: Two fixed graphs}\\
  Smoothing and parameter stability\\
  preserve the certified growth inequality
};
\node[main] (growth) at (0,-5.05) {
  \textbf{Apply the growth criterion of Section~\ref{sec:criterion}}\\
  The graphs from Section~\ref{sec:local-data} give growing resolvent solutions\\
  and normalised densities with small residuals on one period
};
\node[main] (gluing) at (0,-7.05) {
  \textbf{Section~\ref{sec:gluing}: Localisation and the bounded domain}\\
  Floquet localisation gives compactly supported densities\\
  Rescaled insertions give a weakly null sequence on one boundary\\
  $\pm\mathrm{i}t\in\sigma_{\mathrm{ess}}$
};
\node[main] (lifting) at (0,-9.2) {
  \textbf{Section~\ref{sec:lifting}: Higher dimensions}\\
  Products with cubes and compact intertwining\\
  give bounded counterexamples in every dimension $d\ge2$
};
\draw[flow] (operators) -- (limit);
\draw[flow] (limit) -- (refine);
\draw[flow] (certificate) -- (profiles);
\draw[flow] (refine.south) -- ([xshift=-3.3cm]growth.north);
\draw[flow] (profiles.south) -- ([xshift=3.3cm]growth.north);
\draw[flow] (growth) -- (gluing);
\draw[flow] (gluing) -- node[right, fill=white, inner sep=2pt]
  {bounded domain} (lifting);
\end{tikzpicture}
\caption{Main steps of the proof. Sections~\ref{sec:operators}--\ref{sec:finite-depth}
provide the refinement argument, while Appendix~\ref{cert:appendix} and
Section~\ref{sec:local-data} supply the hypotheses of the growth criterion
in Section~\ref{sec:criterion}. The central box represents the application of this criterion.
The arrows show the main dependencies. Section~\ref{sec:gluing} combines localisation and the bounded-domain
construction; its periodic-graph corollary follows from the same planar
construction.}
\label{fig:proof-flow}
\end{figure}

\section{Graph operators and periodisation}\label{sec:operators}

Horizontal coordinates give a common, unweighted space on which to
compare graph operators; we also discuss the conversion to arclength.

\subsection{Horizontal coordinates and normalisation}
For a Lipschitz function \(F:\R\to\R\), set
\[
 \Gamma_F=\{(x,F(x)):x\in\R\},\qquad
 s_F=F',\qquad a_F=(1+s_F^2)^{1/2}.
\]
We orient \(\Gamma_F\) upwards, so it bounds the lower graph domain.
For \(r=x-y\) and \(\Delta=F(x)-F(y)\), define the kernels
\begin{align}
 p_F(x,y)&=\frac{\Delta-rs_F(y)}{2\pi(r^2+\Delta^2)},\label{eq:graph-p}\\
 r_F(x,y)&=-\frac{r}{2\pi(r^2+\Delta^2)},&
 t_F(x,y)&=-\frac{\Delta}{2\pi(r^2+\Delta^2)}.\label{eq:graph-rt}
\end{align}
Their principal-value operators on \(L^2(\R,\dd x)\) are denoted
by \(\mathsf P_F,\mathsf R_F,\mathsf T_F\).
The auxiliary singular integral \(\mathsf R_F\) enters the
coupling created when a new oscillating profile is added. The operator
\(\mathsf T_F\) will allow us to recover the adjoint from the
same limits. We reserve \(P_F,R_F,T_F\), without the sans-serif
font, for period-one operators.

The Cauchy singular-integral theorem \cite{CMM}, in its quantitative
graph form, implies that for each \(L<\infty\) there are finite
constants \(C_P(L),C_R(L),C_T(L)\) such that
\begin{equation}\label{eq:uniform-line-bounds}
 \norm{\mathsf P_F}\le C_P(L),\quad
 \norm{\mathsf R_F}\le C_R(L),\quad
 \norm{\mathsf T_F}\le C_T(L)
 \qquad\text{if }\norm{F'}_\infty\le L.
\end{equation}
To obtain these bounds, remove the bounded invertible tangent factor
at the integration point from the graph Cauchy operator. Its kernel becomes
\((r+\ii\Delta)^{-1}\). If \(A_F\) is this operator and
\(Jf=\overline f\), then \((A_F+JA_FJ)/2\) and
\((A_F-JA_FJ)/(2\ii)\) are bounded complex-linear operators
with kernels \(r/(r^2+\Delta^2)\) and
\(-\Delta/(r^2+\Delta^2)\), respectively. They give the bounds
for \eqref{eq:graph-rt}; bounded multiplication by \(s_F\)
then gives the bound for \eqref{eq:graph-p}. The change between \(\dd x\) and
arclength has uniformly bounded condition number when the slope
is bounded. The same bounds follow from \cite[Consequence~3.2]{AKM}.

We take every adjoint of a horizontal operator in the unweighted
space \(L^2(\dd x)\). The antisymmetry of the two kernels in
\eqref{eq:graph-rt} gives
\begin{equation}\label{eq:joint-identities}
 \mathsf R_F^*=-\mathsf R_F,\quad
 \mathsf T_F^*=-\mathsf T_F,\quad
 \mathsf P_F=\mathsf R_F M_{s_F}-\mathsf T_F,\quad
 \mathsf P_F^*=-M_{s_F}\mathsf R_F+\mathsf T_F.
\end{equation}
Here \(M_g\) denotes multiplication by \(g\).
For smooth graphs, horizontal and Euclidean-distance principal
values agree. Indeed, the two Euclidean exclusion endpoints are
locally \(\varepsilon/a_F(x)+O(\varepsilon^2)\), and the resulting
change in the leading odd \(1/r\) integral tends to zero.
For general Lipschitz graphs we use the bounded principal-value
realisation of these operators.

To pass between these coordinates and the boundary, let
\(\rho\) be a physical density on \(\Gamma_F\), and define
\(\rho^\flat(x)=\rho(x,F(x))\). Cancelling the Jacobian at the integration point
against the unit normal in \eqref{eq:closed-double-layer} shows
that the direct operator is \(\mathsf P_F\) in these coordinates.
Its physical adjoint therefore satisfies
\begin{equation}\label{eq:physical-adjoint}
 (D_{\Gamma_F}^*(u/a_F))^\flat=a_F^{-1}\mathsf P_F^*u,\qquad
 \norm{u/a_F}_{L^2(\Gamma_F)}^2=\int_\R\frac{|u(x)|^2}{a_F(x)}\dd x.
\end{equation}
Here \(u/a_F\) denotes the density whose value at \((x,F(x))\)
is \(u(x)/a_F(x)\).
Thus \eqref{eq:physical-adjoint} transfers both an adjoint residual
and its norm to the boundary, with norm comparisons depending
only on \(L\).

We also fix the signs of the jump relations. With
\(E(X)=-(2\pi)^{-1}\log|X|\), the kernel in
\eqref{eq:closed-double-layer} is \(\partial_{n(Y)}E(X-Y)\).
The divergence theorem gives the double-layer potential of \(1\)
equal to \(-1\) inside a bounded domain and zero outside. The
\(L^2\) trace relations on planar Lipschitz boundaries follow from
the Sokhotski--Plemelj formulas for the Cauchy integral
\cite[Theorems~2.3 and~4.2]{Lenells2018}; the principal-value trace
is the average of the interior and exterior traces. For a bounded
closed boundary \(\Gamma\), this gives \(D_\Gamma1=-1/2\).
The interior and exterior double-layer traces
are \(D_\Gamma-1/2\) and \(D_\Gamma+1/2\).
The corresponding single-layer normal derivatives, in the same
outward direction, are \(D_\Gamma^*+1/2\) and \(D_\Gamma^*-1/2\).
On a bounded closed boundary, the identity \(D_\Gamma1=-1/2\)
implies that \(D_\Gamma^*\) preserves the mean-zero subspace.

\subsection{Periodisation and Floquet decomposition}\label{subsec:periodization}
Let \(F(x)=mx+h(x)\), where \(m\in\R\),
\(h\in C^2(\T;\R)\), and
\(\T=\R/\Z\) has Lebesgue measure of total mass one.
The graph is invariant under translation by \((1,m)\), and \(m\)
is its tilt. In particular,
\[
 F(x+\ell)=F(x)+m\ell,\qquad
 k(x+\ell,y+\ell)=k(x,y)\quad(\ell\in\Z)
\]
for each kernel \(k=p_F,r_F,t_F\).

We will use three operators associated with the same kernel.
The full-line operator \(\mathsf K_F\) acts on \(L^2(\R,\dd x)\).
Its symmetric periodisation \(K_F\) acts on \(L^2(\T,\dd x)\);
it sums the contributions of translated periods symmetrically.
A Floquet fibre instead weights those contributions by a phase.
Its limits at zero phase need not equal the symmetric operator.

For \(k=p_F,r_F,t_F\), define the symmetric periodisation by
\begin{equation}\label{eq:symmetric-periodization}
 k^{\mathrm{per}}(x,y)=\lim_{N\to\infty}
             \sum_{j=-N}^{N}k(x,y+j),\qquad 0\le x,y<1,
\end{equation}
away from the lifted diagonals \(x-y\in\Z\). Their singularities
are interpreted as principal values. More precisely, on smooth
periodic inputs the associated action is
\[
 (K_Fg)(x)=\lim_{\varepsilon\downarrow0}
  \int_{\substack{0<y<1\\ \operatorname{dist}(x-y,\Z)>\varepsilon}}
          k^{\mathrm{per}}(x,y)g(y)\dd y.
\]
The exclusion is unnecessary for the removable double-layer
singularity. On general \(L^2\) inputs, integral notation denotes
the bounded extension established below.

To compute the remote terms, set
\[
 d_j=x-y-j,\qquad H=h(x)-h(y),\qquad b(y)=h'(y).
\]
The signed height difference is \(m d_j+H\) and the double-layer numerator is
\(H-d_jb(y)\). Hence, with \(\Phi(s)=(1+s^2)^{-1}\),
\begin{align*}
 p_F(x,y+j)
  &=\frac1{2\pi}\left(-\frac{b(y)}{d_j}+\frac{H}{d_j^2}\right)
                        \Phi\!\left(m+\frac{H}{d_j}\right),\\
 r_F(x,y+j)
  &=-\frac1{2\pi d_j}\,
       \Phi\!\left(m+\frac{H}{d_j}\right),\\
 t_F(x,y+j)
  &=-\frac1{2\pi d_j}
       \left(m+\frac{H}{d_j}\right)
       \Phi\!\left(m+\frac{H}{d_j}\right).
\end{align*}
Here \(H\) and \(b\) are bounded on \([0,1]^2\). The identities
\(d_j^{-1}=-j^{-1}+O(j^{-2})\) and
\(\Phi(m+H/d_j)=\Phi(m)+O(|j|^{-1})\) therefore give,
uniformly for \(x,y\in[0,1]\),
\begin{equation}\label{eq:periodic-leading-tails}
 p_F(x,y+j)=\frac{s_F(y)-m}{2\pi(1+m^2)j}+O(j^{-2}),\quad
 r_F(x,y+j)=\frac1{2\pi(1+m^2)j}+O(j^{-2}),
\end{equation}
and
\[
 t_F(x,y+j)=\frac{m}{2\pi(1+m^2)j}+O(j^{-2}).
\]
The remainders are bounded by \(C_Fj^{-2}\) for \(|j|\ge2\),
so their sum converges uniformly. Moreover, integer reindexing changes partial sums only by
finitely many \(O(N^{-1})\) end terms. Thus
\eqref{eq:symmetric-periodization} is periodic in both variables.

We next examine one lifted diagonal \(x-y=j\). Put \(a=y+j\)
and \(d=x-a\). Taylor's formula with integral remainder gives
\[
 Q(x,a)=\int_0^1F'(a+td)\dd t,\qquad
 B(x,a)=\int_0^1(1-t)F''(a+td)\dd t.
\]
Since \(F(x)-F(a)=dQ\) and
\(F(x)-F(a)-dF'(a)=d^2B\), the kernels become
\[
 p_F(x,a)=\frac{B(x,a)\Phi(Q(x,a))}{2\pi},\qquad
 r_F(x,a)=-\frac{\Phi(Q(x,a))}{2\pi d}.
\]
The first expression extends continuously across the diagonal.
For the second, we evaluate the Hilbert coefficient at \(x\),
as needed later. The identity
\[
 \frac{Q(x,a)-s_F(x)}d=-\int_0^1vF''(a+vd)\dd v
\]
shows that this divided difference also extends continuously.
Factoring the difference of \(\Phi\) gives
\begin{align*}
 r_F(x,a)+\frac{\Phi(s_F(x))}{2\pi d}
 &=-\frac1{2\pi}\frac{Q(x,a)-s_F(x)}d
       \int_0^1\Phi'\bigl(s_F(x)+t(Q(x,a)-s_F(x))\bigr)\dd t.
\end{align*}
Both factors are continuous, with a bound depending on
\(\norm{F''}_\infty\). Thus subtracting this pole leaves
a continuous remainder.

The matching periodic singular kernel is the cotangent:
\[
 (H_{\T}f)(x)=\pv\int_0^1\cot\bigl(\pi(x-y)\bigr)f(y)\dd y,
 \qquad
 \cot(\pi r)=\frac1\pi\lim_{N\to\infty}
                 \sum_{j=-N}^N\frac1{r-j}.
\]
Its pole at \(x-y=j\) is \(1/[\pi(x-y-j)]\). The other
images are continuous there and their paired tails converge
uniformly. Consequently
\[
 q_F(x,y)=r_F^{\mathrm{per}}(x,y)
                 +\frac12\Phi(s_F(x))\cot\bigl(\pi(x-y)\bigr)
\]
extends continuously across the interior diagonal and both seam
corners, carried by \(j=0,1,-1\). Integer reindexing matches
the seam values, so \(q_F\) and \(p_F^{\mathrm{per}}\) are
continuous kernels on \(\T^2\).

For the Fourier normalisation, the conjugate Poisson kernels
\[
 C_\rho(t)=2\sum_{n=1}^{\infty}\rho^n\sin(2\pi nt)
       \qquad(0<\rho<1)
\]
have multipliers \(-\ii\sgn(k)\rho^{|k|}\) and tend to
\(\cot(\pi t)\) off the integers. Subtracting the value of a smooth input at \(x\)
justifies the principal-value limit, using
\(|C_\rho(t)|\le C/\operatorname{dist}(t,\Z)\) near \(\rho=1\).
Thus, with \(e_k(x)=e^{2\pi\ii kx}\) and \(\sgn(0)=0\),
\[
 H_{\T}e_k=-\ii\sgn(k)e_k,\qquad H_{\T}1=0,\qquad
 \norm{H_{\T}}=1.
\]
These kernels give boundedness of \(P_F,R_F\), and
\(T_F=R_FM_{s_F}-P_F\) gives boundedness of \(T_F\).
These preliminary estimates use the \(C^2\) profile; the
slope-only bounds will instead come from the full-line operators.

For the periodic operators, let \(U:\C\to L^2(\T)\) be constant
extension and let \(W:L^2(\T)\to\C\) be integration over one period:
\[
 (U\xi)(x)=\xi,\qquad Wg=\int_0^1g(y)\dd y.
\]
The Floquet phase condition is \(u(x+1)=e^{\ii\theta}u(x)\).
Multiplication by \(e^{-\ii\theta x}\) makes the density periodic;
the operator in these coordinates is called a \emph{fibre operator}.

\begin{lemma}\label{lem:floquet}
Let \(F(x)=mx+h(x)\) with \(m\in\R\) and
\(h\in C^2(\T;\R)\). The
full-line graph operators admit a unitary Floquet decomposition.
Denote their fibres in periodic coordinates by
\(P^\theta,R^\theta,T^\theta\). These are norm continuous on
each of \((-\pi,0)\) and \((0,\pi)\), and have one-sided norm
limits
\begin{align}
 P^{0,\sigma}&=P_F+
       \frac{\ii\sigma}{2(1+m^2)}UW_b,\label{eq:floquet-p}\\
 R^{0,\sigma}&=R_F+
       \frac{\ii\sigma}{2(1+m^2)}UW,\label{eq:floquet-r}\\
 (P^{0,\sigma})^*&=P_F^*-
       \frac{\ii\sigma}{2(1+m^2)}M_bUW,
       \qquad \sigma\in\{-1,1\},\label{eq:floquet-adjoint}
\end{align}
where \(b=h'\), \(W_bf=\int_\T bf\), and \(\sigma\) is the
sign of the approaching phase. The corresponding limit for the
third operator is
\[
 T^{0,\sigma}=T_F+\frac{\ii\sigma m}{2(1+m^2)}UW.
\]
\end{lemma}

\begin{proof}
Let \(f_\ell(x)=f(x+\ell)\), \(0\le x<1\). The transform
\[
 (\cB f)(\theta,x)=\sum_{\ell\in\Z}e^{-\ii\ell\theta}f_\ell(x),
 \qquad -\pi<\theta<\pi,
\]
is unitary from \(L^2(\R)\) to
\(L^2((-\pi,\pi),\dd\theta/(2\pi);L^2(0,1))\) by Parseval
and Fourier inversion in the cell index. For any of the three kernels
\(k\), define 
\[
(A_jg)(x)=\pv\int_0^1k(x,y+j)g(y)\dd y, \qquad 0<x<1.
\]
Thus \(A_j\) is the operator block of \(\mathsf K_F\) describing
the contribution from the \(j\)-th source cell \((j,j+1)\) to the
target cell \((0,1)\), so \(\norm{A_j}\le\norm{\mathsf K_F}\),
including the two seam blocks. Principal values exclude
\(|x-y-j|\le\varepsilon\).

To justify transforming term by term, first use the Abel-regularised
action \((\mathsf K_\rho f)_\ell=\sum_j\rho^{|j|}A_jf_{\ell+j}\),
\(0<\rho<1\). The remote bound \(\norm{A_j}=O(|j|^{-1})\)
makes this sum absolutely convergent in operator norm. For finite-cell
inputs, translation covariance and \(j=n-\ell\) give
\(k(x+\ell,y+n)=k(x,y+j)\) and
\(e^{-\ii\ell\theta}=e^{\ii j\theta}e^{-\ii n\theta}\), hence
\begin{align*}
 (\cB\mathsf K_\rho f)(\theta,x)
 &=\pv\int_0^1\sum_{\ell,n}\rho^{|n-\ell|}e^{-\ii\ell\theta}
       k(x+\ell,y+n)f_n(y)\dd y\\
 &=\pv\int_0^1\left(\sum_j\rho^{|j|}e^{\ii j\theta}k(x,y+j)\right)
       (\cB f)(\theta,y)\dd y.
\end{align*}
Thus the fibres are \(\widehat K_\rho^\theta
=\sum_j\rho^{|j|}e^{\ii j\theta}A_j\). With
\(M_\theta g(x)=e^{\ii\theta x}g(x)\), the unitary gauge
\(\widetilde{\cB}=M_{-\theta}\cB\) gives
\(M_{-\theta}\widehat K_\rho^\theta M_\theta\), whose kernel is
\(\sum_j\rho^{|j|}e^{\ii\theta(y+j-x)}k(x,y+j)\).
The cell phase and gauge phase suggest the undamped kernel
\begin{equation}\label{eq:gauged-kernel}
 \sum_{j\in\Z}e^{\ii\theta(y+j-x)}k(x,y+j).
\end{equation}

By \eqref{eq:periodic-leading-tails}, the remote expansion is
\(k(x,y+j)=c_k(y)/j+O(j^{-2})\), \(|j|\ge2\), where
\[
 c_p(y)=\frac{b(y)}{2\pi(1+m^2)},\qquad
 c_r(y)=\frac1{2\pi(1+m^2)},\qquad
 c_t(y)=\frac{m}{2\pi(1+m^2)}.
\]
Set \((C_kg)(x)=\int_0^1c_k(y)g(y)\dd y\) and
\(A_j=C_k/j+B_j\) for \(|j|\ge2\); Schur's test gives
\(\norm{B_j}\le C_F/j^2\). Writing
\(S_\rho(\theta)=\sum_{j\ne0}\rho^{|j|}e^{\ii j\theta}/j\), we have exactly
\begin{align*}
 \widehat K_\rho^\theta
 &=A_0+\rho e^{\ii\theta}A_1+\rho e^{-\ii\theta}A_{-1}
   +C_k\bigl(S_\rho(\theta)-2\ii\rho\sin\theta\bigr)+\sum_{|j|\ge2}\rho^{|j|}e^{\ii j\theta}B_j.
\end{align*}
The correction removes the scalar \(j=\pm1\) terms while retaining
their actual blocks. The logarithm branch continuous from \(1\) gives
\(S_\rho=-\log(1-\rho e^{\ii\theta})+\log(1-\rho e^{-\ii\theta})\),
so \(|S_\rho|\le\pi\). Its radial limit, together with Dirichlet's
test and Abel summation, yields
\begin{equation}\label{eq:phase-tail}
 S(\theta):=\sum_{j\ne0}\frac{e^{\ii j\theta}}j
       =\ii(\pi\sgn\theta-\theta),\qquad 0<|\theta|<2\pi.
\end{equation}
The symmetric value at zero is zero, whereas the one-sided limits
are \(\pm\ii\pi\).

The \(B_j\)-series is norm continuous by summable domination.
The finite blocks are bounded, and \(M_\theta\) is norm continuous,
so their gauged contributions are norm continuous, including at both
seam corners. The scalar bound gives a common bound in \(\rho,\theta\)
for this fixed graph. Fibre-norm convergence as \(\rho\uparrow1\)
and dominated convergence give a global strong limit of the
direct-integral multipliers. For finite-cell inputs, each fixed output
cell has the limit \(\sum_{n:f_n\ne0}A_{n-\ell}f_n
=(\mathsf K_Ff)_\ell\). Conjugating back by \(\cB^{-1}\), the global strong limit agrees with
\(\mathsf K_F\) on every cell; density and boundedness identify it on
\(L^2(\R)\). This proves the unitary direct-integral decomposition
with periodic fibres \(K^\theta=M_{-\theta}\widehat K^\theta M_\theta\)
and kernel \eqref{eq:gauged-kernel}, using symmetric summation at zero.

Only \(S\) jumps. The continuous part at zero is
\(A_0+A_1+A_{-1}+\sum_{|j|\ge2}B_j=K_F\), since the symmetric
\(C_k/j\) sum vanishes. As \(M_{\pm\theta}\to I\) in norm,
\[
 (K^{0,\sigma}g)(x)=(K_Fg)(x)
          +\ii\pi\sigma\int_0^1c_k(y)g(y)\dd y.
\]
The same decomposition gives norm continuity on each side of zero.
Substituting the three coefficients proves \eqref{eq:floquet-p},
\eqref{eq:floquet-r}, and the stated limit for \(T^\theta\).
Since \(b\) is real, taking the unweighted adjoint gives
\[
 (UW_b)^*=M_bUW,\qquad
 (P^{0,\sigma})^*=P_F^*-\frac{\ii\sigma}{2(1+m^2)}M_bUW.
\]
Conjugation changes the imaginary scalar's sign, and adjunction
preserves norm convergence. The adjoint correction vanishes when
\(Wg=0\); the direct correction contains \(W_bg=\int_\T bg\)
and need not vanish on a mean-zero density.
\end{proof}

\begin{lemma}\label{lem:periodic-basic}
Let \(F(x)=mx+h(x)\), where \(m\in\R\) and
\(h\in C^2(\T;\R)\). The symmetric periodic operators
\(P_F,R_F,T_F\) are bounded and satisfy
\begin{gather}
 R_F^*=-R_F,\qquad T_F^*=-T_F,\qquad
 P_F=R_FM_{s_F}-T_F,\qquad P_F^*=-M_{s_F}R_F+T_F,\label{eq:periodic-identities}\\
 P_F1=0,\qquad \int_\T R_F1=\int_\T T_F1=0. \label{eq:periodic-constants}
\end{gather}
Their norms have bounds depending only on \(\norm{F'}_\infty\).
Moreover \(P_F\) is compact, and
\begin{equation}\label{eq:hilbert-principal}
 R_F=-\frac12M_{\Phi(s_F)}H_{\T}+\mathcal R_F,\qquad
 \Phi(s)=\frac1{1+s^2},\qquad
 H_{\T}e_k=-\ii\sgn(k)e_k,
\end{equation}
where \(e_k(x)=e^{2\pi\ii kx}\) and \(\mathcal R_F\) is compact.
\end{lemma}

\begin{proof}
Set \(L=\norm{F'}_\infty\), and let \(K\) denote any of
\(P,R,T\). The unitary decomposition in Lemma~\ref{lem:floquet}
and the full-line estimate give
\[
 \operatorname*{ess\,sup}_{-\pi<\theta<\pi}\norm{K^\theta}
       =\norm{\mathsf K_F}\le C_K(L).
\]
The full-measure set of phases is dense, so norm continuity
extends this bound to every nonzero phase and then to both
one-sided limits. Their opposite tail corrections cancel:
\[
 K_F=\frac12\bigl(K^{0,+}+K^{0,-}\bigr),\qquad
 \norm{K_F}\le C_K(L).
\]
Thus the periodic bounds inherit the full-line constants, with
no dependence on higher derivatives.

For \(k=r_F,t_F\), antisymmetry and translation covariance give
\[
 k(y,x+j)=-k(x+j,y)=-k(x,y-j).
\]
Symmetric image sets are invariant under \(j\mapsto-j\), and
the local exclusion \(\operatorname{dist}(x-y,\Z)>\varepsilon\)
is invariant under transposition. Hence the truncated bilinear
forms are skew-adjoint. They converge on smooth inputs by the
Hilbert decomposition for \(R_F\), and by
\(t_F^{\mathrm{per}}=r_F^{\mathrm{per}}s_F-p_F^{\mathrm{per}}\)
for \(T_F\), applied to the \(C^1\) input \(s_Fg\).
Boundedness and density give \(R_F^*=-R_F\), \(T_F^*=-T_F\).
The same kernel identity, with the same cutoffs, gives
\(P_F=R_FM_{s_F}-T_F\); taking adjoints places the real
multiplier on the left, as in \eqref{eq:periodic-identities}.

To evaluate the double layer on a constant, fix \(x\) and set
\[
 A_x(y)=\arctan\frac{F(x)-F(y)}{x-y},\qquad
 A_x'(y)=2\pi p_F(x,y)\quad(y\ne x).
\]
Both one-sided limits of \(A_x(y)\) at \(y=x\) equal
\(\arctan F'(x)\), so the intermediate endpoint terms at
\(x\pm\varepsilon\) cancel in the limit. For \(N\ge2\),
the retained intervals form \([-N,N+1)\), giving
\begin{align*}
 \sum_{j=-N}^N\int_0^1p_F(x,y+j)\dd y
 &=\int_{-N}^{N+1}p_F(x,y)\dd y=\frac{A_x(N+1)-A_x(-N)}{2\pi}.
\end{align*}
At either integer endpoint \(v=-N,N+1\), periodicity of \(h\)
gives
\[
 \frac{F(x)-F(v)}{x-v}
        =m+\frac{h(x)-h(0)}{x-v}.
\]
Both endpoint angles therefore tend uniformly to \(\arctan m\)
for \(x\in[0,1]\). Uniform convergence of the paired remote
remainders identifies the limit with the periodic action:
\[
 (P_F1)(x)
  =\lim_{N\to\infty}
       \frac{A_x(N+1)-A_x(-N)}{2\pi}=0.
\]

Reality makes \(\langle R_F1,1\rangle\) real, while
skew-adjointness makes it purely imaginary; hence it is zero,
and likewise for \(T_F\).

Finally set \((\mathcal R_Fg)(x)=\int_0^1q_F(x,y)g(y)\dd y\),
with \(q_F\) as above. Its definition gives
\eqref{eq:hilbert-principal}. Both \(q_F\) and
\(p_F^{\mathrm{per}}\) are continuous on the compact torus
square, so their operators are Hilbert--Schmidt and compact.
This step uses \(C^2\) regularity, unlike the slope-only bounds.
\end{proof}

\subsection{Continuity of a fixed family}

The construction will vary the periodic correction and tilt of a local graph
while earlier scales are held fixed. We need the resulting local operators to
vary continuously in norm.

\begin{lemma}\label{lem:smooth-continuity}
Suppose \(m_n,m\in\R\), \(m_n\to m\),
\(h_n,h\in C^2(\T;\R)\), \(h_n\to h\) in \(C^2(\T)\), and
\(F_n(x)=m_nx+h_n(x)\), \(F(x)=mx+h(x)\).
Then \(P_{F_n}\to P_F\) and \(R_{F_n}\to R_F\) in
operator norm on \(L^2(\T)\).
The same conclusions hold for \(T_{F_n}\) and all three adjoints.
In particular, a family continuous in the tilt and the \(C^2\)
profile gives a norm-continuous, uniformly bounded operator family
on each compact parameter set.
\end{lemma}

\begin{proof}
We separate the translates \(j=0,\pm1\), which contain all possible singularities on \([0,1]^2\), from the remote terms \(|j|\ge2\).
For the local terms, put \(a=y+j\), \(r=x-a\), \(Q=Q(x,a)\), \(B=B(x,a)\), and \(s=F'(x)\). The kernel identities give
\[
p_F(x,a)=\frac{B\Phi(Q)}{2\pi},
\qquad
r_F(x,a)+\frac{\Phi(s)}{2\pi r}
=-\frac{1}{2\pi}\frac{\Phi(Q)-\Phi(s)}r.
\]
The integral formulas for \(B\) and \(Q\) show that the double-layer expression converges uniformly under the stated assumptions. For the auxiliary remainder, let
\[
\frac{\Phi(Q)-\Phi(s)}r
=
\frac{Q-s}{r}\int_0^1\Phi'\!\bigl(s+t(Q-s)\bigr)\,dt,
\qquad
\frac{Q-s}{r}=-\int_0^1vF''(a+vr)\,dv.
\]
Thus this remainder also converges uniformly, including at \(r=0\), while the pole coefficients \(\Phi(F_n')\) converge uniformly to \(\Phi(F')\). This handles both the interior diagonal and the seam corners.
For the remote translates, let, for \(k=p,r\),
\[
k_{F_n}(x,y+j)=\frac{c_{k,n}(y)}j+e^k_{n,j}(x,y),
\qquad |j|\ge2,
\]
where
\[
c_{p,n}(y)=\frac{h_n'(y)}{2\pi(1+m_n^2)},
\qquad
c_{r,n}=\frac{1}{2\pi(1+m_n^2)}.
\]
Denote the corresponding remainders for \(F\) by \(e^k_j\). The explicit kernel formulas give
\[
\|e^k_{n,j}-e^k_j\|_\infty
\le
\frac{C}{j^2}
\left(|m_n-m|+\|h_n-h\|_{C^1(\mathbb T)}\right),
\]
with \(C\) independent of \(n,j\). For each graph, the \(1/j\) terms cancel exactly in symmetric pairs. The remote sums therefore consist only of the remainders, and the displayed estimate, summed over \(|j|\ge2\), proves their uniform convergence.
Combining the local and remote estimates gives uniform convergence of \(p_{F_n}^{\mathrm{per}}\) and of the full auxiliary remainder kernels
\[
q_{F_n}(x,y)
=
r_{F_n}^{\mathrm{per}}(x,y)
+\frac12\Phi(F_n'(x))\cot\!\bigl(\pi(x-y)\bigr).
\]
For the latter, passing from the three local poles to the cotangent kernel adds only a bounded kernel multiplied by the uniformly convergent coefficient \(\Phi(F_n')\).
Schur’s test now gives \(P_{F_n}\to P_F\) in operator norm. Likewise, using the Hilbert-transform decomposition and \(\|H_{\mathbb T}\|=1\),
\[
\|R_{F_n}-R_F\|
\le
\frac12\|\Phi(F_n')-\Phi(F')\|_\infty
+\|q_{F_n}-q_F\|_\infty
\rightarrow0.
\]
Finally, \(T_F=R_FM_{F'}-P_F\) gives convergence for \(T_{F_n}\), and taking adjoints preserves operator-norm convergence. Continuity also gives uniform boundedness on compact parameter sets.
\end{proof}

\section{A joint two-scale limit}\label{sec:two-scale}

We add an oscillating perturbation to the original graph and follow
the double layer and auxiliary singular integral together. Periodic
unfolding \cite{CDG2008} places them on a fixed product space, using
separate coordinates as in two-scale convergence
\cite{Allaire1992, Nguetseng1989}.

Write \(x\) for the coarse-scale coordinate and \(v\) for the
fine-scale periodic coordinate, and let
\(\cH_1=L^2(\T_x\times\T_v,\dd x\dd v)\), with each factor
of measure one. Constant extension and averaging in \(v\) are
the bounded maps
\[
 U:L^2(\T)\rightarrow\cH_1,\qquad
 W=U^*:\cH_1\rightarrow L^2(\T).
\]
Their actions, and the complementary projection on \(\cH_1\), are
\[
 (Uf)(x,v)=f(x),\qquad (Wg)(x)=\int_\T g(x,v)\dd v,\qquad Q=I-UW.
\]
Thus \(Qg\) has zero integral in \(v\), while the local operators
below act on all period-one densities in \(v\), including constants.

For an integer \(N\ge1\), define
\(E_N:L^2(\T)\to\cH_1\) by
\begin{equation}\label{eq:unfolding}
 (E_Nf)(x,v)=f\bigl((\lfloor Nx\rfloor+v)/N\bigr).
\end{equation}
A change of variables on each interval of length \(1/N\) shows that
\(E_N\) is an isometry.
Writing \(I_{N,k}=[k/N,(k+1)/N)\), its adjoint and range projection are
\[
 (E_N^*g)(t)=N\int_{I_{N,k}}g(y,Nt-k)\dd y,
 \qquad k=\lfloor Nt\rfloor,
\]
\[
 (\Pi_Ng)(x,v):=(E_NE_N^*g)(x,v)
   =N\int_{I_{N,\lfloor Nx\rfloor}}g(y,v)\dd y.
\]
The projection averages only in \(x\); it is not the identity
at finite \(N\). Uniform continuity on continuous inputs,
followed by density, gives
\begin{equation}\label{eq:unfolding-projections}
 E_N^*E_N=I,\qquad E_NE_N^*\rightarrow I
 \quad\hbox{strongly}.
\end{equation}

\begin{proposition}\label{prop:two-scale}
Let \(m\in\R\), \(F_0(x)=mx+h_0(x)\),
\(h_0\in C^\infty(\T;\R)\), and
\(H_0\in C^\infty(\T^2;\R)\). Set
\[
 F_N(x)=F_0(x)+N^{-1}H_0(x,Nx),\quad
 s(x)=F_0'(x),\quad b(x,v)=\partial_vH_0(x,v).
\]
For each \(x\), let \(P_x,R_x,T_x\) be the full symmetric
periodic operators of the local graph
\[
 G_x(t)=s(x)t+H_0(x,t),\qquad t\in\R.
\]
Their direct-integral operators on \(\cH_1\) act as follows:
\begin{align*}
 (\mathcal P g)(x,v)
 &=\bigl(P_x[g(x,\cdot)]\bigr)(v)
   =\int_0^1p_{G_x}^{\mathrm{per}}(v,w)g(x,w)\dd w,\\
 (\mathcal R g)(x,v)
 &=\bigl(R_x[g(x,\cdot)]\bigr)(v)
   =\pv\int_0^1r_{G_x}^{\mathrm{per}}(v,w)g(x,w)\dd w.
\end{align*}
 Define \(\mathcal T\) by the same construction with \(t_{G_x}\).
These are bounded operators: the local slopes \(s(x)+b(x,w)\)
have one common bound, and Lemma~\ref{lem:smooth-continuity}
makes the operator fields norm continuous in \(x\).

Define also \(V:\cH_1\to L^2(\T)\) by \(V=W M_b\). Then
\begin{align}
 E_NR_{F_N}E_N^*&\rightarrow
       R^{(1)}:=\mathcal R+UR_{F_0}W,\label{eq:two-scale-r}\\
 E_NP_{F_N}E_N^*&\rightarrow
       D^{(1)}:=\mathcal P+UP_{F_0}W+UR_{F_0}V,\label{eq:two-scale-p}\\
 E_NM_{F_N'}E_N^*&\rightarrow M_{s+b},\label{eq:two-scale-slope}\\
 E_NP_{F_N}^*E_N^*&\rightarrow (D^{(1)})^*
       \label{eq:two-scale-adjoint}
\end{align}
strongly as \(N\to\infty\) through the positive integers.
In particular they hold along every sequence of integers tending
to infinity. The adjoint limit is
\[
 (D^{(1)})^*=\mathcal P^*+UP_{F_0}^*W-M_bUR_{F_0}W.
\]
\end{proposition}

\begin{proof}
Since
\[
 F_N'(x)=F_0'(x)+\partial_vH_0(x,Nx)
                       +N^{-1}\partial_xH_0(x,Nx),
\]
the slopes are uniformly bounded. Lemma~\ref{lem:periodic-basic}
therefore bounds all graph operators and multipliers in the
statement, independently of the growing higher derivatives.
Fix \(\psi\in C^\infty(\T^2)\), and put
\(\psi_N(y)=\psi(y,Ny)\). The formula for \(E_N^*\) gives
\[
 \norm{E_N^*\psi-\psi_N}_2
 \le N^{-1}\norm{\partial_x\psi}_\infty.
\]
Thus replacing \(E_N^*\psi\) by \(\psi_N\) in either
unfolded action costs \(O_\psi(N^{-1})\). We pass to general \(\psi\in \cH_1\) by density in the end.

Set \(\eta=N^{-1}<\rho<1/2\) and fix an integer \(M\ge2\).
For \(x,y\in[0,1)\), let
\[
 I^{\rm near}_{x,j}=\{y:|x-y-j|<\rho\},\qquad
 I^{\rm far}_{x,j}=\{y:|x-y-j|\ge\rho\}.
\]
For \(K=R,P\), with \(k_N=r_{F_N},p_{F_N}\), respectively,
decompose the actual action as
\[
 (K_{F_N}\psi_N)(x)
 =(K_{N,\rho}^{\rm near}\psi_N)(x)
  +(K_{N,\rho,M}^{\rm far}\psi_N)(x)+\tau^K_{N,M}(x),
\]
where
\begin{align*}
 (K_{N,\rho}^{\rm near}\psi_N)(x)
 &=\pv\sum_{j\in\Z}\int_{I^{\rm near}_{x,j}}
                      k_N(x,y+j)\psi(y,Ny)\dd y,\\
 (K_{N,\rho,M}^{\rm far}\psi_N)(x)
 &=\sum_{|j|\le M}\int_{I^{\rm far}_{x,j}}
                      k_N(x,y+j)\psi(y,Ny)\dd y,\\
 \tau^K_{N,M}(x)
 &=\sum_{j>M}\int_0^1[k_N(x,y+j)+k_N(x,y-j)]\psi_N(y)\dd y.
\end{align*}
The near principal value excludes \(|x-y-j|\le\delta\)
before \(\delta\downarrow0\); for \(P\) its singularity is
removable. Only \(j=-1,0,1\) meet this region. Using the graph
coordinate \(y+j\) joins the near intervals into
\((x-\rho,x+\rho)\), including at the seam. Integer \(N\)
preserves the periodic correction \(H_0\) and the sampled density
under this integer shift.

All action errors below are in the uniform norm in \(x\).
Constants \(C,C_\psi\) depend on the fixed smooth profiles and,
for \(C_\psi\), on the input; they are uniform in \(N\),
\(x\) and the phase \(u\) introduced below. Cutoff dependence is shown
explicitly. These are estimates on a fixed smooth input.

\emph{The auxiliary near action.}
With \(u=Nx\) and the lifted integration coordinate \(y=x-r\),
the near integral to be compared with the operator on \(G_x\) is
\[
 (R_{N,\rho}^{\rm near}\psi_N)(x)
 =-\frac1{2\pi}\pv\int_{|r|<\rho}
   \frac{r\,\psi(x-r,u-r/\eta)}
        {r^2+[F_N(x)-F_N(x-r)]^2}\dd r.
\]
Replace \(F_0(x)-F_0(x-r)\) by \(s(x)r\), and
\(H_0(x-r,u-r/\eta)\) by \(H_0(x,u-r/\eta)\), retaining
the dependence on \(u-r/\eta\). Let
\begin{align*}
 \Delta_N(x,r)&=F_N(x)-F_N(x-r),\\
 \Delta_f(x,r,u)&=s(x)r+
             \eta[H_0(x,u)-H_0(x,u-r/\eta)].
\end{align*}
The dependence of \(\Delta_f\) on \(\eta\) is suppressed.
Put \(v=u-r/\eta\). These are real lifted phases, so
\(u-v=r/\eta\); only profile arguments are reduced modulo one.
Both signed height differences are \(O(|r|)\), and their difference is
\begin{equation}\label{eq:chord-error}
 E:=\Delta_N-\Delta_f
 =F_0(x)-F_0(y)-s(x)r+
             \eta[H_0(x,v)-H_0(y,v)],\qquad y=x-r.
\end{equation}
Taylor's formula gives \(|E|\le C(r^2+\eta|r|)\). On
\(\eta\le|r|<\rho\), this is \(O(r^2)\), whereas
\[
 K(r,\Delta)=-\frac{r}{2\pi(r^2+\Delta^2)},\qquad
 \partial_\Delta K=
       \frac{r\Delta}{\pi(r^2+\Delta^2)^2}=O(|r|^{-2}).
\]
Hence the kernel difference is bounded on this region, and its
contribution to the near action is \(O_\psi(\rho)\).

For \(|r|<\eta\), we must retain the small odd principal part.
Taylor expansion in the second argument and then in the first
gives
\[
 \begin{gathered}
 p=s(x)+b(x,u),\qquad d=\eta\partial_xH_0(x,u),\\
 q:=\Delta_f/r=p+O(|r|/\eta),\qquad
 E/r=d+e,\qquad |e|\le C|r|.
 \end{gathered}
\]
Here replacing \(\partial_xH_0(x,v)\) by
\(\partial_xH_0(x,u)\) costs \(O(|r|/\eta)\), which is
multiplied by \(\eta\). Bounded first and second derivatives of
\(\Phi\) yield the two-increment estimate
\[
 \left|\Phi(q+d+e)-\Phi(q)-\Phi(p+d)+\Phi(p)\right|
 \le C\bigl(|e|+|d|\,|q-p|\bigr)\le C|r|.
\]
Since \(K(r,\Delta)=-\Phi(\Delta/r)/(2\pi r)\), it follows that
\[
 K(r,\Delta_N)-K(r,\Delta_f)=\frac{a_N(x)}r+O(1),\qquad
 a_N(x)=-\frac{\Phi(p+d)-\Phi(p)}{2\pi}=O(\eta).
\]
The coefficient is independent of \(r\). Its constant-input
principal value vanishes, so
\begin{align*}
 a_N(x)\pv\int_{|r|<\eta}\frac{\psi_N(x-r)}r\dd r
 &=a_N(x)\int_{|r|<\eta}
            \frac{\psi_N(x-r)-\psi_N(x)}r\dd r,\\
 \left|a_N(x)\int_{|r|<\eta}
            \frac{\psi_N(x-r)-\psi_N(x)}r\dd r\right|
 &\le C\eta^2\norm{\psi_N'}_\infty\le C_\psi\eta.
\end{align*}
Together with the bounded remainder, this costs \(O_\psi(\eta)\).
The cancellation is essential; an absolute \(\eta/|r|\) bound
cannot be integrated through zero.

We replace \(\psi(x-r,u-r/\eta)\) by
\(\psi(x,u-r/\eta)\); its second argument still varies with \(r\).
The change is bounded by \(\norm{\partial_x\psi}_\infty|r|\),
while the comparison kernel is \(O(|r|^{-1})\), so this
contributes \(O_\psi(\rho)\).
Returning to the original near action, we have obtained
\[
 (R_{N,\rho}^{\rm near}\psi_N)(x)
 =-\frac1{2\pi}\pv\int_{|r|<\rho}
      \frac{r\,\psi(x,u-r/\eta)}{r^2+\Delta_f(x,r,u)^2}\dd r
       +O_\psi(\rho+\eta),\qquad u=Nx.
\]
To identify the resulting integral, set \(r=\eta w\),
\(A=\rho/\eta\), and
\[
 D_x(u,w)=G_x(u)-G_x(u-w)
        =s(x)w+H_0(x,u)-H_0(x,u-w).
\]
The height difference is \(\eta D_x\); numerator and measure each
contribute one factor of \(\eta\), cancelling the denominator's
two factors. The resulting integral is exactly
\[
 \mathcal I_R^A(x,u)
 =-\frac1{2\pi}\pv\int_{|w|<A}
       \frac{w\,\psi(x,u-w)}{w^2+D_x(u,w)^2}\dd w.
\]
This is the truncated graph-kernel action for \(G_x\), evaluated
at \(u\), with lifted integration coordinate \(u-w\) and the
period-one density \(\psi(x,\cdot)\). It remains to restore its tails and compare
the truncation with symmetric periodisation.

The tail integrand is
\[
 r_{G_x}(u,u-w)\psi(x,u-w)
 =-\frac{\Phi(s(x))\psi(x,u-w)}{2\pi w}+O_\psi(w^{-2}).
\]
We first estimate the omitted tails. For a smooth function \(f(x,t)\), periodic in \(t\), let
\[
f(x,t)=\bar f(x)+f_0(x,t),
\qquad
\bar f(x)=\int_0^1 f(x,t)\,dt.
\]
Since \(f_0(x,\cdot)\) has mean zero, it has a periodic primitive \(B_x\), which can be chosen uniformly bounded in \(x\). Integration by parts gives
\[
\int_A^\infty \frac{f_0(x,u-w)}w\,dw
=
\frac{B_x(u-A)}A
-\int_A^\infty\frac{B_x(u-w)}{w^2}\,dw,
\]
and hence
\[
\left|\int_A^\infty \frac{f_0(x,u-w)}w\,dw\right|
\le \frac{2\|B_x\|_\infty}{A}.
\]
The negative tail satisfies the same estimate, uniformly in \(x,u\). The mean part instead cancels between the two tails:
\[
\int_{A<|w|<T}\frac{\bar f(x)}w\,dw=0
\qquad(T>A).
\]
Applying this with \(f=\psi\), and integrating the \(O_\psi(w^{-2})\) remainder absolutely, shows that restoring the tails of \(\mathcal{I}_R^A\) by symmetric truncation changes its value by \(O_\psi(A^{-1})\).
We next check that this limit agrees with the symmetric periodisation defining \(R_x\). Summing over source cells indexed by \(j=-J,\ldots,J\) integrates the source coordinate \(t=u-w\) over \([-J,J+1)\). In the \(w\)-coordinate, this is
\[
u-J-1<w<u+J.
\]
For \(u\in[0,1)\), this interval differs from \((-J,J)\) only by two end intervals of total length one. The integrand there is \(O_\psi(J^{-1})\), so the difference between the two truncated integrals tends uniformly to zero. They therefore have the same limit, and
\[
\mathcal I_R^A(x,u)
=
\bigl(R_x[\psi(x,\cdot)]\bigr)(u)
+O_\psi(A^{-1})
=
(\mathcal{R}\psi)(x,u)+O_\psi(A^{-1}).
\]
Consequently
\[
 (R_{N,\rho}^{\rm near}\psi_N)(x)
 =(\mathcal R\psi)(x,Nx)+O_\psi(\rho+\eta+\eta/\rho).
\]
For fixed \(\rho\), the rescaled interval expands to \(\mathbb R\), and the tail estimate recovers \(R_x[\psi(x,\cdot)]\) without any mean-zero assumption on \(\psi(x,\cdot)\).

\emph{The double-layer near action.}
For the double-layer comparison, also replace the slope
\(F_N'(y)\) at \(y=x-r\) by \(s(x)+b(x,v)\).
Using the same lifted variables and height differences, set
\[
 \begin{aligned}
 n_N&=\Delta_N-rF_N'(y),&
 n_f&=\Delta_f-r(s(x)+b(x,v)),\\
 d_N&=r^2+\Delta_N^2,& d_f&=r^2+\Delta_f^2.
 \end{aligned}
\]
The actual near action is
\[
 (P_{N,\rho}^{\rm near}\psi_N)(x)
 =\frac1{2\pi}\int_{|r|<\rho}
           \frac{n_N}{d_N}\psi(x-r,u-r/\eta)\dd r.
\]
Since \(F_N'(y)=F_0'(y)+b(y,v)+\eta\partial_xH_0(y,v)\),
the numerator difference is
\begin{align*}
 n_N-n_f
 &=F_0(x)-F_0(y)-rF_0'(y)+\eta[H_0(x,v)-H_0(y,v)-r\partial_xH_0(y,v)]\\
 &\quad-r[b(y,v)-b(x,v)].
\end{align*}
The first-order Taylor terms of \(F_0\) and of \(H_0\) in
its first argument cancel, giving \(|n_N-n_f|\le Cr^2\).
Taylor's formula and the Lipschitz bound in the second argument give
\[
 n_f=\eta[H_0(x,u)-H_0(x,v)-(u-v)b(x,v)],\qquad
 |n_f|\le C\min(|r|,r^2/\eta).
\]
Also \(d_N-d_f=(\Delta_N+\Delta_f)E\), so
\[
 |d_N-d_f|\le
 \begin{cases}
 C(\eta r^2+|r|^3),&|r|<\eta,\\
 C|r|^3,&\eta\le|r|<\rho.
 \end{cases}
\]
Both denominators are at least \(r^2\). Hence
\[
 \left|\frac{n_N}{d_N}-\frac{n_f}{d_f}\right|
 \le\frac{|n_N-n_f|}{r^2}
       +\frac{|n_f|\,|d_N-d_f|}{r^4}\le C
\]
in both regions: the second term is bounded by
\(C(1+|r|/\eta)\) for \(|r|<\eta\), and by \(C\) otherwise.
The kernel replacement costs \(O_\psi(\rho)\). Replacing the
density's first argument by \(x\) costs the same, using
\(|n_f/d_f|\le C/|r|\). We obtain
\[
 (P_{N,\rho}^{\rm near}\psi_N)(x)
 =\frac1{2\pi}\int_{|r|<\rho}\frac{n_f}{d_f}
               \psi(x,u-r/\eta)\dd r+O_\psi(\rho).
\]
The exact change of variables \(r=\eta w\) turns this integral into
\[
 \mathcal I_P^A(x,u)=\frac1{2\pi}\int_{|w|<A}
 \frac{[D_x(u,w)-w(s(x)+b(x,u-w))]\psi(x,u-w)}
      {w^2+D_x(u,w)^2}\dd w.
\]
Its numerator is exactly the double-layer numerator for \(G_x\).
Its tail integrand is
\[
 p_{G_x}(u,u-w)\psi(x,u-w)
 =-\frac{\Phi(s(x))b(x,u-w)\psi(x,u-w)}{2\pi w}
       +O_\psi(w^{-2}).
\]
Apply the preceding primitive argument to the entire multiplier
\(b(x,\cdot)\psi(x,\cdot)\). It gives
\(\mathcal I_P^A(x,u)=(\mathcal P\psi)(x,u)+O_\psi(A^{-1})\),
with the same symmetric periodisation convention. Thus
\[
 (P_{N,\rho}^{\rm near}\psi_N)(x)
 =(\mathcal P\psi)(x,Nx)+O_\psi(\rho+\eta/\rho).
\]

\emph{The far actions and recombination.}
For fixed \(\rho,M\), put \(r_j=x-y-j\) and
\(\Delta_{0,j}=F_0(x)-F_0(y+j)\). The height difference
\(F_N(x)-F_N(y+j)\) differs from \(\Delta_{0,j}\) by
\(O(\eta)\) on the retained far intervals. Since the slope at
\(y+j\) is \(F_N'(y)\), the numerator of \(p_{F_N}(x,y+j)\) is
\[
 \Delta_{0,j}-r_jF_0'(y)-r_jb(y,Ny)
           -\eta r_j\partial_xH_0(y,Ny)+O(\eta).
\]
The retained denominators are bounded away from zero. Therefore
\begin{align*}
 r_{F_N}(x,y+j)&=r_{F_0}(x,y+j)+O_{\rho,M}(\eta),\\
 p_{F_N}(x,y+j)&=p_{F_0}(x,y+j)
       +r_{F_0}(x,y+j)b(y,Ny)+O_{\rho,M}(\eta).
\end{align*}
The coupling has a plus sign because
\(r_{F_0}(x,y+j)=-r_j/[2\pi(r_j^2+\Delta_{0,j}^2)]\).
The term containing \(\eta\partial_xH_0\) is absorbed in the
displayed error.

To average the remaining dependence on \(Ny\), let \(a(x,y,t)\)
be a smooth amplitude on a retained interval, periodic in \(t\),
and choose \(\partial_t\mathcal A=a-\bar a\) with
\(\bar a(x,y)=\int_0^1a(x,y,t)\dd t\). We also let \(\bar\psi=W\psi\) and \(\overline{b\psi}=W(b\psi)=V\psi\), where both averages are taken in the fine variable. On each component
\([\alpha,\beta]\) of \(I^{\rm far}_{x,j}\), the chain rule gives
\[
 \int_\alpha^\beta[a(x,y,Ny)-\bar a(x,y)]\dd y
 =\eta[\mathcal A(x,y,Ny)]_\alpha^\beta
       -\eta\int_\alpha^\beta\partial_y\mathcal A(x,y,Ny)\dd y.
\]
The endpoint terms and the integral are \(O_{\rho,M}(\eta)\),
uniformly in \(x\). There are at most two components per image.
Applying this once to the amplitudes
\[
 r_{F_0}(x,y+j)\psi(y,t),\qquad
 [p_{F_0}(x,y+j)+r_{F_0}(x,y+j)b(y,t)]\psi(y,t)
\]
transforms the actual far actions into
\begin{align*}
 (R_{N,\rho,M}^{\rm far}\psi_N)(x)
 &=\sum_{|j|\le M}\int_{I^{\rm far}_{x,j}}
           r_{F_0}(x,y+j)\bar\psi(y)\dd y
              +O_{\psi,\rho,M}(\eta),\\
 (P_{N,\rho,M}^{\rm far}\psi_N)(x)
 &=\sum_{|j|\le M}\int_{I^{\rm far}_{x,j}}
     [p_{F_0}(x,y+j)\bar\psi(y)
       +r_{F_0}(x,y+j)\overline{b\psi}(y)]\dd y\\
 &\quad+O_{\psi,\rho,M}(\eta).
\end{align*}
The slope multiplies the density before averaging. Thus its
contribution is \(W(b\psi)\), which need not vanish when \(Wb=0\).

These are still truncated actions of the operators on \(F_0\).
To restore the physical remote images, the uniform bounds on the
periodic height corrections and slopes give
\begin{align*}
 r_{F_N}(x,y+j)&=\frac1{2\pi(1+m^2)j}+O(j^{-2}),\quad
 p_{F_N}(x,y+j)&=\frac{F_N'(y)-m}{2\pi(1+m^2)j}+O(j^{-2}),
\end{align*}
with remainders uniform in \(N\). The \(1/j\) coefficients,
including \(\psi_N(y)\), agree for opposite images, so pairing
gives \(\norm{\tau^K_{N,M}}_\infty\le C_\psi/M\).
The same bound restores remote images in each averaged action on
\(F_0\). This is the physical image tail; the separate tail of the
rescaled integrals has already contributed \(O_\psi(\eta/\rho)\).

Restoring the missing strip \(|x-y-j|<\rho\) in the actions on
\(F_0\) costs \(O_\psi(\rho)\).
For \(P_{F_0}\) the kernel is bounded there. For \(R_{F_0}\),
its kernel is \(-\Phi(s(x))/(2\pi r)+O(1)\); for either
\(g=\bar\psi\) or \(g=\overline{b\psi}\),
\[
 \pv\int_{|r|<\rho}\frac{g(x-r)}r\dd r
 =\int_{|r|<\rho}\frac{g(x-r)-g(x)}r\dd r=O_\psi(\rho).
\]
Recombining with the completed near actions now gives
\begin{align*}
 (R_{F_N}\psi_N)(x)
 &=(\mathcal R\psi)(x,Nx)+(R_{F_0}\bar\psi)(x)+e_N^R(x),\\
 (P_{F_N}\psi_N)(x)
 &=(\mathcal P\psi)(x,Nx)+(P_{F_0}\bar\psi)(x)
       +(R_{F_0}\overline{b\psi})(x)+e_N^P(x),
\end{align*}
where
\[
 \norm{e_N^R}_\infty+\norm{e_N^P}_\infty
 \le C_\psi\rho+C_\psi\eta+C_\psi\eta/\rho
       +C_\psi/M+C_{\psi,\rho,M}\eta.
\]
Here \(\rho\) accounts for the near comparisons and the missing
strips, \(\eta\) for the inner principal value, \(\eta/\rho\)
for the tails of the rescaled integrals, \(M^{-1}\) for remote
images, and the last term for
fixed-region approximation and averaging. We first let
\(N\to\infty\) with \(\rho,M\) fixed, then \(M\to\infty\),
and finally \(\rho\downarrow0\). The cutoff-dependent constant
does not permit an unrestricted simultaneous limit.

\emph{Unfolding and the adjoint.}
The fixed product-space outputs \(g_R=R^{(1)}\psi\) and
\(g_P=D^{(1)}\psi\) just obtained are smooth. Indeed, the local
divided differences are smooth in \(v,w\) and the parameter \(x\);
derivatives of the tail-subtracted remote kernels remain
\(O(j^{-2})\). For \(R\), first subtract the Hilbert term with its
coefficient evaluated at the output coordinate \(v\), and use that
the Hilbert multiplier preserves rapid Fourier decay. The same
argument applies to the outputs of the operators on \(F_0\).
For either \(g\), with \(g_N(t)=g(t,Nt)\) and
\(t_N(x,v)=(\lfloor Nx\rfloor+v)/N\),
\[
 (E_Ng_N)(x,v)=g(t_N(x,v),v),\qquad
 \norm{E_Ng_N-g}_\infty\le N^{-1}\norm{\partial_xg}_\infty.
\]
Unfolding the two complete action identities, using the error bound
and the initial input comparison, proves
\eqref{eq:two-scale-r}--\eqref{eq:two-scale-p} on smooth inputs.
Density and the common operator bounds extend both strong limits
to every fixed input in \(\cH_1\).

For the slope multiplier the exact identity retains the projection:
\[
 E_NM_{F_N'}E_N^*=M_{\widehat s_N}\Pi_N,\qquad
 \widehat s_N(x,v)=s(t_N)+b(t_N,v)
                         +\eta\partial_xH_0(t_N,v).
\]
Here \(\widehat s_N\to s+b\) uniformly and \(\Pi_N\to I\)
strongly, proving \eqref{eq:two-scale-slope}.

For the adjoints we use the exact graph identities, since strong
convergence alone does not pass to adjoints. Let
\(\widehat R_N=E_NR_{F_N}E_N^*\),
\(\widehat P_N=E_NP_{F_N}E_N^*\) and
\(\widehat S_N=E_NM_{F_N'}E_N^*\).
The isometry identity \(E_N^*E_N=I\) and the rule for uniformly
bounded strong products give
\[
 \widehat T_N:=E_NT_{F_N}E_N^*
 =\widehat R_N\widehat S_N-\widehat P_N
 \rightarrow R^{(1)}M_{s+b}-D^{(1)}=:T^{(1)}.
\]
Using \(WM_{s+b}=M_sW+V\) displays the coupling cancellation:
\begin{align*}
 T^{(1)}
 &=\mathcal RM_{s+b}+UR_{F_0}(M_sW+V)
                    -\mathcal P-UP_{F_0}W-UR_{F_0}V\\
 &=\mathcal T+UT_{F_0}W.
\end{align*}
The slope of \(G_x\) is \(s+b\), and that of \(F_0\) is \(s\).
Both \(R^{(1)}\) and \(T^{(1)}\) are skew-adjoint, so
\[
 E_NP_{F_N}^*E_N^*
 =-\widehat S_N\widehat R_N+\widehat T_N
 \rightarrow -M_{s+b}R^{(1)}+T^{(1)}=(D^{(1)})^*.
\]
Finally \(M_{s+b}U=UM_s+M_bU\) gives
\[
 (D^{(1)})^*=\mathcal P^*+UP_{F_0}^*W-M_bUR_{F_0}W,
\]
with the stated multiplication order and the factor \(b(x,v)\)
evaluated at the output coordinate.
\end{proof}

\section{Finite refinements with a stopping rule}\label{sec:finite-depth}

We iterate the two-scale limit at finitely many separated scales,
then impose a stopping rule that keeps the slope bound independent
of the depth.

\subsection{Realisation at separated scales}
Realisation will transfer a prescribed product-space density and its
residual to one smooth graph.
Fix \(m\in\R\) and an integer \(k\ge0\).
Let \(\cH_j=L^2(\T^{j+1})\), with product Lebesgue measure of
total mass one. Let \(h_j(x_0,\ldots,x_j)\), \(0\le j\le k\), be real
smooth functions, periodic in every variable. Define
\[
 b_j=\partial_{x_j}h_j,\qquad
 S_j(x_0,\ldots,x_j)=m+\sum_{i=0}^j b_i(x_0,\ldots,x_i),
 \qquad S_{-1}=m.
\]
When functions on different levels occur in one product-space
identity, those with fewer coordinates are extended constantly in
the remaining coordinates.
For \(j\ge1\), the operators \(U_j:\cH_{j-1}\to\cH_j\) and \(W_j=U_j^*\)
are constant extension and integration in \(x_j\). Writing
\(\mathbf x=(x_0,\ldots,x_{j-1})\) and \(t=x_j\), their actions are
\[
 (U_jf)(\mathbf x,t)=f(\mathbf x),\qquad
 (W_jg)(\mathbf x)=\int_\T g(\mathbf x,t)\dd t.
\]
In particular \(W_jU_j=I_{\cH_{j-1}}\). Set
\(V_j=W_jM_{b_j}\) and \(Q_j=I-U_jW_j\). For each fixed choice of the preceding
coordinates \(\mathbf x\), let \(P_{j,\mathbf x},R_{j,\mathbf x},T_{j,\mathbf x}\) be
the periodic operators for
\[
 v\longmapsto S_{j-1}(\mathbf x)v+h_j(\mathbf x,v).
\]
Their direct integrals on \(\cH_j\) are denoted by \(P_j,R_j,T_j\);
for example,
\begin{align*}
 (P_jg)(\mathbf x,t)
       =\bigl(P_{j,\mathbf x}[g(\mathbf x,\cdot)]\bigr)(t).
\end{align*}
Thus the local operators hold \(\mathbf x\) fixed and act only
in the newest coordinate. The level-zero
operators are those of \(mx_0+h_0(x_0)\).

\begin{proposition}\label{prop:finite-depth}
Suppose \(\abs{S_j}\le L\) for \(0\le j\le k\). Define
\begin{align}
 R^{(0)}&=R_0,\qquad D^{(0)}=P_0,\notag\\
 R^{(j)}&=R_j+U_jR^{(j-1)}W_j,\label{eq:finite-r}\\
 D^{(j)}&=P_j+U_jD^{(j-1)}W_j
                   +U_jR^{(j-1)}V_j. \label{eq:finite-d}
\end{align}
These are ordered strong limits of unfolded operators of smooth
graphs with period-one corrections and Lipschitz constant at most \(L+1\).
Their adjoints are the corresponding strong limits of the unfolded
adjoints. In particular
\begin{equation}\label{eq:depth-uniform-bounds}
 \norm{R^{(j)}}\le C_R(L+1),\qquad
 \norm{D^{(j)}}\le C_P(L+1),
\end{equation}
where the constants do not depend on \(k\), the profiles, or their
derivatives. For any fixed \(u\in\cH_k\), \(z\in\C\), and
\(\epsilon>0\), a realising graph \(F\) and a vector \(u_F\in L^2(\T)\)
can be chosen so that
\[
 \bigl|\norm{u_F}_2-\norm{u}_{\cH_k}\bigr|<\epsilon,\qquad
 \norm{(P_F^*-z)u_F}_2
       \le \norm{((D^{(k)})^*-z)u}_{\cH_k}+\epsilon.
\]
The choice below preserves the integral of \(u\); in particular,
\(u_F\) has mean zero whenever \(u\) does.
\end{proposition}

\begin{proof}
For \(k=0\), take \(F(x)=mx+h_0(x)\),
\(\cE=I\), and \(u_F=u\). Assume henceforth that \(k\ge1\).
We first construct the graphs and control their slopes. For
integers \(N_1,\ldots,N_k\), let
\(\epsilon_0=1\), \(\epsilon_j=(N_1\cdots N_j)^{-1}\), and set
\begin{equation}\label{eq:actual-finite-graph}
 F_{\boldsymbol N}(x)=mx+
  \sum_{j=0}^k\epsilon_j
      h_j(x,x/\epsilon_1,\ldots,x/\epsilon_j).
\end{equation}
The leading part of the derivative is \(S_k\), evaluated on the
diagonal variables. Differentiation in the earlier coordinates
produces the remaining terms,
\[
 \sum_{j=1}^k\sum_{i=0}^{j-1}
      \frac{\epsilon_j}{\epsilon_i}
      (\partial_{x_i}h_j)(x,x/\epsilon_1,\ldots,x/\epsilon_j).
\]
Choose \(B\ge1+\sum_{i<j}\norm{\partial_{x_i}h_j}_\infty\).
If every \(N_i\ge B\), then \(\epsilon_j/\epsilon_i\le B^{-1}\)
for \(i<j\), so the remainder is at most one. Thus all these
graphs have Lipschitz constant at most \(L+1\).

We apply the two-scale proposition first at the finest inserted scale,
with all earlier integers fixed, and then successively at the
remaining scales. The unfolding is the composition
of \(E_{N_1\cdots N_k}\) on the original variable,
\(E_{N_1\cdots N_{k-1}}\) on the remaining first coordinate, and so on,
ending with \(E_{N_1}\); at each step it acts as the identity on
variables already unfolded. Denote this isometry by
\(\cE_{\boldsymbol N}:L^2(\T)\to\cH_k\).
Writing \(M_0=1\) and \(M_j=N_1\cdots N_j=\epsilon_j^{-1}\),
the composition is
\[
 (\cE_{\boldsymbol N}f)(x_0,\ldots,x_k)
 =f\left(\sum_{j=1}^k\frac{\lfloor N_jx_{j-1}\rfloor}{M_j}
                         +\frac{x_k}{M_k}\right).
\]
Its range projection averages each \(x_{j-1}\) over its
\(N_j\)-partition interval, independently for \(1\le j\le k\),
and leaves \(x_k\) unchanged. Also \(\cE_{\boldsymbol N}1=1\).
The range projections converge strongly to the identity in the
order used for all operator limits:
\begin{equation}\label{eq:ordered-limits}
 \lim_{N_1\to\infty}\lim_{N_2\to\infty}\cdots
                     \lim_{N_k\to\infty},
\end{equation}
with the rightmost limit taken first.

We prove the ordered strong limits for the auxiliary operator,
the double-layer operator, and the slope multiplier by induction
on the depth. The case \(k=1\) is
Proposition~\ref{prop:two-scale}. Assume these limits at depth
\(k-1\), and put \(\boldsymbol N'=(N_1,\ldots,N_{k-1})\).
Let \(F_{\boldsymbol N'}\) be the graph obtained by stopping
the defining sum at depth \(k-1\). Then
\[
 F_{\boldsymbol N}(t)
 =F_{\boldsymbol N'}(t)
   +M_k^{-1}H_{\boldsymbol N'}(t,M_kt),\qquad
 H_{\boldsymbol N'}(t,v)
 =h_k(t,M_1t,\ldots,M_{k-1}t,v).
\]
For fixed \(\boldsymbol N'\), apply
Proposition~\ref{prop:two-scale} as \(N_k\to\infty\), at
frequency \(M_k\) along the multiples of \(M_{k-1}\).
This separates the coordinate \(x_k=v\). The local operators
act in \(x_k\) on the graph
\[
 v\longmapsto F_{\boldsymbol N'}'(t)v
                  +H_{\boldsymbol N'}(t,v),
\]
whose tilt is
\[
 F_{\boldsymbol N'}'(t)
 =S_{k-1}(t,M_1t,\ldots,M_{k-1}t)
  +\sum_{j=1}^{k-1}\sum_{i=0}^{j-1}
     \frac{M_i}{M_j}
     (\partial_{x_i}h_j)(t,M_1t,\ldots,M_jt).
\]
The cross derivatives in this formula must be retained while
the preceding integers are fixed.

We now unfold the preceding coordinate while retaining \(x_k\).
Set \(\mathcal K=L^2(\mathbb T_{x_k})\) and
\[
 \overline E_{\boldsymbol N'}
   =\mathcal E_{\boldsymbol N'}\otimes I_{\mathcal K}
   :L^2(\mathbb T_t;\mathcal K)\rightarrow\mathcal H_k,
 \qquad
 \Pi_{\boldsymbol N'}
   =\overline E_{\boldsymbol N'}\overline E_{\boldsymbol N'}^*.
\]
For either local direct-integral operator
\(\mathcal A_{\boldsymbol N'}\), with field
\(A_{\boldsymbol N'}(t)\) on \(\mathcal K\), unfolding gives
the exact identity
\[
 \overline E_{\boldsymbol N'}\mathcal A_{\boldsymbol N'}
       \overline E_{\boldsymbol N'}^*
 =M_{\widehat A_{\boldsymbol N'}}\Pi_{\boldsymbol N'},
 \qquad
 \widehat A_{\boldsymbol N'}(\mathbf x)
 =A_{\boldsymbol N'}(\tau_{\boldsymbol N'}(\mathbf x)),
\]
where \(\mathbf x=(x_0,\ldots,x_{k-1})\) and
\[
 \tau_{\boldsymbol N'}(\mathbf x)
 =\sum_{j=1}^{k-1}\frac{\lfloor N_jx_{j-1}\rfloor}{M_j}
    +\frac{x_{k-1}}{M_{k-1}}.
\]
Here \(M_{\widehat A_{\boldsymbol N'}}\) acts pointwise in
\(\mathbf x\) and applies the sampled operator in \(x_k\). The sampling formula satisfies
\[
 \operatorname{dist}_{\mathbb T}
   (M_i\tau_{\boldsymbol N'}(\mathbf x),x_i)
 \le N_{i+1}^{-1}\quad(0\le i<k-1),\qquad
 M_{k-1}\tau_{\boldsymbol N'}(\mathbf x)
 =x_{k-1}\pmod1.
\]
Taking the remaining limits from \(N_{k-1}\) down to \(N_1\)
therefore recovers the profile parameters \(\mathbf x\)
uniformly. In the tilt formula, the cross derivatives with
index \(j\) vanish at the \(N_j\)-limit; terms involving only
earlier scales remain until their own limiting stage.
The sampled tilt consequently converges uniformly to
\(S_{k-1}(\mathbf x)\), and the sampled profile converges
uniformly in \(C^2(\mathbb T_{x_k})\) to
\(h_k(\mathbf x,\cdot)\).
Lemma~\ref{lem:smooth-continuity}, applied on the resulting
compact parameter family, gives uniform operator-norm
convergence of the sampled fields to those defining \(P_k\)
and \(R_k\). Since \(\Pi_{\boldsymbol N'}\to I\) strongly
in the same order, their actions converge strongly.
Write \(\mathcal R_{\boldsymbol N'}\) and
\(\mathcal P_{\boldsymbol N'}\) for the local direct-integral
operators above, and set
\(b_{\boldsymbol N'}(t,v)=\partial_vH_{\boldsymbol N'}(t,v)\).
On \(\mathcal H_{k-1}\), define
\[
\begin{aligned}
 \widehat R_{\boldsymbol N'}=\mathcal E_{\boldsymbol N'}R_{F_{\boldsymbol N'}}
       \mathcal E_{\boldsymbol N'}^*,\qquad
 \widehat P_{\boldsymbol N'}=\mathcal E_{\boldsymbol N'}P_{F_{\boldsymbol N'}}
       \mathcal E_{\boldsymbol N'}^*,\qquad
 \widehat S_{\boldsymbol N'}=\mathcal E_{\boldsymbol N'}M_{F_{\boldsymbol N'}'}
       \mathcal E_{\boldsymbol N'}^*,
\end{aligned}
\]
and on \(\mathcal H_k\), set
\[
 \widehat B_{\boldsymbol N'}
 =\overline E_{\boldsymbol N'}M_{b_{\boldsymbol N'}}
       \overline E_{\boldsymbol N'}^*.
\]
Since
\(\mathcal E_{\boldsymbol N}
 =\overline E_{\boldsymbol N'}E_{M_k}\),
Proposition~\ref{prop:two-scale} first gives the equalities
below with \(\boldsymbol N'\) fixed. The subsequent arrows
take the remaining limits from \(N_{k-1}\) down to \(N_1\);
all limits are strong:
\begin{align*}
 \lim_{N_k\to\infty}
 \mathcal E_{\boldsymbol N}R_{F_{\boldsymbol N}}
       \mathcal E_{\boldsymbol N}^*
 &=
 \overline E_{\boldsymbol N'}\mathcal R_{\boldsymbol N'}
       \overline E_{\boldsymbol N'}^*
 +U_k\widehat R_{\boldsymbol N'}W_k \rightarrow
 R_k+U_kR^{(k-1)}W_k=R^{(k)},\\[1ex]
 \lim_{N_k\to\infty}
 \mathcal E_{\boldsymbol N}P_{F_{\boldsymbol N}}
       \mathcal E_{\boldsymbol N}^*
 &=
 \overline E_{\boldsymbol N'}\mathcal P_{\boldsymbol N'}
       \overline E_{\boldsymbol N'}^*
 +U_k\widehat P_{\boldsymbol N'}W_k
 +U_k\widehat R_{\boldsymbol N'}W_k\widehat B_{\boldsymbol N'}\\
 &\rightarrow
 P_k+U_kD^{(k-1)}W_k
     +U_kR^{(k-1)}V_k=D^{(k)},\\[1ex]
 \lim_{N_k\to\infty}
 \mathcal E_{\boldsymbol N}M_{F_{\boldsymbol N}'}
       \mathcal E_{\boldsymbol N}^*
 &=
 \widehat S_{\boldsymbol N'}\otimes I_{\mathcal K}
 +\widehat B_{\boldsymbol N'}\rightarrow
 M_{S_{k-1}+b_k}=M_{S_k}.
\end{align*}
Indeed, the induction hypothesis gives
\(\widehat R_{\boldsymbol N'}\to R^{(k-1)}\),
\(\widehat P_{\boldsymbol N'}\to D^{(k-1)}\), and
\(\widehat S_{\boldsymbol N'}\to M_{S_{k-1}}\).
The preceding sampling argument gives the local-operator
limits and \(\widehat B_{\boldsymbol N'}\to M_{b_k}\),
so \(W_k\widehat B_{\boldsymbol N'}\to W_kM_{b_k}=V_k\).
Integration and constant extension in \(x_k\) commute with
unfolding the preceding coordinate, and
\(\overline E_{\boldsymbol N'}^*
 \overline E_{\boldsymbol N'}=I\)
preserves the products in these identities.
Strong convergence persists under tensoring with
\(I_{\mathcal K}\), and the common bounds justify passage
to limits in products. This completes the induction.
To identify the adjoint limit, set
\[
 T^{(0)}=T_0,\qquad
 T^{(j)}=T_j+U_jT^{(j-1)}W_j.
\]
The identity
\(W_jM_{S_j}=M_{S_{j-1}}W_j+V_j\) and the local kernel
identities give
\[
 D^{(j)}=R^{(j)}M_{S_j}-T^{(j)}.
\]
Both \(R^{(j)}\) and \(T^{(j)}\) are skew-adjoint. On the
realising graphs, the same identities first give the strong limit
of \(T\), and then the strong limit of \(P^*\), namely
\(-M_{S_j}R^{(j)}+T^{(j)}=(D^{(j)})^*\).
Taking limits of the norm bounds gives \eqref{eq:depth-uniform-bounds}. Next, for the prescribed vector \(u\in \cH_k\) take
\[
 F=F_{\boldsymbol N},\qquad u_F=\cE_{\boldsymbol N}^*u.
\]
The established limits, including
\(-z\cE_{\boldsymbol N}\cE_{\boldsymbol N}^*u\to-zu\), give
\[
 \cE_{\boldsymbol N}(P_F^*-z)u_F
       \rightarrow ((D^{(k)})^*-z)u,
 \qquad \norm{u_F}_2\rightarrow\norm{u}_{\cH_k}.
\]
The unfolding is an isometry, so these limits give the required
norm estimates. To attain a prescribed error, choose the coarser
integers first and then the finer ones, with each finer threshold
depending on earlier choices. Since \(\cE_{\boldsymbol N}1=1\),
\(\int_\T u_F=\int_{\T^{k+1}}u\) exactly.
\end{proof}

The ordered limits in Proposition~\ref{prop:finite-depth} establish
approximation at finite scales, rather than define the realising graph
by a geometric limit. With the profiles fixed, \(F_{\boldsymbol N}\)
converges uniformly in height to \(x\mapsto mx+h_0(x)\) as all inserted
scales vanish, while the unfolded operator limits retain the finer-scale contributions.
For each fixed depth, density and tolerance, the proof selects finite
integers from coarse to fine so that \(u_F=\cE_{\boldsymbol N}^*u\)
and its adjoint residual have the required accuracy. In the stopped
construction of Lemma~\ref{lem:hard-stop}, smooth cutoffs are chosen
first; neither the geometric limit nor a sharp-cutoff limit at fixed
scales is used for this realisation.

\subsection{Stopping and the slope bound}
Fix \(a>0\) and two smooth graphs
\[
 F_m(v)=mv+h_m(v),\qquad m\in\{a,-a\},
\]
with period-one corrections \(h_m\) and \(\abs{F_m'}\le L\). Their periodic double layer and
auxiliary singular integral are \(P_m=P_{F_m}\) and
\(R_m=R_{F_m}\). Suppose that the same interval
\(J\Subset(0,1)\) is strictly inside an affine interval on which
\(F_m'=-m\), and that \(0<|J|<1\).
The slope on this interval is the tilt of the other graph.
Choose an initial tilt \(m_0\in\{a,-a\}\) and take \(m=m_0\)
in the preceding construction. Define \(m_j=(-1)^jm_0\) and
\begin{gather*}
 A_0=1,\qquad
 A_j(x_0,\ldots,x_{j-1})=\prod_{i<j}\one_J(x_i),\\
 h_j^{\mathrm{stop}}(x_0,\ldots,x_j)
       =A_j(x_0,\ldots,x_{j-1})h_{m_j}(x_j).
\end{gather*}
The next correction is inserted precisely when every preceding
coordinate lies in \(J\). This defines the stopping rule on the
product torus. We now approximate its operators by
those of smooth graphs.

\begin{lemma}\label{lem:hard-stop}
At each fixed depth the recursions \eqref{eq:finite-r} and
\eqref{eq:finite-d} with these stopped profiles, as well as their
adjoints, are strong limits of the smooth finite-depth realisations
in Proposition \ref{prop:finite-depth}. They obey
\eqref{eq:depth-uniform-bounds} with the same \(L+1\).
\end{lemma}

\begin{proof}
Fix a depth \(k\). We choose the profiles in
Proposition~\ref{prop:finite-depth} to implement the stopping rule.
For \(0\le i<k\), choose smooth functions
\(0\le\chi_i^{(n)}\le1\), supported in \(J\), with
\(\chi_i^{(n)}\to\one_J\) almost everywhere, and take
\[
 h_j^{(n)}(x_0,\ldots,x_j)
   =\left(\prod_{i=0}^{j-1}\chi_i^{(n)}(x_i)\right)h_{m_j}(x_j),
 \qquad m_j=(-1)^jm_0,\quad 0\le j\le k.
\]
The empty product is one, so \(h_0^{(n)}=h_{m_0}\), and these
profiles converge almost everywhere to the limiting profiles
specified above. With \(m=m_0\), substitution into
\eqref{eq:actual-finite-graph} gives
\[
 \begin{aligned}
 F_{\boldsymbol N}^{(n)}(x)
   &=F_{m_0}(x)+\sum_{j=1}^k\epsilon_j
       \left(\prod_{i=0}^{j-1}
          \chi_i^{(n)}(x/\epsilon_i)\right)
          h_{m_j}(x/\epsilon_j),\\
 &\hspace{1em}\epsilon_0=1,\qquad
       \epsilon_j=(N_1\cdots N_j)^{-1}.
 \end{aligned}
\]

Each refinement adds the periodic correction \(h_{m_j}\) to the
inherited affine part. The term \(m_jv\) of the tilted graph is
already supplied by that inherited part when the preceding
cutoffs all equal one. The product coordinates in the displayed
graph are sampled at \(x_i=x/\epsilon_i\pmod1\). Once an earlier
factor vanishes, every later correction vanishes as well. In the
indicator limit, the \(j\)th correction is present precisely on
\(J^j\times\T\). There \(h_{m_i}'(x_i)=-2m_i\) for \(i<j\),
so the inherited product-space slope is \(S_{j-1}=m_j\), and the
local graph is \(F_{m_j}\). The same is true for smooth cutoffs
where all preceding factors equal one. A merely positive product
can include transition factors strictly between zero and one;
the following calculation controls their slopes.

Fix \(n\), suppress it in the notation, and set
\(A_0=1\) and \(A_j=\prod_{i<j}\chi_i(x_i)\).
Then \(b_j=A_jh_{m_j}'(x_j)\). If \(A_j>0\), all preceding
variables lie in \(J\). Substituting \(h_{m_i}'=-2m_i\) for
\(i<j\) and \(m_{i+1}=-m_i\) into the definition of \(S_j\) gives
\begin{equation}\label{eq:convex-slope}
 S_j=\sum_{i<j}(A_i-A_{i+1})m_{i+1}
                       +A_j F_{m_j}'(x_j)
                       \qquad\text{on }\{A_j>0\}.
\end{equation}
At \(j=0\) this reads \(S_0=F_{m_0}'\).
This is a convex combination: the coefficients are nonnegative
and sum to \(\sum_{i<j}(A_i-A_{i+1})+A_j=A_0=1\).
Also \(\abs{m_i}\le L\), since each tilt is the mean of a
graph derivative over one period. Hence \(\abs{S_j}\le L\).
If \(A_j=0\), let \(\ell<j\) be the first index for which
\(\chi_\ell(x_\ell)=0\). Then \(A_\ell>0\), so the same identity
bounds \(S_\ell\). All later profiles vanish and \(S_j=S_\ell\).
Thus the same slope bound holds on the entire product torus.
This is the hypothesis of Proposition~\ref{prop:finite-depth}.
The derivative of the realising graph also contains the
cross derivatives in earlier coordinates, multiplied by
\(\epsilon_j/\epsilon_i\) for \(i<j\). As in that proposition,
sufficiently separated scales bound their sum by one, giving
the graph bound \(L+1\). The required thresholds may depend on
the cutoff derivatives.

For these same profiles, the local graphs are
\(v\mapsto S_{j-1}(\mathbf x)v+A_j(\mathbf x)h_{m_j}(v)\),
and their operators enter \eqref{eq:finite-r} and
\eqref{eq:finite-d}. At fixed depth, all local graphs belong to
the compact family \(v\mapsto sv+\lambda h_m(v)\), with
\(\abs{s}\le L\) and \(0\le\lambda\le1\).
These parameters converge almost everywhere as the cutoffs tend
to their indicators. Norm continuity from
Lemma~\ref{lem:smooth-continuity}, including for the local adjoints,
and a common operator bound on this compact family permit dominated
convergence of the squared fibre norms. Hence the local operators,
their adjoints, and the bounded slope multipliers converge strongly.
The finite recursions give strong convergence of \(R,D,T\);
the identities \((R^{(j)})^*=-R^{(j)}\) and
\((D^{(j)})^*=-M_{S_j}R^{(j)}+T^{(j)}\), with the same uniform
bounds, give the required adjoint limits.

For a fixed vector and its adjoint residual, first choose the
cutoffs to approximate the indicator-based product-space operators.
Then apply Proposition~\ref{prop:finite-depth} to these fixed smooth
profiles, choosing the integers from coarse to fine according to
\eqref{eq:ordered-limits}. This gives the required smooth realising
graph. Passing to the cutoff limit preserves
\eqref{eq:depth-uniform-bounds} with the same \(L+1\);
the indicators appear only in the product-space limit.
\end{proof}

Under the resolvent assumptions of Lemma~\ref{lem:pair-update} below,
the indicator rule gives an exact recursion using the two fixed
local graphs, with the identity update after refinement stops. For each
fixed depth and density, the strong adjoint convergence in
Lemma~\ref{lem:hard-stop} permits smooth cutoff approximation with an
arbitrarily small change in the residual. After those cutoffs are
fixed, Proposition~\ref{prop:finite-depth} transfers the density and
residual estimate to a smooth graph at sufficiently separated finite
scales. The later growth argument in Section \ref{sec:criterion} can therefore use the exact recursion
without requiring sharp cutoffs in the realising graph.

\subsection{The resolvent recursion}
We solve each new equation by averaging in the newest coordinate
and then recovering its mean-zero fluctuation. The fluctuation
depends on both the old density and its auxiliary image, which
therefore form the pair propagated below.

Fix \(z\in\C\setminus\{0\}\). On \(\T\), let
\(Wf=\int_\T f(v)\dd v\), \((U\xi)(v)=\xi\), and \(Q=I-UW\),
so \(QL^2(\T)=L^2_0(\T)\).
Since \(P_m1=0\), this mean-zero space is invariant for \(P_m^*\).
Assume that both inverses
\[
 G_m=(P_m^*|_{L^2_0(\T)}-z)^{-1}:
           L^2_0(\T)\rightarrow L^2_0(\T)
\]
exist, and set
\begin{equation}\label{eq:cell-alpha-beta}
 b_m=F_m'-m,\quad c_m=R_m1,\quad w_m=P_m^*1,\qquad
 \alpha_m=-G_mw_m,\quad \beta_m=G_mb_m.
\end{equation}
Periodicity gives \(\int_\T b_m=0\), and \(P_m1=0\) gives
\(\int_\T w_m=0\). Thus \(\alpha_m,\beta_m\) are well-defined
mean-zero functions. Define the \textit{refinement update matrix}
\begin{equation}\label{eq:pair-matrix}
 M_m(v)=
 \begin{pmatrix}
  1+\alpha_m(v)&\beta_m(v)\\
  \bigl(R_m(1+\alpha_m)\bigr)(v)&1+\bigl(R_m\beta_m\bigr)(v)
 \end{pmatrix}.
\end{equation}
Its entries lie in \(L^2(\T)\), which suffices for the matrix
moments used below. 

\begin{lemma}\label{lem:pair-update}
The stopped finite-depth operators \(D^{(j)}\) have
\((D^{(j)})^*-z\) invertible. For \(j\ge1\), take
\(f_{j-1}\in\cH_{j-1}\), put \(f_j=U_jf_{j-1}\), and define,
for \(i=j-1,j\),
\[
 u_i=((D^{(i)})^*-z)^{-1}f_i,\qquad v_i=R^{(i)}u_i.
\]
Then \(W_ju_j=u_{j-1}\). With
\(\mathbf x=(x_0,\ldots,x_{j-1})\) and \(t=x_j\), on
\(J^j\times\T\) we have
\begin{equation}\label{eq:pair-update}
 \binom{u_j(\mathbf x,t)}{v_j(\mathbf x,t)}
       =M_{m_j}(t)
                         \binom{u_{j-1}(\mathbf x)}{v_{j-1}(\mathbf x)}.
\end{equation}
When \(\mathbf{x}\notin J^j\) we have \(u_j(\mathbf{x},t)=u_{j-1}(\mathbf{x})\) and \(v_j(\mathbf{x},t)=v_{j-1}(\mathbf{x})\). All update identities
hold in \(L^2\), and hence almost everywhere.
\end{lemma}

\begin{proof}
For the periodic operator of the initial graph, the decomposition
into constants and mean-zero
functions gives
\[
 P_m^*-z=
 \begin{pmatrix}
  -z&0\\ w_m&P_m^*|_{L^2_0(\T)}-z
 \end{pmatrix},\qquad
 (P_m^*-z)^{-1}=
 \begin{pmatrix}
  -z^{-1}&0\\ z^{-1}G_mw_m&G_m
 \end{pmatrix}.
\]
Here \(w_m\) in the lower-left block denotes multiplication of a
scalar by \(w_m\).

At level \(j\), split
\(\cH_j=U_j\cH_{j-1}\oplus\ker W_j\), separating the constant
and mean-zero parts in \(t=x_j\). The identities needed for both
the block structure and averaging are
\[
 W_jU_j=I,\qquad P_jU_j=0,\qquad W_jP_j^*=0,\qquad
 W_jM_{b_j}U_j=0.
\]
The middle two follow from the periodic identity \(P_{j,\mathbf x}1=0\);
the last follows from \(\int_\T b_j(\mathbf x,t)\dd t=0\).
Taking the adjoint in \eqref{eq:finite-d}, using the skew-adjointness
of \(R^{(j-1)}\), gives
\[
 (D^{(j)})^*=P_j^*+U_j(D^{(j-1)})^*W_j
                   -M_{b_j}U_jR^{(j-1)}W_j.
\]
Thus, identifying the first summand with \(\cH_{j-1}\),
\[
 (D^{(j)})^*-z=
 \begin{pmatrix}
 (D^{(j-1)})^*-z&0\\
 P_j^*U_j-M_{b_j}U_jR^{(j-1)}&P_j^*|_{\ker W_j}-z
 \end{pmatrix}.
\]
The averaging identities also show that the bounded off-diagonal
block takes values in \(\ker W_j\).

On an active tuple \(\mathbf x\in J^j\), the successive affine
slopes alternate, so \(S_{j-1}(\mathbf x)=m_j\) and
\(P_{j,\mathbf x}=P_{m_j}\). Otherwise the new profile is zero
and \(P_{j,\mathbf x}=0\). Consequently the fibre inverse on
\(L^2_0(\T)\) is
\[
 G_{j,\mathbf x}:=(P_{j,\mathbf x}^*|_{L^2_0(\T)}-z)^{-1}
   =\begin{cases}
       G_{m_j},&\mathbf x\in J^j,\\
       -z^{-1}I,&\mathbf x\notin J^j.
     \end{cases}
\]
Its direct integral \(G_j:\ker W_j\to\ker W_j\) has norm at most
\(\max\{\norm{G_a},\norm{G_{-a}},|z|^{-1}\}\). The other diagonal
block is invertible by induction, so the triangular operator is
invertible. This local inverse bound is independent of depth;
the full inverse norm may depend on depth.

Now average \(((D^{(j)})^*-z)u_j=U_jf_{j-1}\).
The displayed identities remove the local-adjoint and coupling
terms and give
\[
 ((D^{(j-1)})^*-z)W_ju_j=f_{j-1}
                    =((D^{(j-1)})^*-z)u_{j-1}.
\]
Uniqueness yields \(W_ju_j=u_{j-1}\). Hence
\(q_j:=u_j-U_ju_{j-1}=Q_ju_j\) has zero mean in \(t\).
Substitute \(u_j=U_ju_{j-1}+q_j\) and
\(R^{(j-1)}W_ju_j=v_{j-1}\) in the new equation:
\[
 \begin{aligned}
 (P_j^*-z)q_j+P_j^*U_ju_{j-1}-M_{b_j}U_jv_{j-1}+U_j\bigl((D^{(j-1)})^*-z\bigr)u_{j-1}
       =U_jf_{j-1}.
 \end{aligned}
\]
The old resolvent equation cancels the last term against the
right-hand side, leaving
\begin{equation}\label{eq:fluctuation}
 (P_j^*-z)q_j=-P_j^*U_ju_{j-1}+M_{b_j}U_jv_{j-1}.
\end{equation}
Writing \(w_j(\mathbf x,t)=(P_{j,\mathbf x}^*1)(t)\), the
fibre action is
\[
 (P_j^*U_ju_{j-1})(\mathbf x,t)
     =w_j(\mathbf x,t)u_{j-1}(\mathbf x).
\]
Both right-hand terms have zero \(t\)-integral, by
\(P_{j,\mathbf x}1=0\) and periodicity of the new profile.
The equation therefore lies in \(\ker W_j\), where its inverse gives
\[
 q_j(\mathbf x,\cdot)
   =G_{j,\mathbf x}\!\left[
       -w_j(\mathbf x,\cdot)u_{j-1}(\mathbf x)
       +b_j(\mathbf x,\cdot)v_{j-1}(\mathbf x)\right].
\]
The old-coordinate values are scalars on each fibre and can be
taken through this inverse. On \(J^j\) the local data are those
of the graph \(F_{m_j}\), so \(\alpha_{m_j}=-G_{m_j}w_{m_j}\) and
\(\beta_{m_j}=G_{m_j}b_{m_j}\) give the first row:
\[
 u_j(\mathbf x,t)
   =(1+\alpha_{m_j}(t))u_{j-1}(\mathbf x)
                     +\beta_{m_j}(t)v_{j-1}(\mathbf x).
\]

For the second row, \eqref{eq:finite-r} and the averaged identity give
\(v_j=R_ju_j+U_jv_{j-1}\). At a fixed active tuple, substitute
the first row into the input of \(R_{m_j}\) and collect the old-coordinate
factors:
\begin{align*}
 v_j(\mathbf x,t)
   &=\Bigl(R_{m_j}\bigl[(1+\alpha_{m_j})u_{j-1}(\mathbf x)
                   +\beta_{m_j}v_{j-1}(\mathbf x)\bigr]\Bigr)(t)+v_{j-1}(\mathbf x)\\
   &=\bigl(R_{m_j}(1+\alpha_{m_j})\bigr)(t)u_{j-1}(\mathbf x)+\bigl[1+\bigl(R_{m_j}\beta_{m_j}\bigr)(t)\bigr]
                                            v_{j-1}(\mathbf x).
\end{align*}
Together these rows prove \eqref{eq:pair-update}.

On an inactive tuple, the new profile, \(b_j\), and \(w_j\) vanish.
The fluctuation equation \eqref{eq:fluctuation} is then \(-zq_j(\mathbf x,\cdot)=0\), so
\(u_j(\mathbf x,t)=u_{j-1}(\mathbf x)\) when \(\mathbf{x}\notin J^j\). The local operator \(P_j^*\) has zero fibres because its graph is the local graph \(t\mapsto S_{j-1}(\mathbf{x})t\). Likewise, the fibres of the local auxiliary operator simplify to
\[
 R_{j,\mathbf x}=-\frac{1}{2(1+S_{j-1}(\mathbf x)^2)}H_\T, \qquad \mathbf{x}\notin J^j,
\]
which cancels any constant input because \(H_\T1=0\). Thus \(v_j(\mathbf x,t)=v_{j-1}(\mathbf x)\)
as well.

The products in the active rows belong to \(\cH_j\): by Fubini,
\(\one_{J^j}(\mathbf x)a_0(t)g(\mathbf x)\) has norm
\(\|a_0\|_{L^2(\T)}\|\one_{J^j}g\|_{\cH_{j-1}}\) when
\(a_0\in L^2(\T)\) and \(g\in\cH_{j-1}\).
The periodic coefficients and their auxiliary images are in \(L^2(\T)\);
inactive products are constant extensions. 

Finally, \eqref{eq:finite-d}, \(P_jU_j=0\), and \(V_jU_j=0\)
give \(D^{(j)}1=0\) by induction. Integration of the resolvent
equation therefore gives
\(\int_{\T^{j+1}}f_j=-z\int_{\T^{j+1}}u_j\).
Since \(z\ne0\), mean-zero forcing yields a mean-zero solution.
\end{proof}
\section{Growth under refinement}\label{sec:criterion}

We next give a sufficient condition for solutions of the resolvent equations
to grow with the number of refinements. The condition involves two positive
maps on Hermitian matrices. Let \(\Herm_2\) denote the real vector
space of Hermitian \(2\times2\) matrices. For these matrices,
\(X\ge Y\) means that \(X-Y\) is positive semidefinite, and
\(X>Y\) means that it is positive definite. The functions \(M_m\)
in \eqref{eq:pair-matrix} take values in \(2\times2\) matrices;
for \(X\in\Herm_2\), define the linear maps
\begin{equation}\label{eq:moment-maps}
 K_m(X)=\int_J M_m(v)^*XM_m(v)\dd v,\qquad
 L_m(X)=\int_{\T\setminus J}M_m(v)^*XM_m(v)\dd v.
\end{equation}
Both integrals use horizontal Lebesgue measure, with \(|\T|=1\),
without division by the lengths of the integration regions.
If \(X\ge0\), then \(M_m(v)^*XM_m(v)\ge0\), so both maps
preserve positivity and hence preserve the order on \(\Herm_2\).

Fix an incoming pair \(\xi\in\C^2\). Let
\(\xi_r^m(x_1,\ldots,x_r)\) be the stopped pair after \(r\)
new refinements, with the first inserted graph having tilt \(m\).
Define the Hermitian matrix \(Y_r^m\) by its quadratic form
\[
 \xi^*Y_r^m\xi
   =\int_{\T^r}\norm{\xi_r^m(x_1,\ldots,x_r)}_{\C^2}^2
       \dd x_1\cdots\dd x_r,
 \qquad Y_0^m=I.
\]
The first update is \(\xi_1^m(x_1)=M_m(x_1)\xi\), so
\[
 Y_1^m=\int_\T M_m(x_1)^*M_m(x_1)\dd x_1.
\]
After two refinements the alternatives are
\[
 \xi_2^m(x_1,x_2)=
 \begin{cases}
 M_m(x_1)\xi,&x_1\notin J,\\
 M_{-m}(x_2)M_m(x_1)\xi,&x_1\in J.
 \end{cases}
\]
The first update is retained in either case. Once refinement stops,
the pair is independent of subsequent variables, whose integrals
give a factor of one. Thus
\[
 \begin{aligned}
 Y_2^m={}&\int_{\T\setminus J}M_m(x_1)^*M_m(x_1)\dd x_1+\int_J\int_\T
 M_m(x_1)^*M_{-m}(x_2)^*M_{-m}(x_2)M_m(x_1)
       \dd x_2\dd x_1.
 \end{aligned}
\]

For the general calculation, let
\[
 \mathcal M_\ell^m(x_1,\ldots,x_\ell)
   =M_{(-1)^{\ell-1}m}(x_\ell)\cdots M_m(x_1),
 \qquad \mathcal M_0^m=I.
\]
For \(r\ge1\), partitioning according to the first exit from \(J\)
gives
\[
 \begin{aligned}
 Y_r^m={}&\sum_{\ell=1}^{r-1}
   \int_{J^{\ell-1}\times(\T\setminus J)}
       (\mathcal M_\ell^m)^*\mathcal M_\ell^m
       \dd x_1\cdots\dd x_\ell+\int_{J^{r-1}\times\T}
       (\mathcal M_r^m)^*\mathcal M_r^m
       \dd x_1\cdots\dd x_r.
 \end{aligned}
\]
The \(\ell\)-th term records a first exit at \(x_\ell\), after
the \(\ell\)-th update. In the final term all first \(r-1\)
coordinates lie in \(J\), so the \(r\)-th update occurs for every
\(x_r\in\T\). The sum is empty when \(r=1\).

For \(r\ge2\), separate the \(\ell=1\) term and use
\[
 \mathcal M_\ell^m(x_1,\ldots,x_\ell)
   =\mathcal M_{\ell-1}^{-m}(x_2,\ldots,x_\ell)M_m(x_1).
\]
In each remaining term, factor \(M_m(x_1)^*\) and \(M_m(x_1)\)
outside the inner integrals. The products left inside are
\(\mathcal M_{\ell-1}^{-m}\) for \(2\le\ell\le r-1\), and
\(\mathcal M_{r-1}^{-m}\) in the final term. Their domains are
\(J^{\ell-2}\times(\T\setminus J)\) and
\(J^{r-2}\times\T\), respectively. Integration in
\(x_2,\ldots,x_r\) therefore gives exactly \(Y_{r-1}^{-m}\), and
\[
 \begin{aligned}
 Y_r^m
 &=\int_{\T\setminus J}M_m(x_1)^*M_m(x_1)\dd x_1
   +\int_J M_m(x_1)^*Y_{r-1}^{-m}M_m(x_1)\dd x_1=L_m(I)+K_m(Y_{r-1}^{-m}).
 \end{aligned}
\]
The same identity holds for \(r=1\), since \(Y_0^{-m}=I\).
Thus
\begin{equation}\label{eq:moment-recursion}
 Y_0^m=I,\qquad Y_r^m=L_m(I)+K_m(Y_{r-1}^{-m})\quad(r\ge1).
\end{equation}
Applying this recursion twice, for \(q\ge1\), gives
\[
 \begin{aligned}
 Y_{2q}^m
 &=L_m(I)+K_m(Y_{2q-1}^{-m})\\
 &=L_m(I)+K_m\bigl(L_{-m}(I)\bigr)
       +(K_m\circ K_{-m})(Y_{2q-2}^m)\ge(K_m\circ K_{-m})(Y_{2q-2}^m).
 \end{aligned}
\]
The discarded terms are positive semidefinite, and the composition
preserves order. Iterating from \(Y_0^m=I\) yields
\[
 Y_{2q}^m\ge(K_m\circ K_{-m})^q(I).
\]
Now specialise to \(m=a\) and set
\[
 A:=K_a\circ K_{-a},\qquad A(X)=K_a\bigl(K_{-a}(X)\bigr).
\]
Then
\begin{equation}\label{eq:survival-lower-bound}
 Y_{2q}^a\ge A^q(I).
\end{equation}

\begin{proposition}\label{prop:growth-criterion}
Suppose \(a=1/2\), \(L=7/2\), \(z\ne0\), the two inverses on mean-zero functions in
\eqref{eq:cell-alpha-beta} exist, and for some \(X\in\Herm_2\) and
\(\tau>1\),
\begin{equation}\label{eq:growth-hypothesis}
 \tfrac14 I\le X\le I,\qquad A(X)\ge\tau X.
\end{equation}
Then there are smooth periodic graphs \(F_q(x)=m_qx+h_q(x)\)
with \(\norm{F_q'}_\infty\le9/2\), and mean-zero vectors
\(u_q\in L^2(\T)\), such that
\begin{equation}\label{eq:periodic-quasimodes}
 \norm{u_q}_2=1,\qquad
 \norm{(P_{F_q}^*-z)u_q}_2\rightarrow0.
\end{equation}
\end{proposition}

\begin{proof}
Order preservation and \eqref{eq:growth-hypothesis} give
\[
 Y_{2q}^a\ge A^q(I)\ge A^q(X)\ge\tau^q X
                                      \ge\tfrac14\tau^q I.
\]
This bound holds for every initial vector. We can therefore use
one fixed function at every depth. Take \(F_{-a}\) as the initial
graph. Its slope on \(J\) is \(a\), so the first inserted graph
has tilt \(a\). Fix
\[
 g\in C_c^\infty(J),\qquad \norm{g}_2=1,\qquad \int_\T g=0,
\]
regarded as a smooth function on \(\T\) by extension by zero.
For example, take the normalised derivative of a nonconstant smooth
function compactly supported in \(J\). Define the forcing
\[
 f=(P_{-a}^*-zI)g.
\]
It has integral zero because \(P_{-a}1=0\) and \(\int_\T g=0\),
and its solution on one period of the initial graph is \(g\).
We use this same forcing at every depth.

For each \(q\ge1\), let \(u^{(q)}\in\cH_{2q}\) be the adjoint
solution after \(2q\) refinements with the prescribed stopping rule:
\[
 u^{(q)}=\bigl((D^{(2q)})^*-zI\bigr)^{-1}f,\qquad
 \bigl((D^{(2q)})^*-zI\bigr)u^{(q)}=f.
\]
On \(\cH_{2q}\), the forcing is extended constantly as
\(f(x_0,\ldots,x_{2q})=f(x_0)\), so its norm is independent of
\(q\). The inverse exists by Lemma~\ref{lem:pair-update}.
The initial graph is level zero. Define
\[
 \xi(x_0)=\binom{g(x_0)}{(R_{-a}g)(x_0)}.
\]
Fixing \(x_0\) and integrating the pair update over the remaining
variables gives
\begin{align*}
 \norm{u^{(q)}}_2^2+\norm{R^{(2q)}u^{(q)}}_2^2
 &=\int_J \xi(x_0)^*Y_{2q}^a\xi(x_0)\dd x_0
     +\int_{\T\setminus J}\norm{\xi(x_0)}_{\C^2}^2\dd x_0.
\end{align*}
For \(x_0\in J\), the matrices above account
for every later refinement. Outside \(J\), all updates are the identity.
The outside integral is nonnegative. Since \(g\) is supported in
\(J\) and has unit norm,
\[
 \begin{aligned}
 \int_J\norm{\xi(x_0)}_{\C^2}^2\dd x_0
 &=\int_J\bigl(|g(x_0)|^2+|(R_{-a}g)(x_0)|^2\bigr)\dd x_0\\
 &=1+\int_J|(R_{-a}g)(x_0)|^2\dd x_0\ge1.
 \end{aligned}
\]
Together with \(Y_{2q}^a\ge\tau^q I/4\), this gives
\begin{equation}\label{eq:pair-growth}
 \norm{u^{(q)}}_2^2+\norm{R^{(2q)}u^{(q)}}_2^2
 \ge\frac{\tau^q}{4}\int_J\norm{\xi(x_0)}_{\C^2}^2\dd x_0
 \ge\frac{\tau^q}{4}.
\end{equation}
The bound for the auxiliary operator in
\eqref{eq:depth-uniform-bounds}, which is independent of depth,
therefore gives
\[
 \norm{u^{(q)}}_2^2
 \ge\frac{\tau^q}{4\bigl(1+C_R(9/2)^2\bigr)}.
\]
The forcing \(f\) is fixed, so
\[
 \begin{aligned}
 \frac{\norm{((D^{(2q)})^*-zI)u^{(q)}}_2}
      {\norm{u^{(q)}}_2}
 =\frac{\norm f_2}{\norm{u^{(q)}}_2}\le 2\norm f_2\sqrt{1+C_R(9/2)^2}\,\tau^{-q/2}
 \rightarrow0.
 \end{aligned}
\]

Each \(u^{(q)}\) has mean zero: successive averaging in
Lemma~\ref{lem:pair-update} recovers \(g\), whose integral is zero.
Set
\[
 \widehat u^{(q)}=\frac{u^{(q)}}{\norm{u^{(q)}}_2}.
\]
At each fixed depth \(2q\), apply Lemma~\ref{lem:hard-stop} and
Proposition~\ref{prop:finite-depth} to this fixed normalised vector.
Choose sufficiently accurate smooth cutoffs first, and then finite
scale integers \(\boldsymbol N(q)\) from coarse to fine, so that
the realisation error tends to zero as \(q\to\infty\). Explicitly, fix a sequence \(0<\varepsilon_q<1/2\) with \(\varepsilon_q\to0\). At each depth \(2q\), first choose the smooth cutoffs in Lemma 4.2 so that the adjoint-action error on \(\widehat u^{(q)}\) is less than \(\varepsilon_q/2\), and then choose finite scale integers \(\boldsymbol N(q)\) from coarse to fine by applying Proposition 4.1 with tolerance \(\epsilon=\varepsilon_q/2\). The combined realisation error is therefore less than \(\varepsilon_q\).
The cutoffs change the operators; the vector \(\widehat u^{(q)}\)
is kept fixed. Let \(F_q\) be the resulting smooth graph and take
the preliminary density by exact folding:
\[
 \widetilde u_q=\cE_{\boldsymbol N(q)}^*\widehat u^{(q)},
 \qquad
 \cE_{\boldsymbol N(q)}:L^2(\T)\rightarrow\cH_{2q}.
\]
The realisation estimates and the preceding residual bound give
\[
 \norm{F_q'}_\infty\le\frac92,\qquad
 \norm{\widetilde u_q}_2\rightarrow1,\qquad
 \norm{(P_{F_q}^*-zI)\widetilde u_q}_2\rightarrow0.
\]
Here \(P_{F_q}\) is the exact periodised operator of the selected
graph. Choose each norm error less than \(1/2\), so that
\(\widetilde u_q\ne0\).

Definition~\eqref{eq:unfolding} gives \(E_N1=1\); hence the
iterated unfolding in Proposition~\ref{prop:finite-depth} satisfies
\(\cE_{\boldsymbol N(q)}1=1\). Its adjoint therefore preserves
the integral with respect to the measures of total mass one:
\begin{align*}
 \int_\T\cE_{\boldsymbol N(q)}^*w\dd x
   &=\int_{\T^{2q+1}}w\dd\mathbf x
 \quad(w\in\cH_{2q}),\\
 \int_\T\widetilde u_q\dd x
   &=\int_{\T^{2q+1}}\widehat u^{(q)}\dd\mathbf x=0.
\end{align*}
This holds for every chosen finite set of scales. Folding need not
preserve the norm exactly, so define
\[
 u_q=\frac{\widetilde u_q}{\norm{\widetilde u_q}_2},\qquad
 \int_\T u_q\dd x=0,\qquad \norm{u_q}_2=1.
\]
Then
\[
 \norm{(P_{F_q}^*-zI)u_q}_2
   =\frac{\norm{(P_{F_q}^*-zI)\widetilde u_q}_2}
          {\norm{\widetilde u_q}_2}\rightarrow0,
\]
which proves \eqref{eq:periodic-quasimodes}.
\end{proof}

\section{Two fixed graphs}\label{sec:local-data}

We now construct the two graphs required by
Proposition \ref{prop:growth-criterion}. We first define a piecewise
polynomial graph and establish a strict matrix inequality. Smoothing
then preserves this inequality and gives two fixed smooth graphs.
Appendix \ref{cert:appendix} contains the estimates for the
polynomial graph.

Throughout this section, \(z_0=-\ii/2\). The certified inequality
and the smoothing argument are first established at \(z_0\).
We then vary the spectral parameter while keeping both smooth
graphs fixed.

Set \(p=3/4\), \(\delta=1/16\), and
\[
 \Theta(t)=\frac{35t-35t^3+21t^5-5t^7}{16},\qquad -1\le t\le1.
\]
Define the periodic function \(w_0\) by the following formula, with
the rising transition from constant value \(0\) to constant value \(1\) taken across \(x=0\):
\[
 w_0(x)=
 \begin{cases}
  (1+\Theta(x/\delta))/2,&-\delta\le x\le\delta,\\
  1,&\delta\le x\le p-\delta,\\
  (1-\Theta((x-p)/\delta))/2,&p-\delta\le x\le p+\delta,\\
  0,&p+\delta\le x\le1-\delta.
 \end{cases}
\]
Let \(F(0)=0\) and \(F'=7/2-4w_0\). Since \(\Theta\) is odd it follows that
\(\int_\T w_0=p\), which in turn implies \(F(x+1)-F(x)=1/2\). Hence, the function \(h(x)=F(x)-x/2\) is periodic.
The identity \(\Theta'(t)=35(1-t^2)^3/16\) shows that \(w_0\in C^3\) and therefore \(F\in C^4\),
and
\[
 -\tfrac12\le F'\le\tfrac72,\qquad
 F'=-\tfrac12\ \hbox{on }[1/16,11/16].
\]
Choose the subinterval
\begin{equation}\label{eq:chosen-child}
 J=(5/64,43/64).
\end{equation}
The interval \(J\) lies inside the interval of constant slope, at distance
\(1/64\) from each endpoint. This leaves room to smooth the graph
without changing its slope on \(J\).
For \(F\) and its reflection \(-F\), the subscripts \(+\) and \(-\)
on periodic operators and resolvent data denote the tilts \(m=1/2\)
and \(m=-1/2\), respectively. We omit the subscript for \(m=1/2\)
on \(\alpha,\beta,M,K\).

\begin{lemma}[Local matrix inequality]\label{lem:certified-cell}
For the \(C^4\) graph just defined and its reflection \(-F\), the
resolvents on mean-zero functions at \(z=-\ii/2\) exist and have norms less
than \(40\). The maps \eqref{eq:moment-maps} for the exact graph
operators satisfy
\begin{equation}\label{eq:certified-growth}
 A(X)\ge\frac{21}{20}X+\frac{49}{4950}I,\qquad
 X=\begin{pmatrix}1/3&37\ii/200\\-37\ii/200&37/40\end{pmatrix},
 \qquad \tfrac14 I<X<I.
\end{equation}
\end{lemma}

\begin{proof}
Proposition~\ref{cert:local-lemma} gives
\(\norm{(P_+^*+\ii/2)^{-1}}<40\) on \(L^2(\T)\).
Since \(P_+1=0\), this inverse preserves mean-zero functions,
and its restriction is \(G_+\).

Reflection gives \(P_-=-P_+\) and \(R_-=R_+\).
Since the kernels are real and \(z=-\ii/2\),
\[
 G_-=-\overline{G_+},
 \qquad \overline{G_+}f=\overline{G_+(\overline f)}.
\]
Thus \(\norm{G_-}=\norm{G_+}<40\).
Using \(b_-=-b_+\) and \(w_-=-w_+\) in
\eqref{eq:cell-alpha-beta}--\eqref{eq:pair-matrix} gives
\(M_-=\overline{M_+}\), and hence
\[
 K_-(Y)=\overline{K_+(\overline Y)}.
\]
The map \(A=K_+\circ K_-\) therefore agrees with the map in
Proposition~\ref{cert:local-lemma}, which proves
\eqref{eq:certified-growth}, including the bounds on \(X\).
\end{proof}

For \(2\times2\) matrices, \(\norm{\cdot}_{\HS}\) denotes the
Hilbert--Schmidt (Frobenius) norm. We equip \(\Herm_2\) with
the matrix operator norm and linear maps on this space with
the induced norm.

\begin{lemma}\label{lem:fixed-smoothing}
There are two fixed \(C^\infty\) graphs \(F_m(v)=mv+h_m(v)\),
\(m=\pm1/2\), with slopes bounded by \(7/2\) and slope
\(-m\) on a neighbourhood of \(J\), whose resolvents on
mean-zero functions at \(z_0\) have norm less than \(80\), and whose
maps at \(z_0\) satisfy
\[
 K_{1/2}K_{-1/2}(X)
       \ge\tfrac{21}{20}X+\tfrac{49}{9900}I.
\]
\end{lemma}

\begin{proof}
Take a nonnegative smooth even approximate identity
\(\rho_\epsilon\) on \(\T\), supported in
\((-\epsilon,\epsilon)\), and replace \(h\) by
\(h_\epsilon=\rho_\epsilon*h\).
For \(\epsilon<1/64\), the slope on \(J\) remains \(-1/2\),
and convolution preserves the slope interval
\([-1/2,7/2]\). The graph \(F_\epsilon(x)=x/2+h_\epsilon(x)\)
converges to \(F\) in \(C^2\). Thus \(P_{F_\epsilon}\to P_F\) and
\(R_{F_\epsilon}\to R_F\) in operator norm.
The resolvent identity gives invertibility and a norm bound of
less than \(80\) for all sufficiently small \(\epsilon\).
It follows from \eqref{eq:cell-alpha-beta} and
\eqref{eq:pair-matrix} that \(M_\epsilon\to M\) in
\(L^2(\T;\HS)\), and hence that their moment maps converge
in norm. Reflect the smoothed graph to obtain the graph with tilt \(-1/2\).
The strict margin in \eqref{eq:certified-growth} gives the
claimed inequality for a sufficiently small \(\epsilon\), which
we now fix. Thus the two graphs are independent of the depth.
\end{proof}

\begin{lemma}[Stability in the spectral parameter]\label{lem:parameter-stability}
For the fixed smooth graphs in Lemma~\ref{lem:fixed-smoothing},
there exists \(\varepsilon\in(0,1/160)\) such that, whenever
\(|z-z_0|<\varepsilon\), both mean-zero resolvents \(G_m(z)\)
exist and
\begin{equation}\label{eq:stable-growth}
 \norm{G_m(z)}<160,\qquad
 A_z(X)\ge\frac{21}{20}X+\frac{49}{19800}I.
\end{equation}
Here \(X\) is the matrix in \eqref{eq:certified-growth}, and
\(A_z=K_{1/2,z}K_{-1/2,z}\) is formed from
\eqref{eq:cell-alpha-beta}--\eqref{eq:moment-maps} at parameter \(z\).
\end{lemma}

\begin{proof}
For \(|z-z_0|<1/160\), the bound
\(\norm{G_m(z_0)}<80\) and the Neumann series give
\[
 G_m(z)=\bigl[I-(z-z_0)G_m(z_0)\bigr]^{-1}G_m(z_0),
 \qquad
 \norm{G_m(z)}\le
 \frac{\norm{G_m(z_0)}}{1-|z-z_0|\norm{G_m(z_0)}}<160.
\]
Thus \(G_m(z)\) varies continuously in operator norm. The periodic
operators and the functions \(b_m,w_m\) are fixed, so
\(\alpha_m(z),\beta_m(z)\) and their images under \(R_m\)
vary continuously in \(L^2(\T)\). Consequently
\(M_m(z)\) varies continuously in \(L^2(\T;\HS)\).
For two such matrix functions \(M,N\) and \(Y\in\Herm_2\),
Cauchy--Schwarz gives
\[
 \begin{aligned}
 &\left\|\int_J\bigl(M(v)^*YM(v)-N(v)^*YN(v)\bigr)\dd v\right\|\\
 &\qquad\le\norm Y\bigl(\norm M_{L^2(J;\HS)}+\norm N_{L^2(J;\HS)}\bigr)
              \norm{M-N}_{L^2(J;\HS)}.
 \end{aligned}
\]
Hence the moment maps, and their composition \(A_z\), vary
continuously in norm. Choose \(\varepsilon\in(0,1/160)\) so small
that \(\norm{A_z(X)-A_{z_0}(X)}<49/19800\) whenever
\(|z-z_0|<\varepsilon\). The inequality in
Lemma~\ref{lem:fixed-smoothing} then gives \eqref{eq:stable-growth}.
\end{proof}

Fix \(t\in[1/2,1/2+\varepsilon)\) and \(z=-\ii t\) for the
remainder of the main proof. Lemma~\ref{lem:parameter-stability}
verifies the hypotheses of Proposition~\ref{prop:growth-criterion}
with \(\tau=21/20\). The two graph tilts remain \(\pm1/2\).

With this parameter and the two smoothed graphs fixed, choose \(g\) and its forcing \(f\) once. At each depth \(2q\), solve the adjoint resolvent equation
\[
\bigl((D^{(2q)})^*-zI\bigr)u^{(q)}=f,
\]
with the same forcing extended constantly in the new variables. Normalise \(u^{(q)}\) in \( \cH_{2q}\) and apply Lemma 4.2 and Proposition 4.1 to the resulting unit vector, choosing sufficiently accurate smooth stopping cutoffs and then sufficiently separated finite periods from coarse to fine. Exact folding, followed by normalisation, gives a mean-zero unit density \(u_q\) on one period of the resulting smooth graph \(F_q\), with
\[
\|(P_{F_q}^*-zI)u_q\|_2\rightarrow0
\qquad\text{as }q\to\infty.
\]
Both the graph and its periodised operator depend on \(q\). For this graph and vector, Section~\ref{sec:gluing} first chooses
a small interval of nonzero Floquet phases and truncates the resulting
vector on the full graph. It then chooses the taper radius and the scale
at which the graph is inserted into a bounded boundary.
The slope and operator bounds remain uniform. Higher
derivatives, inverse norms at a fixed depth, and the phase interval
may depend on the preceding choices.

\section{Localisation and the bounded domain}\label{sec:gluing}

We first transfer the one-period densities to compactly supported
densities on the full smooth graphs. We then insert rescaled copies into
one bounded boundary and complete the planar proof of
Theorem~\ref{thm:main}. Repeating the top boundary periodically gives
the additional conclusion in Corollary~\ref{cor:periodic-graph}.

\subsection{From one period to the full graph}\label{subsec:packets}

The adjoint fibre operators in Lemma~\ref{lem:floquet} have a
rank-one jump at zero phase. This term vanishes on mean-zero
densities, so phases sufficiently close to zero preserve their
small residuals.

\begin{lemma}\label{lem:full-line-packets}
The vectors in \eqref{eq:periodic-quasimodes} give densities
\(\varphi_q\in L^2(\Gamma_{F_q},\dd\sigma)\), with
\(\norm{\varphi_q}_2=1\), such that
\begin{equation}\label{eq:full-line-quasimodes}
 \norm{(D_{\Gamma_{F_q}}^*-z)\varphi_q}_2\rightarrow0.
\end{equation}
The densities may be chosen with compact horizontal support.
\end{lemma}

\begin{proof}
Write \(P^\theta\) for the fibres associated with \(F_q\), and set
\[
r_q:=\|(P_{F_q}^*-zI)u_q\|_{L^2(\mathbb T)}\rightarrow0.
\]
Choose tolerances \(\delta_q\downarrow0\). Since \(\int_{\mathbb T}u_q=0\), the correction in \eqref{eq:floquet-adjoint} vanishes:
\[
(P^{0,+})^*u_q=P_{F_q}^*u_q.
\]
By the one-sided operator-norm convergence, choose \(0<\theta_q<\pi\) such that
\[
\|(P^\theta)^*-(P^{0,+})^*\|<\delta_q
\qquad(0<\theta<\theta_q).
\]
Taking \(E_q=(\theta_q/3,2\theta_q/3)\), we therefore have
\[
\sup_{\theta\in E_q}
\|((P^\theta)^*-zI)u_q\|_{L^2(\mathbb T)}
\le r_q+\delta_q.
\]
Thus \(E_q\) has positive length and stays away from zero for each fixed \(q\), while lying sufficiently close to zero for the required accuracy.   
Set \(a_q=(1+F_q'^2)^{1/2}\). In the unweighted representation of the physical fibres, take the direct-integral vector
\[
\Psi_q(\theta,x)
=
\left(\frac{2\pi}{|E_q|}\right)^{1/2}
\mathbf1_{E_q}(\theta)\frac{u_q(x)}{\sqrt{a_q(x)}}.
\]
Apply the inverse gauged Floquet transform and return to physical arclength coordinates by defining
\[
\widetilde\varphi_q(t,F_q(t))
=
\frac{((\widetilde{\cB})^{-1}\Psi_q)(t)}{\sqrt{a_q(t)}}.
\]
Writing \(t=x+\ell\), with \(0\le x<1\) and \(\ell\in\mathbb Z\), Fourier inversion gives explicitly
\[
\widetilde\varphi_q(x+\ell,F_q(x+\ell))
=
\frac{u_q(x)}{a_q(x)\sqrt{2\pi|E_q|}}
\int_{E_q}e^{i\theta(x+\ell)}\,d\theta.
\]
The factor \(e^{i\theta x}\) restores the gauge, while the second factor \(a_q^{-1/2}\) converts the unweighted density to a physical density. This formula follows from the transform convention in §2.2 and the coordinate conversion (2.5).     
Unitarity and the physical adjoint fibre formula now give
\[
\|\widetilde\varphi_q\|_{L^2(\Gamma_{F_q})}^2
=
\int_0^1\frac{|u_q(x)|^2}{a_q(x)}\,dx
\]
and
\[
\begin{aligned}
&\|(D_{\Gamma_{F_q}}^*-zI)\widetilde\varphi_q\|_{L^2(\Gamma_{F_q})}^2=
\frac1{|E_q|}
\int_{E_q}\int_0^1
\frac{\bigl|[((P^\theta)^*-zI)u_q](x)\bigr|^2}{a_q(x)}
\,dx\,d\theta
\le(r_q+\delta_q)^2.
\end{aligned}
\]
Since \(\|u_q\|_{L^2(\mathbb T)}=1\) and
\[
1\le a_q\le\sqrt{1+(9/2)^2},
\]
the norms of \(\widetilde\varphi_q\) are bounded above and away from zero uniformly in \(q\). 
Finally, choose \(s_q\) sufficiently large that
\[
\|\mathbf1_{\{|t|>s_q\}}\widetilde\varphi_q\|_{L^2(\Gamma_{F_q})}
<\delta_q,
\]
and define
\[
\varphi_q=
\frac{\mathbf1_{\{|t|\le s_q\}}\widetilde\varphi_q}
{\|\mathbf1_{\{|t|\le s_q\}}\widetilde\varphi_q\|_{L^2(\Gamma_{F_q})}}.
\]
The uniform bound for \(D_{\Gamma_{F_q}}^*-zI\) makes the truncation change the residual by at most \(C\delta_q\). The denominator stays bounded away from zero, so
\[
\|\varphi_q\|_{L^2(\Gamma_{F_q})}=1,
\qquad
\|(D_{\Gamma_{F_q}}^*-zI)\varphi_q\|_{L^2(\Gamma_{F_q})}
\rightarrow0.
\]
Each \(\varphi_q\) has compact horizontal support, proving the lemma.
\end{proof}

\subsection{The bounded domain}\label{subsec:gluing}

We now place small copies of the compactly supported vectors on
disjoint pieces of a single boundary. Two choices are needed:
a sufficiently large segment of each graph to control the tail,
and a sufficiently small scale to fit these segments into the
bounded boundary.

\begin{lemma}\label{lem:gluing}
Suppose \(F_q:\R\to\R\) have Lipschitz constant at most \(L_0\),
and compactly supported densities
\(\varphi_q\in L^2(\Gamma_{F_q},\dd\sigma)\) satisfy
\(\norm{\varphi_q}_2=1\) and
\[
 \norm{(D_{\Gamma_{F_q}}^*-z)\varphi_q}_2\to0.
\]
Then there is a bounded simply connected Lipschitz domain with
boundary \(\Gamma\) and a weakly null sequence satisfying
\eqref{eq:desired-sequence}. Its top graph has Lipschitz constant
at most \(3L_0\) and is differentiable at the sole accumulation
point of the inserted pieces.
\end{lemma}

\begin{proof}
Translate each graph vertically so that \(F_q(0)=0\), and choose \(s_q\ge1\)
so that 
\[
\supp(\varphi_q)\subset \{(x,F_q(x)): |x|\le s_q\}.
\]
Throughout the proof, restrictions such as \(|x|>R\) refer to the horizontal coordinate of a point \((x,F_q(x))\) on the graph.
For the double layer and its physical adjoint, the kernel
has absolute value at most \(1/(2\pi|X-Y|)\).
On a graph with slope bounded by \(L_0\), it follows that
for \(R>2s_q\),
\begin{equation}\label{eq:graph-tail}
 \norm{\one_{\{|x|>R\}}D_{\Gamma_{F_q}}^*\varphi_q}_2
       \le C_{L_0}R^{-1/2}\norm{\varphi_q}_{L^1(\Gamma_{F_q})}.
\end{equation}
For integration points with coordinate \(y\) in the support, \(|x-y|\ge|x|/2\).
Integrating \(x^{-2}\) over \(|x|>R\) proves the estimate.

We first truncate the graph outside a region much larger than the
support of the vector. Choose \(R_q>2s_q\) so that the right side of
\eqref{eq:graph-tail} tends to zero. Let \(\chi_R\) equal one
on \([-R,R]\), vanish outside \([-2R,2R]\), and be linear
on each intervening interval. Let \(G_q=\chi_{R_q}F_q\).
Since \(|F_q(x)|\le L_0|x|\), the function \(G_q\) has
Lipschitz constant at most \(3L_0\), vanishes at the two
endpoints of its support, and agrees with \(F_q\) on
\([-R_q,R_q]\).

We next choose the positions and scales of the truncated graphs.
Let \(c_q=2^{-q-3}\). Choose positive \(r_q\) so small that
the intervals \([c_q-r_q,c_q+r_q]\) are disjoint,
\(r_q/c_q\to0\), and, with \(\rho_q=r_q/(2R_q)\),
\begin{equation}\label{eq:side-smallness}
 \rho_q^{1/2}\norm{\varphi_q}_{L^1(\Gamma_{F_q})}\to0.
\end{equation}
By further decreasing the radii, arrange that all intervals lie in
\((0,1/4)\) and that all heights in the following construction
have absolute value less than one.
Define
\[
 f(x)=\sum_q\rho_qG_q((x-c_q)/\rho_q),\qquad -1\le x\le1.
\]
The \(q\)-th summand is supported in \([c_q-r_q,c_q+r_q]\), and these intervals are pairwise disjoint. Each summand has Lipschitz constant
at most \(3L_0\) and is zero at its support endpoints.
For points in different intervals, each height is bounded by
\(3L_0\) times the distance to the endpoint facing the other
interval. The sum of these distances is at most the distance
between the points. Thus \(\Lip(f)\le3L_0\); the same argument
applies when one or both points lie off the supports.
On the \(q\)th interval, \(|f|\le L_0r_q\), so
\(f(x)/|x|\to0\) as \(x\to0\). Set \(f(0)=0\).
Thus \(f'(0)=0\).

Choose \(H>1+\norm f_\infty\), and define
\[
 \Omega=\{(x,y):-1<x<1,\ -H<y<f(x)\}.
\]
This domain is bounded. Its boundary consists of the top graph,
two vertical sides, and the bottom segment. Since \(f=0\) near
\(x=\pm1\), its four corners are rectangle corners. Thus the
boundary is locally a Lipschitz graph everywhere. The map
\[
 (x,t)\longmapsto
 \left(x,-H+(t+H)\frac{H+f(x)}H\right)
\]
is a bi-Lipschitz homeomorphism of the closed rectangle
\([-1,1]\times[-H,0]\) onto \(\overline\Omega\).
In particular, \(\Omega\) is simply connected. The top normal
points upwards, as required by the graph convention.

Transfer \(\varphi_q\) to the central part of the corresponding graph piece
by
\[
 \psi_q(c_q+\rho_qu,\rho_qF_q(u))
       =\rho_q^{-1/2}\varphi_q(u,F_q(u))
       \quad (|u|\le s_q),
\]
and set it to zero elsewhere on \(\Gamma\).
Arclength scales by \(\rho_q\), so \(\norm{\psi_q}_2=1\).
We estimate the residual on three parts of the boundary. For
points on the top graph with
\(|x-c_q|\le\rho_qR_q\), the graph and normal agree with the
scaled full graph. Homogeneity of the kernel identifies the residual
there with the scaled full-line residual, with the same norm bound.
For the rest of the top graph, let \(d=|x-c_q|>\rho_qR_q\).
Since \(R_q>2s_q\), the horizontal distance from the support of \(\psi_q\)
is at least \(d/2\). This density has \(L^1\) norm
\(\rho_q^{1/2}\norm{\varphi_q}_1\), so the kernel bound gives
\[
 |(D_\Gamma^*\psi_q)(x,f(x))|
       \le C\frac{\rho_q^{1/2}\norm{\varphi_q}_1}{d}.
\]
Arclength on the entire top graph is bounded by a fixed multiple
of \(\dd x\). Consequently
\[
 \norm{\one_{\mathrm{remaining\ top}}D_\Gamma^*\psi_q}_2^2
 \le C_{3L_0}\rho_q\norm{\varphi_q}_1^2
       \int_{|x-c_q|>\rho_qR_q}\frac{\dd x}{|x-c_q|^2}
 \le \frac{C_{3L_0}}{R_q}\norm{\varphi_q}_1^2.
\]
The factors from scaling have cancelled, and the estimate includes
all the other inserted graph pieces. Finally, the two sides
and the bottom have a fixed positive distance from the supports of
all \(\psi_q\). Their contribution is therefore bounded by
\(C\rho_q^{1/2}\norm{\varphi_q}_1\), which tends to zero
by \eqref{eq:side-smallness}. The scalar term \(-z\psi_q\)
vanishes off the support of \(\psi_q\). Together, these estimates give
the residual bound \eqref{eq:desired-sequence} on the whole boundary.

The supports of the \(\psi_q\) are pairwise disjoint. For any
\(g\in L^2(\Gamma)\),
\[
 \abs{\langle g,\psi_q\rangle}\le
 \norm{\one_{\supp\psi_q}g}_2\rightarrow0.
\]
Hence \(\psi_q\rightharpoonup0\), completing the proof.
\end{proof}

\begin{proof}[Proof of Theorem~\ref{thm:main} for \(d=2\)]
Fix \(t\in[1/2,1/2+\varepsilon)\), with \(\varepsilon\) as in
Lemma~\ref{lem:parameter-stability}, and set \(z=-\ii t\).
The two fixed smooth graphs satisfy the hypotheses of
Proposition~\ref{prop:growth-criterion} at this parameter. Lemma \ref{lem:full-line-packets} supplies compactly supported unit densities with vanishing residuals on the graphs \(\Gamma_{F_q}\), whose Lipschitz constants are bounded by \(L_0=9/2\). Lemma \ref{lem:gluing} produces
one bounded domain \(\Omega_t\) and the weakly null normalised sequence
\eqref{eq:desired-sequence}; its top slope is at most \(27/2\).

Here \(\Gamma=\partial\Omega_t\). If \(D_\Gamma^*+\ii t\) were
Fredholm, it would have a bounded
left parametrix modulo compact operators:
\(B(D_\Gamma^*+\ii t)=I+C\), with \(C\) compact.
Apply this identity to \(\psi_q\). The left side tends to zero
in norm and \(C\psi_q\to0\), contradicting
\(\norm{\psi_q}_2=1\). Thus \(D_\Gamma^*+\ii t\) is not
Fredholm. Taking adjoints shows that
\(D_\Gamma-\ii t\) is not Fredholm. Since the kernel is real,
complex conjugation gives the same conclusion for
\(D_\Gamma+\ii t\). This proves
\eqref{eq:main-essential-point} for \(d=2\).

The identity \(D_\Gamma1=-1/2\), established in
Section~\ref{sec:operators}, shows that \(L^2_0(\Gamma)\) is
invariant under \(D_\Gamma^*\).
Since \(\psi_q\rightharpoonup0\), subtracting
\(|\Gamma|^{-1}\int_\Gamma\psi_q\) changes it by a vector
tending to zero in norm. The resulting vectors are in
\(L^2_0(\Gamma)\), their norms tend to one, and their
restricted adjoint residuals still tend to zero. Thus
\(-\ii t\) lies in the spectrum of the restriction, whose
spectral radius is at least \(t\).
\end{proof}

\begin{remark}[Additional planar regularity]
\label{rem:planar-regularity}
The planar example in Theorem~\ref{thm:main} admits a variant
whose boundary is differentiable everywhere and \(C^\infty\)
away from a single point, with the same essential spectral
points.

To see this, replace the piecewise linear cutoff in
Lemma~\ref{lem:gluing} by \(\chi_R(x)=\chi(x/R)\), where
\(\chi\in C^\infty(\R)\), \(0\le\chi\le1\), \(\chi=1\) on
\([-1,1]\), \(\chi=0\) outside \([-2,2]\), and
\(\norm{\chi'}_\infty\le3/2\). Such a cutoff exists.
For \(1\le|u|\le2\),
\[
 \chi(u)\le\tfrac32(2-|u|),
 \qquad
 \chi(u)+|u|\,|\chi'(u)|\le3.
\]
Consequently, \(G_q=\chi_{R_q}F_q\) still has Lipschitz
constant at most \(3L_0\), and the support and height
estimates are unchanged. Since the graphs \(F_q\) are
smooth and the insertion intervals accumulate only at
\(0\), the modified top function \(f\) is smooth away
from \(0\). The same height estimate gives \(f'(0)=0\),
so \(f\) is differentiable everywhere.

Next round the four outer corners using \(C^\infty\)
transitions, leaving the inserted graph pieces unchanged.
The resulting domain remains bounded, simply connected
and Lipschitz. Write \(\widetilde\Gamma\) for its boundary
and \(A\) for the new arcs. Transfer the same densities
\(\psi_q\) to the unchanged central pieces. The preceding
residual estimates apply on the remaining boundary,
while the positive distance between \(A\) and the density
supports gives
\[
 \norm{\one_A D_{\widetilde\Gamma}^*\psi_q}_2
 \le C\rho_q^{1/2}\norm{\varphi_q}_1
 \rightarrow0
\]
by \eqref{eq:side-smallness}. The densities remain
normalised and weakly null, so the same Fredholm argument
yields
\[
 \pm\ii t\in
 \essspec(D_{\widetilde\Gamma};
          L^2(\widetilde\Gamma,\dd\sigma)).
\]
The unit normal exists everywhere and is continuous away
from \(0\). It cannot also be continuous at \(0\), since
then \(\widetilde\Gamma\) would be \(C^1\) and its
double-layer operator would be compact \cite{FJR1978},
contrary to these nonzero essential spectral points.
\end{remark}

The same construction gives a single periodic graph containing rescaled
segments from the successive smooth graphs within one period.

\begin{corollary}[A periodic graph]\label{cor:periodic-graph}
There exists \(\varepsilon>0\) such that, for every
\(t\in[1/2,1/2+\varepsilon)\), there is a real \(1\)-periodic
Lipschitz function \(F\), with \(\Lip(F)\le27/2\), such that
\[
 \pm\ii t\in\essspec(D_{\Gamma_F};L^2(\Gamma_F,\dd\sigma)),
 \qquad \Gamma_F=\{(x,F(x)):x\in\R\}.
\]
The normal to the graph points upwards.
\end{corollary}

\begin{proof}
Fix \(t\in[1/2,1/2+\varepsilon)\), with \(\varepsilon\) as in
Lemma~\ref{lem:parameter-stability}, and use the graphs, scales and
original compactly supported densities \(\psi_q\) from
Lemma~\ref{lem:gluing} in the planar construction above.
Let \(F\) be the
\(1\)-periodic extension of the top function \(f\) from
\([-1/2,1/2]\). The insertion intervals lie in \((0,1/4)\),
so \(f\) vanishes near both endpoints of this period. Hence
\(F\) is Lipschitz on \(\R\), with \(\Lip(F)\le27/2\).
On \(\Gamma_F\), define \(\widetilde\psi_q\) by the same scaled
density as \(\psi_q\) on the \(q\)th insertion in the central
period, and by zero everywhere else. In particular, the density
is zero on the other periods, and \(\norm{\widetilde\psi_q}_2=1\).

The normal points upwards. For \(|x-c_q|\le\rho_qR_q\), the
residual agrees with the scaled full-line residual, as in
Lemma~\ref{lem:gluing}. On the entire remaining graph, the same
kernel estimate gives
\[
 \bigl\|\one_{\{|x-c_q|>\rho_qR_q\}}
          D_{\Gamma_F}^*\widetilde\psi_q\bigr\|_2^2
 \le C\rho_q\norm{\varphi_q}_1^2
       \int_{|x-c_q|>\rho_qR_q}\frac{\dd x}{|x-c_q|^2}
 \le\frac{C'}{R_q}\norm{\varphi_q}_1^2\rightarrow0.
\]
This integral includes all other periods, and the scalar term
\(\ii t\widetilde\psi_q\) vanishes there. Thus
\(\norm{(D_{\Gamma_F}^*+\ii t)\widetilde\psi_q}_2\to0\).
The supports remain disjoint, so \(\widetilde\psi_q\rightharpoonup0\)
in \(L^2(\Gamma_F,\dd\sigma)\). The same Fredholm argument,
followed by adjunction and complex conjugation, proves the corollary.
\end{proof}

\section{Lifting to higher dimensions}\label{sec:lifting}

We transfer the planar sequence to the boundary of a product domain.
The additional coordinates are confined to a fixed cube. A smooth
cutoff in those coordinates gives an isometric embedding under which
the two adjoint operators intertwine modulo a Hilbert--Schmidt operator.

We first discuss the principal-value facts used below. All truncations
in this section exclude an ambient Euclidean ball.

\begin{lemma}\label{lem:ambient-pv}
Let \(U\subset\R^d\) be a bounded Lipschitz domain, with
\(S=\partial U\). The direct truncations in
\eqref{eq:closed-double-layer} and the target-normal truncations
\[
 K_{S,\epsilon}g(X)=\frac1{|S^{d-1}|}
 \int_{\substack{Y\in S\\|X-Y|>\epsilon}}
 \frac{(Y-X)\cdot n(X)}{|X-Y|^d}g(Y)\dd\sigma(Y)
\]
converge almost everywhere and strongly on \(L^2(S)\) to
\(D_S\) and \(D_S^*\), respectively. Moreover,
\(D_S1=-1/2\), so \(D_S^*\) preserves \(L^2_0(S)\).
\end{lemma}

\begin{proof}
Consider the coordinate singular integrals with kernels
\((X_i-Y_i)/(|S^{d-1}||X-Y|^d)\). These kernels are smooth, odd and
homogeneous of degree \(-(d-1)\). Mas's \(L^2\) variation estimate
\cite[Theorem~1.2]{Mas2013} applies on Lipschitz graphs of arbitrary
slope. For a compactly supported density, the truncations vanish at
each target when the radius is large enough, so the variation bounds
the maximal truncation. Localising to a finite boundary atlas gives
the same maximal \(L^2\) bound on \(S\): on each smaller target patch,
the sources outside a larger graph patch are positively separated.
Corollary~1.6 of \cite{Mas2013} gives almost-everywhere principal values,
since \(S\) is rectifiable and \(g\dd\sigma\) is a finite measure for
\(g\in L^2(S)\). Dominated convergence using the maximal bound gives
strong \(L^2\) convergence of each coordinate truncation
\(\mathcal R_{i,\epsilon}\).

Writing \(M_{n_i}\) for multiplication by the \(i\)th normal component,
we have
\[
 D_{S,\epsilon}=\sum_{i=1}^d\mathcal R_{i,\epsilon}M_{n_i},
 \qquad
 K_{S,\epsilon}=-\sum_{i=1}^d M_{n_i}\mathcal R_{i,\epsilon}.
\]
Thus both families have the asserted limits. At positive \(\epsilon\),
Fubini's theorem and the symmetric distance exclusion give
\(K_{S,\epsilon}=D_{S,\epsilon}^*\). Passing to limits in pairings
identifies the second strong limit with \(D_S^*\).

At a boundary differentiability point \(X\), the divergence theorem
with a radial cutoff about \(X\), followed by coarea, gives for almost
every \(\epsilon>0\)
\[
 D_{S,\epsilon}1(X)=
 -\frac{\mathcal H^{d-1}(U\cap\partial B(X,\epsilon))}
 {|S^{d-1}|\epsilon^{d-1}}.
\]
The rescaled domain tends to the inward half-space, so the spherical
fraction tends to \(1/2\). The principal-value limit therefore gives
\(D_S1=-1/2\) almost everywhere. Pairing with \(1\) yields
\(\int_S D_S^*g\dd\sigma=-\tfrac12\int_Sg\dd\sigma\), proving the
invariance assertion.
\end{proof}

\begin{proposition}\label{prop:dimensional-lifting}
Let \(\Omega\subset\R^2\) be a bounded Lipschitz domain, with
\(\Gamma=\partial\Omega\), and let \(m\ge1\) be an integer. Set
\[
 Q=(-1,1)^m,\qquad U=\Omega\times Q,\qquad\Sigma=\partial U.
\]
Then \(U\) is a bounded Lipschitz domain. Choose
\(\chi\in C_c^\infty(Q;\R)\) with \(\norm\chi_2=1\) and
\(\int_Q\chi\dd z=0\), and define
\[
 (Jf)(p,z)=f(p)\chi(z)\quad\text{on }\Gamma\times Q,
 \qquad Jf=0\quad\text{on the other boundary faces}.
\]
The map \(J:L^2(\Gamma)\to L^2_0(\Sigma)\) is an isometry, and
\begin{equation}\label{eq:compact-lifting}
 E=D_\Sigma^*J-JD_\Gamma^*
 \quad\text{is Hilbert--Schmidt}.
\end{equation}
\end{proposition}

\begin{proof}
We first check the geometry at product corners. In a planar boundary
chart let \(\Omega=\{v>\varphi(u)\}\). At a point incident to
\(k\) cube faces, choose inward coordinates \(t_1,\ldots,t_k\).
The local conditions are \(v>\varphi(u)\) and \(t_i>0\).
Rotate \((v,t_1,\ldots,t_k)=y+\tau e\), where
\(e=(1,\ldots,1)/\sqrt{k+1}\) and \(y\perp e\). The conditions become
\[
 \tau>\sqrt{k+1}\max\{\varphi(u)-y_0,-y_1,\ldots,-y_k\},
\]
a Lipschitz epigraph. At points with planar coordinate interior to
\(\Omega\), the same rotation of the active cube coordinates gives
a Lipschitz epigraph for the cube alone. This proves the assertion
about \(U\).

Up to a surface-measure null set of edges, \(\Sigma\) consists of
\(\Gamma\times Q\) and \(\Omega\times F\), where \(F\) ranges over
the open faces of the cube. Their measures are \(\dd s\dd z\) and
\(\dd x\dd\mathcal H^{m-1}\), and their normals are
\((n_\Gamma,0)\) and \((0,\nu_F)\), respectively. For \(m=1\),
\(\mathcal H^0\) is counting measure. Fubini's theorem now gives
\[
 \norm{Jf}_2=\norm f_2,\qquad
 \int_\Sigma Jf\dd\sigma=
 \left(\int_\Gamma f\dd s\right)\left(\int_Q\chi\dd z\right)=0.
\]
Such a cutoff is obtained by taking a product of real bumps, with one
factor replaced by its derivative, and then normalising.

Let \(A=D_\Gamma^*\), \(B=D_\Sigma^*\), and extend \(\chi\) by zero
to \(\R^m\). For \(r=|p-q|>0\), set
\begin{equation}\label{eq:transverse-kernel}
 k(p,q)=\frac{(q-p)\cdot n_\Gamma(p)}{2\pi r^2},
 \qquad
 P_r(h)=\frac{2\pi}{|S^{m+1}|}
             \frac{r^2}{(r^2+|h|^2)^{(m+2)/2}}.
\end{equation}
The lateral kernel of \(B\) factors exactly as \(k(p,q)P_r(z-w)\).
Polar integration gives
\begin{equation}\label{eq:transverse-mass}
 \int_{|h|\le a}P_r(h)\dd h
 =\left(\frac{a^2}{r^2+a^2}\right)^{m/2},\qquad
 \int_{\R^m}P_r(h)\dd h=1.
\end{equation}
Indeed, the radial density is the derivative of
\((\rho^2/(r^2+\rho^2))^{m/2}\). The same integration and scaling give
\(\int_{\R^m}|h|P_r(h)\dd h=C_mr\), with \(C_m<\infty\).
The unit mass in \eqref{eq:transverse-mass} is the reason the spectral
parameter is unchanged.

Since \(|k(p,q)|\le(2\pi r)^{-1}\), the kernel
\begin{equation}\label{eq:lifting-remainder}
 K_R(p,z;q)=k(p,q)\int_{\R^m}
 P_r(h)\bigl(\chi(z-h)-\chi(z)\bigr)\dd h
\end{equation}
is bounded by \(C_m\norm{\nabla\chi}_\infty/(2\pi)\). Set it equal
to zero on the null diagonal \(p=q\). The curve and cube have finite
measure, so it defines a Hilbert--Schmidt operator
\(R:L^2(\Gamma)\to L^2(\Gamma\times Q)\).

To identify this remainder with the actual principal values, let
\(A_\epsilon\) and \(B_\epsilon\) be the Euclidean truncations of
\(A\) and \(B\). By Lemma~\ref{lem:ambient-pv}, they converge strongly
to \(A\) and \(B\). The retained transverse mass is
\[
 M_\epsilon(r)=\int_{\R^m}P_r(h)
              \one_{\{r^2+|h|^2>\epsilon^2\}}\dd h
 =\begin{cases}
  1-(1-r^2/\epsilon^2)^{m/2},&0<r<\epsilon,\\
  1,&r\ge\epsilon,
 \end{cases}
\]
by \eqref{eq:transverse-mass}. With
\(w_m(s)=ms(1-s^2)^{m/2-1}\), \(0<s<1\), we have
\[
 \int_0^1w_m(s)\dd s=1,\qquad
 M_\epsilon(r)=\int_0^1w_m(s)\one_{\{r>\epsilon s\}}\dd s.
\]
Thus the weighted planar operator
\[
 C_\epsilon f(p)=\int_\Gamma
 k(p,q)M_\epsilon(|p-q|)f(q)\dd s(q)
\]
satisfies
\begin{equation}\label{eq:averaged-truncations}
 C_\epsilon f=\int_0^1w_m(s)A_{\epsilon s}f\dd s,
 \qquad
 \norm{C_\epsilon f-Af}_2
 \le\sup_{0<\rho\le\epsilon}\norm{A_\rho f-Af}_2\rightarrow0.
\end{equation}
These identities are justified by absolute integration. With
\(a_m=\max\{1,m/2\}\), the mass formula gives
\(M_\epsilon(r)\le a_m r^2/\epsilon^2\) for \(r\le\epsilon\),
so \(|k(p,q)|M_\epsilon(r)\le a_m/(2\pi\epsilon)\) for all \(r>0\).
Since \(f\in L^2(\Gamma)\subset L^1(\Gamma)\), Fubini applies.
The average in \eqref{eq:averaged-truncations} is a Bochner integral:
the truncated family is strongly measurable and uniformly bounded.

Let \(R_\epsilon\) be obtained from \eqref{eq:lifting-remainder} by
inserting \(\one_{\{r^2+|h|^2>\epsilon^2\}}\) into the transverse
integral. Splitting
\(\chi(z-h)=(\chi(z-h)-\chi(z))+\chi(z)\) gives the exact identity
\[
 B_\epsilon Jf=R_\epsilon f+\chi(z)C_\epsilon f
 \quad\text{on }\Gamma\times Q.
\]
The kernels of \(R_\epsilon\) have the same uniform bound as \(K_R\)
and agree with it when \(\epsilon<|p-q|\). Dominated convergence on
\(\Gamma\times Q\times\Gamma\) gives \(R_\epsilon\to R\) in
Hilbert--Schmidt norm. Passing to strong limits therefore gives
\((BJ-JA)f=Rf\) on the lateral boundary.

Finally, let \(\delta_\chi=\operatorname{dist}(\supp\chi,\partial Q)>0\).
For a target \((x,z)\in\Omega\times F\), the effective kernel is
\[
 H_F((x,z),q)=\frac1{|S^{m+1}|}\int_Q
 \frac{(w-z)\cdot\nu_F}{(|x-q|^2+|z-w|^2)^{(m+2)/2}}
 \chi(w)\dd w.
\]
On the source support, \(|z-w|\ge\delta_\chi\), so
\[
 |H_F((x,z),q)|\le
 \frac{\norm\chi_1}{|S^{m+1}|\delta_\chi^{m+1}}.
\]
Every face has finite measure, so these operators are Hilbert--Schmidt.
For \(\epsilon<\delta_\chi\), truncation leaves them unchanged.
The term \(JA\) vanishes on these faces. Combining their contributions
with the lateral remainder proves \eqref{eq:compact-lifting}.
\end{proof}

\begin{proof}[Proof of Theorem~\ref{thm:main} for \(d\ge3\)]
Fix \(t\) in the same interval as in the planar construction, and let
\(\Omega_t\) and \(\psi_q\) be the domain and sequence obtained in
Section~\ref{sec:gluing}. Set \(m=d-2\) and
\(\Omega_{t,d}=\Omega_t\times(-1,1)^m\). Proposition~\ref{prop:dimensional-lifting}
gives a bounded Lipschitz domain. The explicit subgraph description of
\(\Omega_t\) gives a homeomorphism of its closure with a closed rectangle,
mapping boundary to boundary. Taking its product with the identity on
the closed cube shows that \(\Omega_{t,d}\) is simply connected and its
boundary is connected.

Use \(J\) and \(E\) from Proposition~\ref{prop:dimensional-lifting}.
The vectors \(\Phi_q=J\psi_q\) have unit norm, belong to
\(L^2_0(\Gamma_{t,d})\), and converge weakly to zero. Compactness of
\(E\) and \eqref{eq:desired-sequence} give
\[
 (D_{\Gamma_{t,d}}^*+\ii t)\Phi_q
 =J(D_{\partial\Omega_t}^*+\ii t)\psi_q+E\psi_q\rightarrow0.
\]
Lemma~\ref{lem:ambient-pv} gives invariance of the mean-zero subspace.
The Fredholm argument from Section~\ref{sec:gluing}, applied on the full
space and on this invariant subspace, puts \(-\ii t\) in both adjoint
essential spectra. Reality gives \(\ii t\), and adjunction gives both
signs for the direct operator. This proves the theorem with the same
\(\varepsilon\) in every dimension.
\end{proof}

\addtocontents{toc}{\protect\setcounter{tocdepth}{1}}
\appendix
\section{Certification of the local graph}
\label{cert:appendix}

We prove the local matrix inequality used in the graph construction by
certifying an approximate pair of functions. First, an approximation in
Hilbert--Schmidt norm gives an inverse bound for the continuous local
operator. We then integrate rational polynomials exactly to bound the
residuals of the proposed functions. The inverse bound controls their
distance from the true pair, and a final perturbation estimate transfers
the positive matrix inequality to that pair.

The graph operators in this appendix act on the unweighted space
$L^2(\T,\dd x)$, where $\T=\R/\mathbb Z$ has measure one. Their
adjoints and operator norms refer to this space, and $1$ denotes the
constant function. For finite matrices, $\|\cdot\|_{\HS}$ is the
Hilbert--Schmidt (Frobenius) norm.
For $F(x)=mx+h(x)$, with $h$ real and periodic, the line kernels are
\begin{align}
 p_F(x,y)&=\frac{F(x)-F(y)-(x-y)F'(y)}
 {2\pi\big((x-y)^2+(F(x)-F(y))^2\big)},\label{cert:raw-P}\\
 r_F(x,y)&=-\frac{x-y}
 {2\pi\big((x-y)^2+(F(x)-F(y))^2\big)}.\label{cert:raw-R}
\end{align}
Their symmetric zero-phase periodisations define the operators
$P=P_F$ and $R=R_F$, with periodic kernels $K_P$ and $K_R$.
We take symmetric principal values. Since the periodised operator $P$
annihilates constants, $P^*$ preserves the mean-zero subspace. The
resolvent on this subspace determines the pair: set
\begin{equation}
 z=-\ii/2,\qquad b=F'-m,\qquad
 G=(P^*|_{\ker W}-z)^{-1},\qquad Wf=\int_\T f\dd x,
 \label{cert:quotient-definition}
\end{equation}
when the inverse exists, and define
\begin{equation}
 \alpha=-GP^*1,\qquad \beta=Gb,\qquad
 M=\begin{pmatrix}1+\alpha&\beta\\R(1+\alpha)&1+R\beta\end{pmatrix}.
 \label{cert:M-definition}
\end{equation}
The interval $J$ in the following proposition is the subinterval on which
the graph construction continues at the next scale.

\begin{proposition}[Certified local inequality]
\label{cert:local-lemma}
For the graph specified in Subsection~\ref{cert:profile}, let
\[
 J=(5/64,43/64),\qquad
 K(Y)=\int_J M(v)^*Y M(v)\dd v.
\]
Set $K_-(Y)=\overline{K(\overline Y)}$ and $A=K\circ K_-$.
Then
\begin{equation}
 \|(P^*+\ii/2)^{-1}\|<40,\qquad \|R\|<10,
 \label{cert:local-inverse-and-R}
\end{equation}
and, for
\begin{equation}
 X=\begin{pmatrix}1/3&37\ii/200\\-37\ii/200&37/40\end{pmatrix},
 \qquad \frac14 I<X<I,
 \label{cert:witness}
\end{equation}
one has
\begin{equation}
 A(X)\ge \frac{21}{20}X+\frac{49}{4950}I.
 \label{cert:final-local-inequality}
\end{equation}
The inverse in \eqref{cert:local-inverse-and-R} acts on the full space.
The bound for this inverse uses the interval enclosures in
Subsection~\ref{cert:arithmetic-model}; the subsequent residual and
matrix calculations use integers and rationals.
\end{proposition}

\subsection{The exact graph and exact input functions}
\label{cert:profile}

We use the $C^4$ graph $F$ and the function $w_0$ defined
in Section~\ref{sec:local-data}, with $m=1/2$ and $h=F-x/2$.
An explicit formula for the graph uses the polynomial
\[
 \mathcal P(t)=35t^2-4480t^4+458752t^6-20971520t^8.
\]
The exact polynomial pieces on $[0,1]$ are
\begin{equation}
 F(x)=
 \begin{cases}
  \frac32x-\mathcal P(x),
       &0\le x\le1/16,\\
  35/1024-x/2,&1/16\le x\le11/16,\\
  -157/512+\frac32(x-3/4)+\mathcal P(x-3/4),
       &11/16\le x\le13/16,\\
  \frac72x-3037/1024,&13/16\le x\le15/16,\\
  \frac12+\frac32(x-1)-\mathcal P(x-1),
       &15/16\le x\le1.
 \end{cases}
 \label{cert:profile-pieces}
\end{equation}
Extend by $F(x+1)=F(x)+1/2$. The pieces match through order four.
Across the periodic seam they are translates of one polynomial piece.
Only the four transition joins will therefore require separate estimates
when we enclose the kernel.

The polynomial pieces give
\begin{equation}
 \|F'\|_\infty\le L:=7/2,\quad
 \|F''\|_\infty\le M_2:=70,\quad
 \|F'''\|_\infty\le M_3:=6720,\quad
 \osc h\le a:=3/2.
 \label{cert:geometry-bounds}
\end{equation}
On a transition centred at $c$, put $u=16(x-c)$, so
$|u|\le1$. There $F''=\pm70(1-u^2)^3$, and $F''$ vanishes elsewhere.
The third-derivative bound follows by differentiation.
Also $h'=3-4w_0$ has absolute value at most three, so integration along
the shorter circular arc gives the oscillation bound.  On each
polynomial piece,
\begin{equation}
 |F^{(5)}|\le M_5:=55\,050\,240.
 \label{cert:M5}
\end{equation}
For example, three derivatives of $70(1-u^2)^3$, with $u=16x$ up to
translation and sign, are bounded by $70\cdot16^3(72+120)$.
At each transition join, $F''$, $F'''$, and $F''''$ vanish.  Consequently
the common fourth-order Taylor polynomial there is affine.
The interval $J$ lies in the plateau of slope $-1/2$, and each of its
endpoints is at distance $1/64$ from the adjacent transition endpoint.

The exact input \path{pair_rational_approximants.json} specifies integers $c_{l k j}^{\mathrm r}$ and
$c_{l k j}^{\mathrm i}$, $0\le l<64$, $0\le k<20$, $0\le j<4$,
and rational complex correction constants $\mu_j$.  In its JSON arrays, these are
\texttt{integers[$l$][$k$][$j$][0]},
\texttt{integers[$l$][$k$][$j$][1]}, and \texttt{means[$j$]};
the last entries give the real and imaginary parts as integer
numerator--denominator pairs.
The constants $\mu_0,\mu_1$ make the first two functions have
integral zero, while $\mu_2=\mu_3=0$. On $I_l=[l/64,(l+1)/64]$, set
\begin{equation}
 \varphi_j(x)=2^{-64}\sum_{k=0}^{19}
  (c_{l k j}^{\mathrm r}+\ii c_{l k j}^{\mathrm i})T_k(128x-2l-1)
  -\mu_j,
 \label{cert:data-definition}
\end{equation}
where $T_k$ is the Chebyshev polynomial. The four functions satisfy
$(\varphi_0,\varphi_1,\varphi_2,\varphi_3)=
(\widetilde\alpha,\widetilde\beta,\widetilde\gamma,\widetilde\delta)$.  These coefficients define exact piecewise polynomials. Continuity at the
panel endpoints and the two zero-mean conditions reduce to integer equalities.  Endpoint values use $T_k(1)=1$ and $T_k(-1)=(-1)^k$.
The mean uses
\begin{equation}
 \int_{-1}^1T_k(t)\dd t=
 \begin{cases}2/(1-k^2),&k\text{ even},\\0,&k\text{ odd}.
 \end{cases}
 \label{cert:Chebyshev-mean}
\end{equation}
These tests give continuity at all 64 joins, including the periodic
one, and $W\widetilde\alpha=W\widetilde\beta=0$ exactly.

The two operator inputs used throughout the calculation are
\begin{equation}
 f_0=1+\widetilde\alpha,\qquad f_1=\widetilde\beta.
 \label{cert:inputs}
\end{equation}
Using \eqref{cert:Chebyshev-mean} and the product identity for Chebyshev
polynomials, exact integration gives
\begin{equation}
 \|f_0\|_2^2<2,\qquad \|f_1\|_2^2<11.
 \label{cert:input-L2}
\end{equation}
Thus both $L^2$ and $L^1$ norms are less than four. We use this factor
to convert each subsequent kernel enclosure into an action error.

\subsection{Bounds for the continuous kernels}
\label{cert:kernel-bounds}

The inverse calculation requires uniform first-derivative bounds.
Let $\Phi(t)=(1+t^2)^{-1}$ and take
$B_s=4\ge\|h'\|_\infty$. Choose the nearest integer translate of $y$ so that
$r=x-y\in[-1/2,1/2]$. For the other translates, let $d_j=r-j$,
$H=h(x)-h(y)$, and $b_y=h'(y)$. Their direct kernel is
\begin{equation}
 k_j(x,y)=\frac1{2\pi}
  \left(-\frac{b_y}{d_j}+\frac H{d_j^2}\right)
   \Phi\left(m+\frac H{d_j}\right).
 \label{cert:remote-P}
\end{equation}
We use $|\Phi|\le1$, $|\Phi'|\le1$, and the elementary estimates
\begin{equation}
 \left|\sum_{j\ge1}\left(\frac1{r-j}+\frac1{r+j}\right)\right|
 \le2,\qquad
 \sum_{j\ne0}|r-j|^{-k}\le
 \begin{cases}\pi^2,&k=2,\\20,&k=3,\\112/3,&k=4.
 \end{cases}
 \label{cert:remote-sums}
\end{equation}
The first inequality follows by comparison with
$\sum_{j\ge1}(j^2-1/4)^{-1}=2$.  The second uses the half-integer
square sum for $k=2$, and a first-term plus integral bound for $k=3,4$.

To control the local term through the diagonal, define
\[
 Q_r=\int_0^1F'(y+tr)\dd t,\qquad
 B_r=\int_0^1(1-t)F''(y+tr)\dd t.
\]
Then $k_0=B_r\Phi(Q_r)/(2\pi)$, including the diagonal, and
\begin{align*}
 |B_r|&\le M_2/2,& |\partial_x B_r|&\le M_3/6,
 &|\partial_y B_r|&\le M_3/3,\\
 |\partial_x Q_r|&\le M_2/2,
 &|\partial_y Q_r|&\le M_2/2.
\end{align*}
For the derivative in $x$, the absolutely summable remote bound is
\[
 \frac1{2\pi}\left\{
 \frac{B_s^2+2B_s}{d_j^2}
 +\frac{2a(1+B_s)}{|d_j|^3}+\frac{a^2}{d_j^4}\right\}.
\]
For the derivative in $y$, if
$g=-b_y/d_j+H/d_j^2$ and $q=m+H/d_j$, then
\[
 \partial_y g=-b_y'/d_j-2b_y/d_j^2+2H/d_j^3,
 \qquad \partial_y q=g.
\]
The only conditionally summable term in $\partial_y(g\Phi(q))$ is
$-b_y'\Phi(m)/d_j$.  Its paired sum is bounded by $M_2/\pi$ after
the factor $1/(2\pi)$.  It follows that
\begin{align}
 |\partial_x K_P|&\le
 \frac{M_3}{12\pi}+\frac{M_2^2}{8\pi}
 \notag\\
 &\quad+\frac{(B_s^2+2B_s)\pi^2+40a(1+B_s)+(112/3)a^2}{2\pi}
 <\frac{2969}{6}<500,\label{cert:target-derivative}\\
 |\partial_y K_P|&\le
 \frac{M_3}{6\pi}+\frac{M_2^2}{8\pi}+\frac{M_2}{\pi}
 \notag\\
 &\quad+\frac{(M_2a+2B_s+B_s^2)\pi^2+40a(1+B_s)+(112/3)a^2}{2\pi}
 <\frac{5279}{6}<900.
 \label{cert:source-derivative}
\end{align}
The rational comparisons use only $3<\pi<22/7$ and $\pi^2<10$.
These estimates apply to the full periodised kernel. The removable representation and the $C^4$ matching
likewise exclude singular derivative terms at the diagonal and profile
joins. Hence the kernel is absolutely continuous on horizontal and
vertical segments, with the displayed derivative bounds.

To bound $R$, we separate its Hilbert coefficient evaluated at $x$. With
$H_{\T} e_k=-\ii\operatorname{sgn}(k)e_k$, this gives
\begin{equation}
 R=-\tfrac12 M_{\Phi(F')}H_{\T}+E_R.
 \label{cert:R-Hilbert}
\end{equation}
The local remainder is bounded by $M_2/(4\pi)$, since
$|Q_r-F'(x)|\le M_2|r|/2$.  The change from the remote coefficient
$\Phi(m)$ to the coefficient evaluated at $x$ costs at most $1/\pi$ by the
paired sum; the remaining remote terms cost at most $a\pi/2$.
Schur's test therefore proves
\begin{equation}
 \|R\|\le\frac12+\frac{M_2}{4\pi}+\frac1\pi+\frac{a\pi}{2}
 <\frac12+\frac{35}{6}+\frac13+\frac{33}{14}<10.
 \label{cert:R-bound}
\end{equation}
The estimate applies in particular to $R1$ and includes every remote image.

\subsection{An inverse bound from the squared operator}
\label{cert:squared-inverse}

We obtain the inverse bound by estimating the square of the operator.
Let $n=4096$ and $x_i=(i+1/2)/n$, $0\le i<n$. Let $B$ be the exact
step-kernel operator obtained by sampling $K_P$ at these midpoints.
Its matrix on the space of functions constant on each grid interval,
in the orthonormal basis $\sqrt n\,1_{[j/n,(j+1)/n)}$, is
$K_P(x_i,x_j)/n$. We identify matrices in this basis with operators on
$L^2(\T)$ by extending them by zero on the orthogonal complement of
that space. Their Frobenius norms equal the Hilbert--Schmidt norms of
the corresponding operators.
Integrating the square of
$500|u|+900|v|$ over the midpoint offsets gives
\begin{equation}
 \|P-B\|_{\HS}^2\le\frac1{n^2}
 \left(\frac{500^2+900^2}{12}+\frac{500\cdot900}{8}\right)
 <\left(\frac{93}{1000}\right)^2.
 \label{cert:midpoint-error}
\end{equation}
The cross coefficient $1/8$ comes from integrating $2|u||v|$ on the
normalised square.

The arithmetic calculation below proves
\begin{equation}
 \|B\|_{\HS}<\frac{61}{100},\qquad
 \|B^2\|_{\HS}<\frac{95}{1000}.
 \label{cert:step-norms}
\end{equation}
Consequently
\begin{align}
 \|P\|&<\frac{703}{1000},\notag\\
 \|P^2\|&\le\|B^2\|+\|P-B\|(2\|B\|+\|P-B\|)
 <\frac{217109}{10^6}<\frac14.
 \label{cert:P-square}
\end{align}
The Neumann series for $P^{*2}+1/4$ therefore converges in the continuous
operator norm. Factorisation gives the full-space inverse and
\begin{equation}
 \|(P^*+\ii/2)^{-1}\|
 \le\frac{703/1000+1/2}{1/4-217109/10^6}
 =\frac{1203000}{32891}<40.
 \label{cert:inverse-40}
\end{equation}
Since $W(P^*+\ii/2)u=(\ii/2)Wu$, a mean-zero right-hand side has a
mean-zero solution. Restriction therefore gives the inverse in
\eqref{cert:quotient-definition}.

\subsection{Interval evaluation of the sampled operator}
\label{cert:arithmetic-model}

We establish \eqref{cert:step-norms} by interval arithmetic using Arb
within FLINT \cite{Johansson2017,FLINT2026}. Each real quantity is enclosed by
a ball $[a-\rho,a+\rho]$, and each complex quantity by a product of
two real balls. The arithmetic operations enclose their exact values,
including the errors from conversion and rounding. The calculation
uses 96-bit working precision through \texttt{python-flint}~0.9.0
with FLINT~3.6.0.

The midpoint coordinates and the values of $F,F',F''$ are
first formed as exact rational numbers using the polynomial pieces in
\eqref{cert:profile-pieces}.

For $q>1$, the alternating sums
\[
 a_{q,N}=\sum_{k=0}^{N-1}\frac{(-1)^k}{(2k+1)q^{2k+1}}
\]
enclose $\arctan(1/q)$ between $a_{q,N}$ and $a_{q,N+1}$.
Machin's identity
$\pi=16\arctan(1/5)-4\arctan(1/239)$, with $N=32,10$
respectively, therefore gives rational endpoints $\pi_-<\pi<\pi_+$.
Their outward conversion gives the ball used for $\pi$.

Let $c=1+\ii/2$. For each pair of distinct grid points, choose the
integer $j$ such that $r=x-y-j\in[-1/2,1/2)$ and define
$\Delta=F(x)-F(y)-j/2$. The exact periodised bare Cauchy kernel is
\begin{equation}
 Z(x,y)=c^{-1}\cot\zeta
 =\frac1{\pi(r+\ii\Delta)}
     +\frac{\zeta T_0(\zeta^2)}{cS_0(\zeta^2)},
 \qquad \zeta=\frac{\pi(r+\ii\Delta)}c,
 \label{cert:cot-decomposition}
\end{equation}
where
\begin{equation}
 S_0(v)=\sum_{k\ge0}\frac{(-1)^kv^k}{(2k+1)!},\qquad
 T_0(v)=\sum_{k\ge1}\frac{(-1)^k2k\,v^{k-1}}{(2k+1)!}.
 \label{cert:entire-series}
\end{equation}
Indeed, $S_0(\zeta^2)=\sin\zeta/\zeta$ and
$\cot\zeta-1/\zeta=\zeta T_0(\zeta^2)/S_0(\zeta^2)$.
The corresponding real kernel is
\begin{equation}
 K_P(x,y)=\frac{\Delta-rF'(y)}{2\pi(r^2+\Delta^2)}
 -\frac12\operatorname{Im}\left[
   \frac{\zeta T_0(\zeta^2)}{cS_0(\zeta^2)}(1+\ii F'(y))\right].
 \label{cert:interval-kernel}
\end{equation}
On the diagonal the regular term vanishes. We use the removable value
$F''(x)/[4\pi(1+F'(x)^2)]$ directly.

The polynomial evaluations retain $k\le24$ in
\eqref{cert:entire-series}. At every off-diagonal sample, interval
arithmetic verifies $|\zeta|<6$. On this disk the omitted tails satisfy
\begin{align}
 e_S&\le\frac{6^{50}}{51!}
          \left(1-\frac{36}{52\cdot53}\right)^{-1}<2^{-85},\notag\\
 e_T&\le\frac{50\,6^{48}}{51!}
          \left(1-\frac{36}{50\cdot53}\right)^{-1}<2^{-85}.
 \label{cert:interval-series-tails}
\end{align}
These are exact rational comparisons. The second ratio includes the
factor $(k+1)/k$ in $T_0$. Hence, adding $[-2^{-85},2^{-85}]$ to each real
and imaginary component of each polynomial enclosure means that we
enclose the entire function. Before division, the program verifies
that the enclosure of $S_0$ excludes zero and that the real local
denominator is positive. Consequently \eqref{cert:interval-kernel},
followed by division by $n$, encloses every entry of the exact matrix
$B$.

The interval matrix is multiplied in blocks of 128 rows: each row
block of $B$ is multiplied by the full matrix $B$. These products
enclose the corresponding rows of $B^2$. Repeated occurrences of the
same entry can enlarge the intervals but preserve inclusion of the
exact product. Squaring and summing the real entries gives interval
enclosures for $\|B\|_{\HS}^2$ and $\|B^2\|_{\HS}^2$. The 32
blocks cover all rows exactly once. Each subtotal has outward dyadic
endpoints, which are then summed by exact rational arithmetic.

The resulting enclosures imply
\begin{equation}
\begin{aligned}
 0.36879506825017661274
   &<\|B\|_{\HS}^2<0.36879506825017661275,\\
 0.00878723246832160633
   &<\|B^2\|_{\HS}^2<0.00878723246832160634.
\end{aligned}
\label{cert:actual-squared-sums}
\end{equation}
All decimal endpoints in \eqref{cert:actual-squared-sums} are exact
rationals chosen outwards from the dyadic enclosures. Their upper
endpoints are strictly below $(61/100)^2$ and $(95/1000)^2$,
respectively. This proves \eqref{cert:step-norms}, and hence
\eqref{cert:inverse-40}.

\subsection{Exact treatment of the infinite periodic tail}
\label{cert:tail}

The remaining operator calculations concern four residuals. The two
adjoint residuals are
\begin{equation}
 P^*f_0+\ii\widetilde\alpha/2,\qquad
 P^*f_1+\ii\widetilde\beta/2-(F'-1/2)
 \label{cert:Pstar-residual-definitions}
\end{equation}
and the two residuals for the auxiliary singular integral are
\begin{equation}
 \widetilde\gamma-R(1+\widetilde\alpha),\qquad
 \widetilde\delta-1-R\widetilde\beta.
 \label{cert:R-residual-definitions}
\end{equation}
The inverse bound converts the first two into errors in
$\widetilde\alpha,\widetilde\beta$. Together with \eqref{cert:R-bound},
the last two then control the second row of the pair. We estimate the
four actions by separating finitely many line images from their
infinite tail.

Throughout these calculations we use \emph{unwrapped} $x,y\in[0,1]$
and retain the line images $-8\le j\le8$. The complementary images
form the tail. Define
\[
 q(x)=\frac{x+\ii F(x)}{1+\ii/2},\qquad
 \eta=q(y)-q(x),\qquad S_s(8)=\sum_{j=9}^\infty j^{-s}.
\]
The geometry gives $|\eta|<5/2$.  With the convention
$Z(x,y)=c^{-1}\cot(\pi(q(x)-q(y)))$, the paired remote Cauchy sum is
\[
 Z_{\rm tail}(x,y)=\frac{2}{\pi c}
          \sum_{k\ge0}\eta^{2k+1}S_{2k+2}(8).
\]
For a complex input $f$ define the exact moments
\begin{equation}
 \mu_l(f)=\int_\T q(y)^l f(y)\dd y,\qquad
 \nu_l(f)=\int_\T\overline{q(y)}^{l}f(y)\dd y,
 \quad 0\le l\le17.
 \label{cert:tail-moments}
\end{equation}
The second family conjugates only the coordinate function, leaving
the density unchanged. Define
\begin{align*}
 W_f(x)&=\sum_{k=0}^8S_{2k+2}(8)
   \sum_{l=0}^{2k+1}\binom{2k+1}{l}(-q(x))^{2k+1-l}\mu_l(f),\\
 V_f(x)&=\sum_{k=0}^8S_{2k+2}(8)
   \sum_{l=0}^{2k+1}\binom{2k+1}{l}
       (-\overline{q(x)})^{2k+1-l}\nu_l(f).
\end{align*}
The retained polynomial tail actions are
\begin{equation}
 P^*_{\rm tail}f=\frac{q'W_f-\overline{q'}V_f}{2\pi\ii},
 \qquad
 R_{\rm tail}f=-\frac{W_f/c+V_f/\overline c}{2\pi}.
 \label{cert:tail-actions}
\end{equation}
The signs follow from the Cauchy sum, or equivalently from
$K_P(y,x)=\tfrac12\operatorname{Im}[Z(x,y)(1+\ii F'(x))]$ and
$K_R(x,y)=-\tfrac12\operatorname{Re}Z(x,y)$. These are real kernels acting
complex-linearly on the density, as in \eqref{cert:tail-actions}.

Every moment in \eqref{cert:tail-moments} is computed by exact polynomial
multiplication and integration.  In the local coordinate
$t=128y-2l-1$, $q$ has degree eight and $f$ has degree nineteen.
Consequently degree 155 suffices for their products.  If $d_q,d_f$ are
common coefficient denominators, the moment of power $a$ can be formed
with denominator
\[
 64d_f\,\operatorname{lcm}(1,\ldots,156)\,d_q^a.
\]
For an even local power $t^r$, its integral on one integration panel is
$1/[64(r+1)]$; odd powers vanish.  These factors determine the
integer moment calculation; the derivative $q'$ is taken in physical
$x$, not in the normalised interval coordinate.

To enclose the scalar tails $S_s(8)$, we use a rational
Euler--Maclaurin formula. Define the Bernoulli numbers by
$B_0=1$ and
$B_k=-(k+1)^{-1}\sum_{j<k}\binom{k+1}{j}B_j$.  With $N=64$, $p=16$,
let
\begin{align}
 C_s={}&\sum_{j=9}^{63}j^{-s}+\frac{64^{1-s}}{s-1}
       +\frac{64^{-s}}2
       +\sum_{r=1}^{16}\frac{B_{2r}(s)_{2r-1}}
                    {(2r)!\,64^{s+2r-1}},\label{cert:EM-center}\\
 E_s={}&\left(\sum_{j=0}^{32}\binom{32}{j}|B_j|\right)
       \frac{(s)_{31}}{32!\,64^{s+31}}.
 \label{cert:EM-error}
\end{align}
Here $(s)_k=s(s+1)\cdots(s+k-1)$.  The coefficient sum bounds the
periodic Bernoulli polynomial on $[0,1]$, and integration of the
absolute $32$nd derivative of $x^{-s}$ gives
\begin{equation}
 C_s-E_s\le S_s(8)\le C_s+E_s.
 \label{cert:EM-interval}
\end{equation}
For $s=2,4,\ldots,18$, we round the interval midpoint to a multiple
of $2^{-100}$ and bound the resulting error by its maximum distance
from the two rational endpoints.

The omitted odd powers begin at nineteen.  Using $|q'|<4$ and
$\pi>3$, their kernel error for either operator is bounded by
\begin{equation}
 e_{\rm geom}=\frac{4}{3}\,
        \frac{(5/16)^{19}}{19(1-25/324)}.
 \label{cert:tail-geometric-error}
\end{equation}
Indeed the denominator of the geometric expansion is bounded using
$j\ge9$, and
$\sum_{j>8}j^{-20}\le\int_8^\infty t^{-20}\dd t$.
If $e_{2k+2}$ is the interval/rounding error in its scalar coefficient,
the additional kernel error is at most
\begin{equation}
 e_{\rm coeff}=\frac43\sum_{k=0}^8(5/2)^{2k+1}e_{2k+2}.
 \label{cert:tail-coefficient-error}
\end{equation}
By \eqref{cert:input-L2}, multiplication of $e_{\rm geom}+e_{\rm coeff}$
by four gives the corresponding action error.

The exact binomial calculation initially gives polynomials in $x$ of
degree at most 143.  To store them, enclose $1/(2\pi)$ by the Machin
interval with 64 and 24 terms, and use a dyadic representative with
denominator $2^{120}$.  A highest-degree coefficient is discarded only
when the cumulative sum of discarded absolute real and imaginary
coefficients, divided by six times the cleared pre-$\pi$ denominator,
stays below $10^{-13}$.  Round the retained real and imaginary
coefficients to multiples of $2^{-80}$.  Charge the full discarded sum,
the $\pi$ interval radius times the full pre-$\pi$ coefficient norm,
and $(d+1)2^{-80}$ for a stored degree $d$.  The exact rational
comparisons give a stored degree at most twelve and, including
\eqref{cert:tail-geometric-error}--\eqref{cert:tail-coefficient-error},
\begin{equation}
 e_{\rm tail}<7.688\,10^{-11}
 \label{cert:tail-final-error}
\end{equation}
uniformly in $x$, for each of the four actions on $f_0,f_1$.
This uniform action bound includes the input norm and is used directly
in the residual estimates.

\subsection{Finite-image polynomial enclosures}
\label{cert:box-enclosure}

To enclose the retained images, we refine near the four transition
joins. Let $\mathcal C=\{1/16,11/16,13/16,15/16\}$. The common partition for $x$ and $y$ has endpoints
\begin{equation}
 \{l/64:0\le l\le64\}
 \ \cup\ \{c\pm2^{-k}:c\in\mathcal C, 7\le k\le14\}.
 \label{cert:box-grid}
\end{equation}
The partition has 128 intervals, each contained in an original
polynomial panel. We treat their Cartesian square for each of the
seventeen retained images, giving
\begin{equation}
 128^2\cdot17=278528
 \label{cert:box-count}
\end{equation}
image boxes for each operator calculation.

For the adjoint image use
\begin{equation}
 r=y-x-j,\quad \Delta=F(y)-F(x)-j/2,\quad s=F'(x),\qquad
 K_{P^*,j}=\frac{\Delta-sr}{2\pi(r^2+\Delta^2)}.
 \label{cert:Pstar-image}
\end{equation}
If the two lifted polynomial pieces match, divide twice by the exact
linear polynomial $r$:
\[
 Q=\Delta/r,\qquad B=(Q-s)/r,
 \qquad K_{P^*,j}=B/[2\pi(1+Q^2)].
\]
We verify that both rational remainders vanish identically. The same
calculation treats the translated polynomial at the periodic seam.
When there is horizontal separation $\rho>0$, the unfactored denominator
instead has the lower bound $\rho^2$.

The only remaining boxes touching a transition join from opposite sides have
side $\eta_g=2^{-14}$.  There $|F''|\le M_5\eta_g^3/6$, by the common
affine fourth jet.  Therefore $|B|\le M_5\eta_g^3/12$ and the local
kernel is bounded by $M_5\eta_g^3/72$.  Replacing that local image by
zero gives operator norm at most $M_5\eta_g^4/72$: the four corner
supports are disjoint in both variables.  The computation charges
the slightly larger bound
\begin{equation}
 e_{P,\rm corner}=\eta_g
       \left(M_5\eta_g^3/72+M_5\eta_g^4/4\right).
 \label{cert:Pstar-corner}
\end{equation}
We make this replacement only on these four corner supports; all other
images and the periodic seam use their polynomial enclosures.

\begin{lemma}[Exact polynomial-division enclosure]
\label{cert:NDS-lemma}
On a normalised box $[-1,1]^2$, let $N,D$ be real rational polynomials
and let $n,d$ be the corresponding integer polynomials after clearing
one common denominator.  Suppose $d\ge d_*>0$ on the box.  For any
integer polynomial $s$, let $h=2^{80}n-ds$. Then
\begin{equation}
 \left|\frac{N}{2\pi D}-\frac{s}{2\pi\,2^{80}}\right|
 \le\frac{\sum_{a,b}|h_{ab}|}{6\,2^{80}d_*}.
 \label{cert:NDS-bound}
\end{equation}
\end{lemma}
\begin{proof}
Multiply the difference by $2\pi\,2^{80}d$, use $2\pi>6$, and bound
the polynomial $h$ by the sum of its absolute coefficients on the box.
\end{proof}

We take the larger of the geometric lower bound, scaled to the integer
denominator, and the coefficient bound
$d_{00}-\sum_{(a,b)\ne(0,0)}|d_{ab}|$.  For a factored adjoint denominator the
geometric bound is one; for a separated image it is $\rho^2$ as above.

We construct the candidate $s$ by a Taylor recursion in total degree:
\[
 t_{ab}=\frac{n_{ab}-\sum_{(i,j)\ne(0,0)}d_{ij}t_{a-i,b-j}}
                   {d_{00}},\qquad
 s_{ab}=\operatorname{round}(2^{80}t_{ab}),
\]
Here negative subscripts give zero. We can safely use floating point arithmetic to
propose the coefficients here, since acceptance depends on computing the
complete integer residual $2^{80}n-ds$ exactly, including all degrees
beyond the candidate polynomial. The tested degrees are
$2,4,6,8,12,16,24,32,40,48$.  Acceptance requires the exact rational
right-hand side of \eqref{cert:NDS-bound} to be at most $10^{-9}$.
Every box in \eqref{cert:box-count} passed without further
subdivision.

The input density is restricted by exact affine substitution.  If an
integration interval has length $w$ and normalised coordinate $t$, then
\begin{equation}
 \int t^j\sum_k a_k t^k\dd y
       =w\sum_{j+k\ {\rm even}}\frac{a_k}{j+k+1}.
 \label{cert:source-moment}
\end{equation}
The Jacobian factor $w/2$ combines with the factor two from integration
of an even power on $[-1,1]$ to give the factor $w$.  The common integer moment denominator
uses $\operatorname{lcm}(1,\ldots,68)$, the input denominator, and the
dyadic interval length.  This covers kernel degree at most 48 and
input degree at most 19.  Contributions from intervals of different widths
are aligned by exact integer shifts.  Thus \eqref{cert:source-moment} evaluates every integral in $y$
exactly.

After accumulation, multiply by the dyadic representative of
$1/(2\pi)$ and charge its interval radius times the full coefficient
norm of the accumulated polynomial.  This action error is less than $10^{-15}$ on every interval in $x$.  For fixed $x$, the integration intervals have total length
one.  Thus the seventeen uniform image errors contribute at most
\begin{equation}
 4\cdot17\cdot10^{-9}=6.8\,10^{-8}
 \label{cert:finite-action-error}
\end{equation}
to either input, independently of the number of rectangles.

\subsection{The two adjoint residuals}
\label{cert:pstar-residuals}

The finite-image and tail enclosures now give the two adjoint actions.
Add the tail polynomial to each finite-image action and form the
polynomial approximations to the residuals in
\eqref{cert:Pstar-residual-definitions} before taking norms.  For a complex polynomial
$p(t)=\sum c_k t^k$ on an interval in $x$ of length $w$,
\begin{equation}
 \int |p|^2\dd x
 =w\sum_{i+j\ {\rm even}}
       \frac{\operatorname{Re}(c_i\overline{c_j})}{i+j+1}.
 \label{cert:exact-poly-norm}
\end{equation}
After clearing the coefficient and integration denominators, this is
an integer calculation applied directly to the residual polynomial.
We also check that the completed intervals in $x$ cover $[0,1]$
without gaps or overlaps.

The exact resulting squared norms $s_0,s_1$ satisfy
\begin{equation}
 s_0<2.56\,10^{-20},\qquad s_1<5.52\,10^{-20},
 \qquad s_0,s_1<(3\,10^{-10})^2.
 \label{cert:Pstar-poly-comparison}
\end{equation}
All displayed terminating decimals in inequalities denote rationals.
The remaining analytic action error is
\begin{equation}
 e_P=4\cdot17\cdot10^{-9}+4e_{P,\rm corner}
                      +e_{\rm tail}+10^{-15}
       <6.812\,10^{-8}.
 \label{cert:Pstar-total-error}
\end{equation}
The triangle inequality and \eqref{cert:Pstar-poly-comparison} prove
\begin{equation}
 \|P^*(1+\widetilde\alpha)+\ii\widetilde\alpha/2\|_2<7\,10^{-8},
 \qquad
 \|P^*\widetilde\beta+\ii\widetilde\beta/2-b\|_2<7\,10^{-8}.
 \label{cert:Pstar-final}
\end{equation}

\subsection{The principal value for the auxiliary singular integral}
\label{cert:companion-pv}

It remains to estimate the two actions of $R$ in
\eqref{cert:R-residual-definitions}. We use the same unwrapped image
set and separate the singularity on three neighbouring panels. Fix an
original panel $I_l=[a,b]$ of length $h_0=1/64$. For every $x$
in this panel, select the three lifted integration panels
$I_{l-1},I_l,I_{l+1}$, whose union is $[a-h_0,b+h_0]$.
At the period seam a neighbour comes from image $j=1$ or $j=-1$.
This selection does not change when the panel containing $x$ is subsequently
refined.  Every other lifted integration panel has horizontal separation at
least $h_0$.

With $r=v-x$, $\Delta=F(v)-F(x)$, and $s=F'(x)$, split only these
three near panels as
\begin{equation}
 K_{R,j}(x,v)=\frac{r}{2\pi(r^2+\Delta^2)}
       =\frac{a(x)}{v-x}+E_j(x,v),
 \qquad a(x)=\frac1{2\pi(1+s^2)}.
 \label{cert:companion-split}
\end{equation}
On matching polynomial pieces, exact divided differences give
\begin{equation}
 E_j=-\frac{B(Q+s)}{2\pi(1+Q^2)(1+s^2)},\qquad
 Q=\Delta/r,\quad B=(Q-s)/r.
 \label{cert:R-factored}
\end{equation}
On separated near pieces the equivalent expression is
\begin{equation}
 E_j=-\frac{(\Delta-sr)(\Delta+sr)}
                {2\pi r(r^2+\Delta^2)(1+s^2)}.
 \label{cert:R-separated}
\end{equation}
Multiply its numerator and denominator by the constant sign of $r$
on the box; the denominator is then positive and at least the cube
of the horizontal separation.  Formula \eqref{cert:R-factored} has
denominator at least one before the factor $2\pi$.  For all other
images use the original nonsingular $K_{R,j}$, with denominator lower
bound equal to squared separation.  Lemma~\ref{cert:NDS-lemma} therefore encloses all three rational forms.

Since
$|(Q+s)/((1+Q^2)(1+s^2))|\le1$, the common affine-jet argument at
a transition corner permits replacement of $E_j$ by zero there,
with operator error
\begin{equation}
 e_{R,\rm corner}=M_5\eta_g^4/72
                =35/3298534883328.
 \label{cert:R-corner}
\end{equation}
We retain the Hilbert coefficient $a(x)$ and integrate the smooth
image actions exactly by \eqref{cert:source-moment}. Adding the stored
$R$ tail once gives the complete error
\begin{equation}
 e_{R,\rm sm}=4\cdot17\cdot10^{-9}+4e_{R,\rm corner}
                         +e_{\rm tail}+10^{-15}
       <6.812\,10^{-8}.
 \label{cert:R-smooth-error}
\end{equation}

It remains to integrate the removed three-panel principal value.
For a polynomial $p$ on $[L,U]$, polynomial division gives the exact
identity
\begin{equation}
 \pv\int_L^U\frac{p(v)}{v-x}\dd v
  =\int_L^U\frac{p(v)-p(x)}{v-x}\dd v
     +p(x)\big(\log|U-x|-\log|L-x|\big).
 \label{cert:PV-polynomial}
\end{equation}
The first integrand is a polynomial.  In the panel coordinates
$t=128v-2k-1$, $z=128x-2k-1$, it is evaluated as
\begin{equation}
 \sum_{n\ge1}p_n\sum_{j=0}^{n-1}z^j
 \begin{cases}2/(n-j),&n-1-j\text{ even},\\0,&n-1-j\text{ odd}.
 \end{cases}
 \label{cert:PV-quotient-polynomial}
\end{equation}
The factors from the coordinate changes in the measure and denominator cancel.

Sum the three instances of \eqref{cert:PV-polynomial} before bounding
their logarithms.  At an internal seam $c=a$ or $c=b$, the coefficient
of $\log|c-x|$ is
\begin{equation}
 C_c(x)=p_{\rm left}(x)-p_{\rm right}(x).
 \label{cert:seam-coefficient}
\end{equation}
The two outer coefficients are $-p_{l-1}(x)$ and $p_{l+1}(x)$; their
distances from $x$ are at least $h_0$.  All four logarithmic
coefficients sum to zero identically.  Thus conversion between physical and normalised distances preserves
the constant-input term.  For constant density the formula is exactly
$\log((b+h_0-x)/(x-a+h_0))$.

The functions \eqref{cert:data-definition} are continuous piecewise
polynomials and hence globally Lipschitz. The three-panel principal-value
integrals therefore exist, and continuity gives $C_c(c)=0$ at every seam. To estimate the
seam contribution, we reconstruct
all nineteen higher Taylor coefficients exactly:
\[
 C_c(c+t)=\sum_{k=1}^{19}d_{c,k}t^k,\qquad
 \omega_c=\sum_{k=1}^{19}
      (|\operatorname{Re}d_{c,k}|+|\operatorname{Im}d_{c,k}|)\eta_s^k,
 \qquad \eta_s=1/4096.
\]
For $|t|\le\eta_s$, monotonicity of
$t^k\log(1/t)$ and $\log(1/\eta_s)<12$ imply
\begin{equation}
 |C_c(c+t)\log|t||\le12\omega_c.
 \label{cert:seam-pointwise}
\end{equation}
Only this combined seam logarithm is omitted in its window.
The windows are disjoint and have total measure $128\eta_s$ on the
circle, including the two halves at the periodic endpoint.  Since
$a(x)<1/6$, the squared action error is at most
\begin{equation}
 128\eta_s\left(\frac{12\max_c\omega_c}{6}\right)^2.
 \label{cert:seam-L2}
\end{equation}
Exact calculation for $f_0,f_1$, respectively, gives the upper bounds
\begin{equation}
 \max_c\omega_c<8.242\,10^{-9},\quad 1.8457\,10^{-8},
 \qquad
 e_{{\rm seam},0}^2<8.50\,10^{-18},\quad
 e_{{\rm seam},1}^2<4.26\,10^{-17}.
 \label{cert:seam-comparisons}
\end{equation}
We therefore allow $10^{-8}$ in $L^2$ for each input.

\subsection{Rational logarithms and the final residuals}
\label{cert:companion-residuals}

We enclose the remaining logarithms by rational series. For rational $r>0$,
exact dyadic scaling gives $r=2^e u$, $1\le u<2$. With
$t=(u-1)/(u+1)$,
\begin{equation}
 2\sum_{k=0}^{47}\frac{t^{2k+1}}{2k+1}
 \le\log u\le
 2\sum_{k=0}^{47}\frac{t^{2k+1}}{2k+1}
       +\frac{2t^{97}}{97(1-t^2)}.
 \label{cert:atanh-log}
\end{equation}
The same formula at $t=1/3$ encloses $\log2$.  Combining the endpoints
with the sign of $e$ gives a rational enclosure of $\log r$.

To approximate these logarithms uniformly, we refine the intervals
in $x$ after accumulating the smooth actions. Add endpoints at distances $2^{-7},\ldots,2^{-12}$ from every original
panel endpoint.  Together with \eqref{cert:box-grid}, this gives 784
intervals.  The three selected integration panels remain fixed.  On each
interval every retained logarithm has distance
$d_0(1+\theta t)$ with $|t|\le1$, $|\theta|\le1/2$.  Use
\begin{equation}
 \log d_0+\sum_{k=1}^q\frac{(-1)^{k+1}\theta^kt^k}{k},
 \qquad
 |\mathrm{remainder}|\le
       \frac{|\theta|^{q+1}}{(q+1)(1-|\theta|)}.
 \label{cert:log-Taylor}
\end{equation}
The constant is enclosed by \eqref{cert:atanh-log} and rounded to
denominator $2^{120}$.  Choose $q$ so that the full endpoint-polynomial
coefficient norm times its log error is at most $10^{-9}/16$.
After all terms are summed, round each component of the complete
three-panel polynomial to denominator $2^{100}$, charging the sum of
all coefficient errors, including coefficients that round to zero.

Approximate $a(x)=1/[2\pi(1+F'(x)^2)]$ on the same intervals by the
one-variable instance of Lemma~\ref{cert:NDS-lemma}, with geometric
denominator lower bound one.  Enclose its factor $1/(2\pi)$ and the
coefficient rounding in the same manner.  The completed rational
comparisons give
\begin{equation}
 e_H<2\,10^{-10},\qquad e_a<2.25\,10^{-11},
 \label{cert:H-a-errors}
\end{equation}
uniformly on every interval in $x$ and for both inputs.  The largest
accepted logarithm degree is nineteen and the largest degree for $a$
is twelve.

To propagate the error in the scalar coefficient, we need a uniform
bound for the principal value. Denote the unmultiplied three-panel
integral by $H_{\rm near}f$. Subtracting $f(x)$ in the integrand gives
\begin{equation}
 |H_{\rm near}f|\le\frac3{64}\|f'\|_\infty
                           +\|f\|_\infty\log2.
 \label{cert:H-near-bound}
\end{equation}
Exact power-coefficient bounds on the original panels give
\[
 \|f_0\|_\infty<15/4,\quad\|f_0'\|_\infty<62,
 \qquad \|f_1\|_\infty<33/4,\quad\|f_1'\|_\infty<214.
\]
These estimates follow from sums of absolute coefficients on each panel.  Hence \eqref{cert:H-near-bound} is less than nineteen for
both inputs.  At most one seam term is omitted on an interval, its
unmultiplied size is less than $10^{-6}$ by
\eqref{cert:seam-pointwise}--\eqref{cert:seam-comparisons}, and the
polynomial error is less than $10^{-9}$.  The computed near-Hilbert
polynomial is therefore bounded in modulus by twenty.  The $a$
approximation costs at most $20e_a$, and the $H$ approximation at
most $e_H/6$.

Restrict the smooth action and tail in $x$ exactly to these final
intervals. Add the product of the $a$ and near-Hilbert polynomials,
and form polynomial approximations to the two residuals in
\eqref{cert:R-residual-definitions}.
After forming the residuals, round the coefficients to denominator
$2^{100}$. Their full coefficient-error sum is less than $10^{-20}$ on
each interval and is included in the total error.  Applying
\eqref{cert:exact-poly-norm} gives exact squared polynomial norms
\begin{equation}
 t_0<6\,10^{-20},\qquad t_1<1.52\,10^{-19},
 \qquad t_0,t_1<(4\,10^{-10})^2.
 \label{cert:R-poly-comparison}
\end{equation}
The total analytic error is at most
\begin{equation}
 e_R=e_{R,\rm sm}+10^{-8}+e_H/6+20e_a+10^{-20}
       <7.861\,10^{-8}.
 \label{cert:R-total-error}
\end{equation}
Combining \eqref{cert:R-poly-comparison} and \eqref{cert:R-total-error}
proves
\begin{equation}
 \|\widetilde\gamma-R(1+\widetilde\alpha)\|_2<8\,10^{-8},\qquad
 \|\widetilde\delta-1-R\widetilde\beta\|_2<8\,10^{-8}.
 \label{cert:R-final}
\end{equation}
Together with \eqref{cert:Pstar-final}, this establishes the four
residual bounds used in Subsection~\ref{cert:actual-margin}.

\subsection{Exact moments and the positive matrix test}
\label{cert:moments}

With the residual bounds established, we now verify the positive matrix
inequality for the approximate pair by computing its Gram matrix.

Define
\[
 \widetilde M=\begin{pmatrix}v_1&v_2\\v_3&v_4\end{pmatrix}
 =\begin{pmatrix}1+\widetilde\alpha&\widetilde\beta\\
                  \widetilde\gamma&\widetilde\delta\end{pmatrix},
 \qquad g_{ab}=\int_J\overline{v_a(x)}v_b(x)\dd x.
\]
The interval $J$ is the union of original panels 5 through 42.  On each such
panel, let $t_l=128x-2l-1$. Then
\begin{equation}
 \int_{I_l}T_j(t_l)T_k(t_l)\dd x
 =\begin{cases}
 \displaystyle\frac1{128}\left(
 \frac1{1-(j+k)^2}+\frac1{1-(j-k)^2}\right),&j+k\text{ even},\\
 0,&j+k\text{ odd}.
 \end{cases}
 \label{cert:Chebyshev-Gram}
\end{equation}
Let $d_{\mathrm{coef}}$ be a common positive integer denominator of the real and
imaginary parts of all polynomial coefficients. Each real and
imaginary part of $g_{ab}$ is an integer divided by
\[
 128d_{\mathrm{coef}}^2\,\operatorname{lcm}\{|1-k^2|:k=0,2,\ldots,38\}.
\]
We thereby obtain the complex Hermitian Gram matrix exactly.

For a Hermitian input
$Y=\bigl(\begin{smallmatrix}a&c+\ii d\\c-\ii d&b\end{smallmatrix}\bigr)$,
the resulting map $\widetilde K(Y)=\int_J\widetilde M(x)^*Y\widetilde M(x)\dd x$
is determined by
\begin{align*}
 (\widetilde K(Y))_{11}
   &=a g_{11}+b g_{33}+2c\operatorname{Re}g_{13}
                         -2d\operatorname{Im}g_{13},\\
 (\widetilde K(Y))_{22}
   &=a g_{22}+b g_{44}+2c\operatorname{Re}g_{24}
                         -2d\operatorname{Im}g_{24},\\
 (\widetilde K(Y))_{12}
   &=a g_{12}+b g_{34}+c(g_{14}+g_{32})
                         +\ii d(g_{14}-g_{32}).
\end{align*}
Let $\mathbf K$ be the real $4\times4$ matrix representing
$\widetilde K$ in the coordinates $(a,b,c,d)$ on $\Herm_2$.
Complex conjugation is represented by
$C=\operatorname{diag}(1,1,1,-1)$. Define
$\widetilde K_-(Y)=\overline{\widetilde K(\overline Y)}$
and $\widetilde A=\widetilde K\circ\widetilde K_-$.
Their coordinate matrices are $C\mathbf K C$ and
$\mathbf K C\mathbf K C$, respectively.

The exact rational tests obtained from these formulas are
\begin{equation}
 \int_J\|\widetilde M(x)\|_{\HS}^2\dd x<4,\qquad
 (7/2)I-\widetilde K(I)>0,\qquad
 \widetilde A(X)-(21/20)X-(1/50)I>0.
 \label{cert:exact-moment-tests}
\end{equation}
We verify positive definiteness using the first leading entry and
determinant. For the middle matrix these exceed $11/4$
and $1/4$, respectively.  Those of the last matrix exceed
$119/10000$ and $1/2000$, respectively. The determinants of
$X-I/4$ and $I-X$ are $881/40000$ and $631/40000$, and their first
entries are positive, proving \eqref{cert:witness}.

\subsection{Perturbation of the pair and the positive map}
\label{cert:actual-margin}

We now transfer the matrix inequality to the true pair. It is enough
to weaken \eqref{cert:Pstar-final} to $5\,10^{-7}$ and
\eqref{cert:R-final} to $5\,10^{-6}$. The functions in
\eqref{cert:M-definition} then satisfy
\[
 \|\alpha-\widetilde\alpha\|_2,
 \|\beta-\widetilde\beta\|_2<20\,10^{-6},
\]
by \eqref{cert:inverse-40}.  Applying \eqref{cert:R-bound} to those
errors gives the two lower-row errors at most $205\,10^{-6}$ each.
Consequently
\begin{equation}
 e:=\|M-\widetilde M\|_{L^2(\T;\HS)}
 \le\sqrt{2\cdot20^2+2\cdot205^2}\,10^{-6}<1/3000.
 \label{cert:pair-error}
\end{equation}
The reflected graph has $P_-=-P$, $R_-=R$, and, at $z=-\ii/2$,
$M_-=\overline M$.  Hence the same bound applies to its pair map.

For Hermitian $Y$ of operator norm at most one, Cauchy--Schwarz and
\eqref{cert:exact-moment-tests} give
\begin{equation}
 \|K-\widetilde K\|
 \le e\bigl(2\|\widetilde M\|_{L^2(J;\HS)}+e\bigr)
 <\frac{4+1/3000}{3000}<\frac1{700}.
 \label{cert:one-step-error}
\end{equation}
A positive real-linear map on Hermitian matrices, equipped with the
operator norm, has norm $\|K(I)\|$: use $-I\le Y\le I$.
Thus $\|\widetilde K\|<7/2$ and
$\|K\|<7/2+1/700$, and similarly for the reflected maps.
It follows that
\begin{equation}
 \|A-\widetilde A\|
 <\frac{7+1/700}{700}<\frac1{99}.
 \label{cert:two-step-error}
\end{equation}
Together with $\|X\|<1$ and \eqref{cert:exact-moment-tests}, this gives
\[
 A(X)-(21/20)X>\left(\frac1{50}-\frac1{99}\right)I
                 =\frac{49}{4950}I.
\]
This proves Proposition~\ref{cert:local-lemma}.

\subsection{Certificate files and verification}
\label{cert:reproduction}

The computational files are available in the \path{certificate/}
subdirectory of the repository
\[
 \text{\certificaterepository}.
\]
All filenames below are relative to that subdirectory.
The following files contain the exact input and the computed
polynomials and scalar bounds used above. Each computed data file has
a generating script with the same basename.
\begin{description}
\item[Exact inputs]
Coefficients in \eqref{cert:data-definition}:
\path{pair_rational_approximants.json}.
\item[Coarse inverse]
\path{periodic_squared_resolvent_certificate.json} contains the
dyadic enclosures and block subtotals that imply
\eqref{cert:actual-squared-sums}.
\item[Remote images]
\path{pair_exact_periodic_tail.json} contains all stored tail polynomials.
\item[Adjoint residuals]
\path{pair_pstar_residual_certificate.json} contains all residual
polynomials in $x$.
\item[Smooth terms]
The smooth actions and their tail are in the file
\par\noindent
\path{pair_companion_smooth_certificate.json}.
\item[Principal-value assembly]
\path{pair_companion_pv_certificate.json} contains all final
residual polynomials for the auxiliary singular integral.
\item[Matrix moments]
\path{pair_exact_moments.json} contains the coordinate matrices of
$\widetilde K$ and $\widetilde A$, in the fields \texttt{K} and
\texttt{A}, together with the rational matrix comparisons.
\end{description}
The shared exact arithmetic is in \path{exact_arithmetic.py},
and \path{verify.py} performs the
checks described below. The implementation uses arbitrary-precision Python
integers and the standard \texttt{Fraction} class. The coarse matrix
calculation uses Arb; its exact profile, series coefficients and scalar
estimates are in \path{coarse_inverse_model.py}. NumPy supplies the
floating-point polynomial proposals for the finite-image calculations.
The finite-image polynomial algorithm is implemented in
\texttt{pair\_finite\_kernel\_boxes.py}. Integer/Fraction arithmetic
clears the denominators and verifies each polynomial enclosure before
integration in $y$. Before either combined certificate is formed,
the rows sorted by their intervals in $x$ are checked to cover the 128 intervals of
\eqref{cert:box-grid} without gaps or overlaps. All 278528 image boxes
passed for each smooth operator calculation; the principal-value
calculation then used the 784 intervals described above.

The consolidated verifier checks the coarse interval subtotals, their
coverage and the exact comparisons in \eqref{cert:actual-squared-sums}.
It recomputes the exact input moments, every stored squared residual
norm, the total analytic errors, and the final positive matrix
inequalities. Its \texttt{--replay} option also rebuilds
all 784 principal-value, coefficient, and final assembly rows from the
saved smooth actions, without a preparation cache, and compares the exact
coefficients. It integrates all 1568 final complex polynomials using a
separate integer Gram summation. The \texttt{--coarse-replay} option
recomputes all sampled kernel enclosures and the interval matrix square.
The tail polynomials and two-dimensional kernel enclosures are regenerated
by their respective scripts. The continuous estimates follow by combining these
finite checks with the kernel identities and the tail, corner and
perturbation bounds above.

From the repository root, \texttt{python -B certificate/verify.py} runs
the stored-data checks; adding both \texttt{--replay} and
\texttt{--coarse-replay} runs the one-dimensional reconstruction and
the full Arb calculation as well. The \path{README.md} in the certificate
directory gives the dependencies and full regeneration commands.

\small\linespread{0.9}

\begingroup
\renewcommand{\addcontentsline}[3]{}
\section*{Acknowledgements}
The authors thank Simon Chandler-Wilde for helpful feedback. We
gratefully acknowledge the David Crighton Fund for providing a fellowship
supporting SS's visit to MJC in Cambridge. SS was funded by the UK Engineering and Physical Sciences Research Council (EPSRC) through a PhD Studentship. MJC also thanks the Isaac Newton
Institute for Mathematical Sciences, Cambridge, for its hospitality and
support during the programme \emph{Geometric spectral theory and applications}.
This work was supported by EPSRC grant EP/Z000580/1.

\section*{Declaration of AI assistance}

During this work, the authors used OpenAI's ChatGPT 5.6 and ChatGPT 6 to assist in the discussion of the separation
of singular terms and bounded remainders in the continuity argument of Lemma~\ref{lem:smooth-continuity},
the perturbation arguments preserving the resolvent bounds and strict matrix inequality in Lemmas~\ref{lem:fixed-smoothing} and~\ref{lem:parameter-stability}, and the propagation of residual errors to the matrix estimates in Appendix~\ref{cert:actual-margin}. The authors directed the investigation, assessed the proposed approaches, and developed the mathematical arguments into the proofs presented here.

The authors' mathematical work included formulating the two-scale analysis and the corresponding implementation of such at finite separation of scales, as well as the derivation of the stopping criterion which ensures uniform boundedness of the graph slopes. The construction of the transition from periodic approximate eigenvectors to a single bounded boundary was also done by the authors. They developed and verified the coupled resolvent recursion for the density and its auxiliary image, and the argument converting a positive-map inequality into growth of resolvent solutions with fixed forcing. For the computer-assisted part, the authors specified the estimates required by the analytic argument and checked that the certificates establish them for the continuous operators, including the treatment of principal values, infinite periodic tails and approximation errors. ChatGPT was used to assist with debugging and final polishing of the computational code. Responsibility for the results, proofs, computations, accompanying code and final manuscript rests entirely with the authors.

\par
\endgroup

\raggedbottom
\bibliographystyle{amsplain}
\bibliography{references}

\end{document}